\documentclass[11pt,oneside]{article}
\usepackage[utf8]{inputenc}

\newcommand\blfootnote[1]{%
  \begingroup
  \renewcommand\thefootnote{}\footnote{#1}%
  \addtocounter{footnote}{-1}%
  \endgroup
}

\usepackage[english]{babel}%
\addto{\captionsenglish}{\renewcommand{\refname}{References~$^\ast$\blfootnote{$^\ast$ Numerical libraries are typeset \texttt{[lib]} and are listed at the end of the bibliography.}}}
\usepackage[a4paper]{geometry}
\usepackage{indentfirst}

\usepackage{amsmath,amssymb,amsfonts,amsthm,mathrsfs,stmaryrd}  
\usepackage{graphicx,caption,float}
\usepackage[usenames, dvipsnames]{xcolor}

\usepackage[hang,flushmargin]{footmisc} 

\usepackage{enumitem}
\usepackage{multirow}

\usepackage{longtable}
\usepackage[linesnumbered, ruled, vlined]{algorithm2e}

\usepackage[colorlinks=true, pdfstartview=FitV,linkcolor=ForestGreen,citecolor=ForestGreen, urlcolor=blue]{hyperref}

\author{
Nicolae Mihalache$^{1}$
\quad and \quad
Fran\c{c}ois Vigneron$^{2}$
}

\newcommand{\authornicu}{{\small
\textbf{Nicolae Mihalache.}
Univ Paris-Est Creteil, CNRS UMR 8050, LAMA, F-94010 Creteil, France and
Univ Gustave Eiffel, LAMA, F-77447 Marne-la-Vallée, France\\
\phantom{x}\hfill\texttt{nicolae.mihalache@u-pec.fr}
}}

\newcommand{\authorfv}{{\small
\textbf{François Vigneron.}
Universit\'{e} de Reims Champagne-Ardenne,
Laboratoire de Mathématiques de Reims,
UMR 9008 CNRS, Moulin de la Housse, BP 1039,
F-51687 Reims\\
\phantom{x}\hfill\texttt{francois.vigneron@univ-reims.fr}
}}

\newtheorem{thm}{Theorem}
\newtheorem{prop}{Proposition}

\newtheorem{lemma}[prop]{Lemma}
\newtheorem{corollary}[prop]{Corollary}
\newtheorem{remark}[prop]{Remark}

\newcommand{\N}{\mathbb{N}}
\newcommand{\Z}{\mathbb{Z}}

\newcommand{\R}{\mathbb{R}}
\newcommand{\C}{\mathbb{C}}
\newcommand{\CC}{\overline{\mathbb C}}

\DeclareMathOperator{\Hol}{Hol}
\newcommand{\U}{\mathbb{U}}

\newcommand{\M}{\mathcal{M}}

\newcommand{\julia}{\mathcal{J}}

\newcommand{\fatou}{\mathcal{F}}

\DeclareMathOperator{\attract}{\mathcal{A}}

\DeclareMathOperator{\crit}{Crit}

\DeclareMathOperator{\dist}{dist}

\newcommand\ii[2]{\ensuremath{\llbracket{#1}, {#2} \rrbracket}}

\newcommand{\fprec}{\mathcal{R}}
\newcommand{\fprecRef}{\mathcal{Z}}
\DeclareMathOperator{\round}{R}  
\DeclareMathOperator{\ulp}{ulp}
\DeclareMathOperator{\disk}{D}
\DeclareMathOperator{\interval}{I}
\renewcommand{\Re}{\operatorname{Re}}
\renewcommand{\Im}{\operatorname{Im}}
\DeclareMathOperator{\eval}{V}
\newcommand{\defequal}{=}

\newcommand{\eg}[1][~]{\textit{e.g.}#1}
\newcommand{\ie}[1][~]{\textit{i.e.}#1}

\def\cod#1{\texttt{#1}}

\DeclareMathOperator{\hyp}{Hyp}
\DeclareMathOperator{\mis}{Mis}
\renewcommand{\div}{\operatorname{Div}}
\newcommand{\citelib}[1]{\texttt{\cite{#1}}}
\newcommand{\mesh}{\mathbb{M}}
\newcommand{\lev}{\mathbb{L}}

\makeatletter 
\newcommand{\newparallel}{\mathrel{\mathpalette\new@parallel\relax}}
\newcommand{\new@parallel}[2]{%
  \begingroup
  \sbox\z@{$#1T$}
  \resizebox{!}{\ht\z@}{\raisebox{\depth}{$\m@th#1/\mkern-5mu/$}}%
  \endgroup
}
\makeatother

\begin{document}

\title{How to split a tera-polynomial}
\maketitle

\begin{abstract}  
This article presents a new algorithm to compute all the roots of two families of polynomials that are
of interest for the Mandelbrot set $\mathcal{M}$ : the roots of those polynomials are respectively
the parameters $c\in\mathcal{M}$ associated with periodic critical dynamics for $f_c(z)=z^2+c$ (hyperbolic centers)
or with pre-periodic dynamics (Misiurewicz-Thurston parameters). The algorithm is based on the
computation of discrete level lines that provide excellent starting points for the Newton method.
In practice, we observe that these polynomials can be split in linear time of the degree.

\medskip
This article is paired with a code library~\citelib{MLib} that implements this algorithm. Using this library and
about 723\,000 core-hours on the HPC center \textit{Roméo} (Reims), we have successfully found all
hyperbolic centers of period $\leq 41$ and all Misiurewicz-Thurston parameters whose period
and pre-period sum to $\leq 35$.
Concretely, this task involves splitting a tera-polynomial, \ie a polynomial of degree~$\sim10^{12}$,
which is orders of magnitude ahead of the previous state of the art. It also involves dealing
with the certifiability of our numerical results, which is an issue that we address in detail, both mathematically
and along the production chain. The certified database is available to the scientific community \citelib{MLibData}.

\medskip
For the smaller periods that can be represented using only hardware arithmetic (floating points \texttt{FP80}),
the implementation of our algorithm can split the corresponding polynomials of degree $\sim10^{9}$
in less than one day-core.
We complement these benchmarks with a statistical analysis of the separation of the roots, which
confirms that no other polynomial in these families can be split without using higher precision arithmetic.
\end{abstract}

{\small
\tableofcontents
}
\newpage


\section{Introduction}
\label{par:intro}

Algorithms to find all the roots of a given polynomial date back to the beginning of mathematics.
Greeks and Babylonians understood quadratic equations 2000~BC and applied them to
compute the boundaries of their agricultural fields, to define just taxes and fair trade.
Al-khawarizmi's algorithm (825AD) constitutes the official starting point of \textit{algebra}
and gave it its name.

\medskip
At a theoretical level, the fundamental theorem of algebra ensures that every polynomial
can be split over $\C$ and the proof of this results was successively refined
by d'Alembert (1746), C.F.~Gauss (1799), J.R.~Argand (1806) and A.L.~Cauchy (1821).
The race towards general solutions by radicals, though punctuated by great achievements,
ultimately stumped all generations until N.~Abel (1829) and E.~Galois (1832) proved the
problem to be unsolvable and discovered a new branch of mathematics (groups theory)
along the way.

\medskip
Polynomial roots are also of key importance for modern numerical analysis.
For example, Gauss quadrature formulas provide an extremely efficient way to compute the
value of integrals as a linear combination of evaluations at the roots of some orthogonal
polynomial (Legendre, Jacobi, Chebyshev, Laguerre, Hermite,\ldots); see \eg\cite{BERMAD}.
Another striking example is the computation of
eigenvalues of linear operators, which are, by definition (at least when the ambiant dimension
is finite), the roots of the characteristic polynomial. 
The knowledge of eigenvalues is crucial to understand the behavior of dynamical systems
as illustrated, dramatically, by the collapse of Tacoma's bridge in 1940. Similarly, the Covid-19 pandemic
has put $R_0$ (the largest eigenvalue of the regeneration matrix in a population model) under
the international spotlight. Eigenvalues are also essential in exploratory data analysis when
a predictive model is based on a principal component analysis.

\medskip

In this article, we focus on two important families of polynomials.
The family of \textit{hyperbolic polynomials} is defined recursively by
\begin{equation}\label{def:hyp_poly}
p_0(z) \defequal 0, \qquad p_{n+1}(z) \defequal p_n(z)^2 + z.
\end{equation}
For $n\geq 1$, one has $\deg p_n = 2^{n-1}$.
The family of \textit{Misiurewicz-Thurston polynomials} is a two parameters family defined for $\ell,n\in\N$ by
\begin{equation}\label{def:mis_poly}
q_{\ell,n}(z)=p_{\ell+n}(z) - p_\ell(z).
\end{equation}
The polynomials $p_n$ and $q_{\ell,n}$ are closely related to the Mandelbrot set $\M$
(see \cite{CARLESON}, \cite{MILNOR} or Section~\ref{par:mandel} for further details).
This set is both one of the most intricate mathematical objects and a very famous computer-generated image (Figure~\ref{fig:mandelbrot}).
Its fractal nature has fascinated the general public, the mathematical and the computer science communities
since 1980 and a great deal of research has been done around its properties during this time.
The century old conjecture of P.~Fatou~\cite{FATOU1919} remains largely open and can be stated as the following question:
are all connected components of the interior of~$\M$ hyperbolic, \ie do they each contain exactly one root of some $p_n$ ?

\begin{figure}[H]
\begin{center}
\includegraphics[width=.55\textwidth]{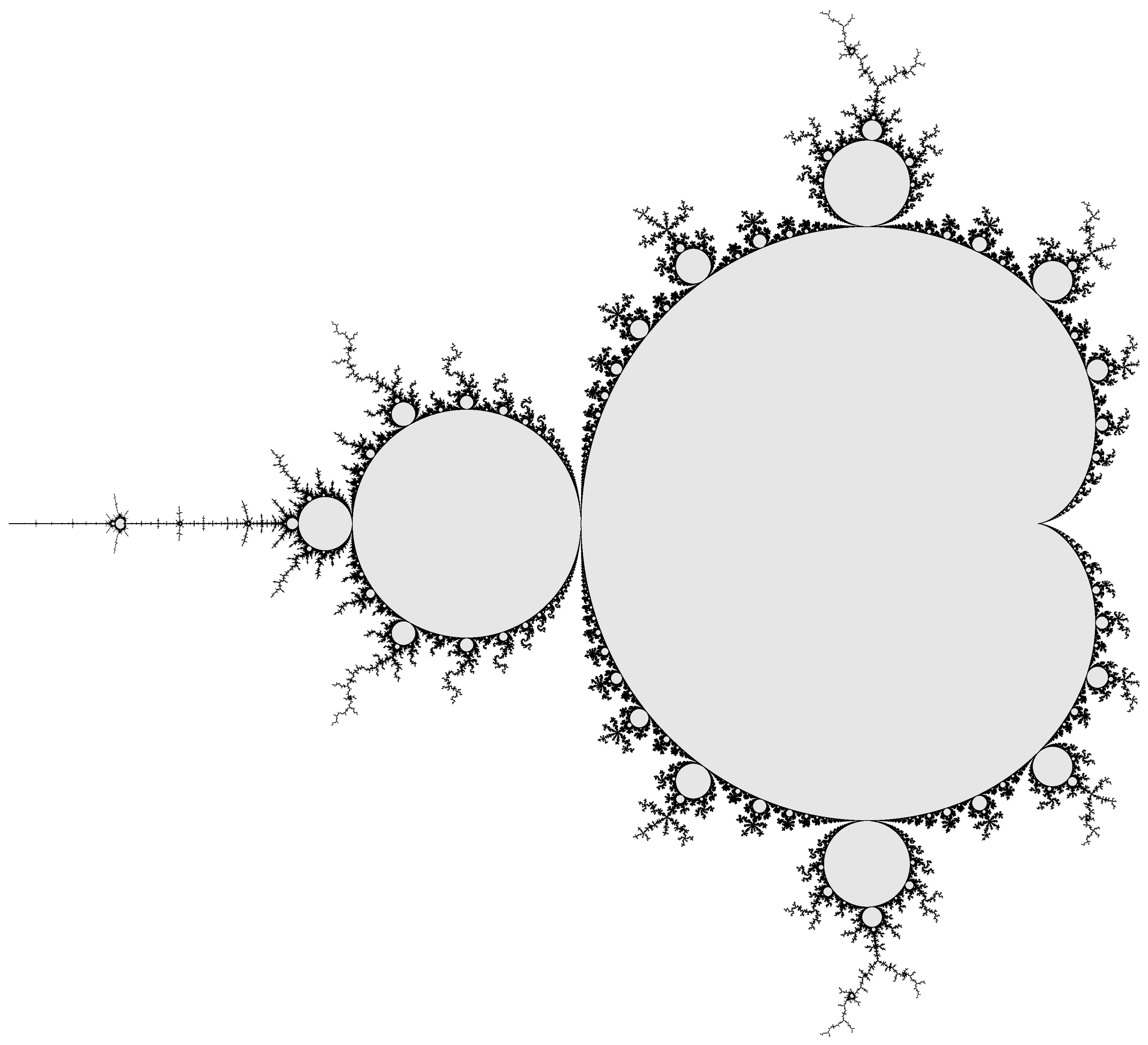}
\caption{\label{fig:mandelbrot}\small\sl
The Mandelbrot set~$\M$ with $\partial\M$ in black and in gray, the interior of $\M$.}
\end{center}
\end{figure}

The current state of the art for root-finding algorithms is limited, in practice, to polynomials
whose degree does not exceed a few millions or, at best, a few hundred millions.
In Section~\S\ref{par:roots} we survey briefly the corresponding algorithms.
For example, D. Schleicher and R. Stoll \cite{SS2017} have mastered this scale and give computing times%
\footnote{\label{timecore}A \textit{time-core} unit is a rough measure of the practical complexity of an implementation,
regardless of the architecture of the computer. It is obtained by 
adding the individual compute time of each active computing unit involved throughout the computation.
A convenient estimate and upper bound is the product of the wall-clock time of the computation by the number of CPU cores in use.
Up to a point, mostly limited by memory management,
parallelization leads to a better wall-clock time, even though the overall time-core requirements may be substantially higher.}
ranging from 18 hours-core to 15 days-core for various polynomials of degree $2^{20}\sim\nobreak10^6$.
Other implementations using a different splitting algorithm and massively parallel architectures \cite{GSCG19}
end up with similar core requirements and degree limitations.
Eigenvalue methods encounter similar limitations on the size of non-zero elements in the matrices~\cite{DT93}, \cite{IYM2011}, \cite{SZIYKH2015}, \cite{SK2019}.
A substantial leap forward was accomplished by~\cite{RSS2017}: using an adaptative mesh-refinement technique
for the Newton method they split a polynomial of degree $2^{24}$ in less than 7 days-core, however with a few missing roots,
and even one instance of a polynomial of degree $2^{30}\sim  10^{9}$ in 4 days-core.
This remarquable implementation is even backed by theoretical results~\cite{HSS2001}, \cite{Sch2023} (see \S\ref{par:HSS} below) that ensure that a certain set of starting points
guaranties that all roots will be found.

\medskip
For the application to Gaussian quadrature, one can exploit the specific connection between orthogonal
polynomials and the associated ODEs to compute the roots in linear time.
The published benchmark time \cite{Bre17} is about 6 day-core to compute a Gauss-Jacobi rule of record size $10^{12}$.
The algorithm is explained in \cite{BMF2012}, \cite{HT2013}, \cite{TTO2016}.

Let us point out that, while numerical analysis has recently reached problems
with $2\times 10^9$ degrees of freedom for hyperbolic PDEs in field-ready medical imaging~\cite{Hecht2019},
 $550\times\nobreak10^9$ degrees of freedom for elliptic PDEs on the \textit{Titan} supercomputer~\cite{GMSB2016}
and even~$10^{15}$ degrees of freedom for a turbulence simulation on \textit{Summit}~\cite{RAY2019},
those remarkable achievements do not yet imply the ability to solve routinely eigenvalue and polynomial splitting problems of this size.
In particular, as we will see in Section~\S\ref{par:num_results}, the low separation between roots of polynomials of extreme degree
cannot, in general, be resolved using hardware floating point arithmetic.

\subsection{Our results in a nutshell}\label{intro:our_results}

This article raises the state of the art of polynomial splitting to a degree about one trillion (\ie $2^{40}\sim 10^{12}$),
called a \textsl{tera-polynomial}\footnote{As a trinomial denotes a polynomial expression with 3 terms,
we should call it a \textsl{teranomial}. We decided against it as the single letter difference with the prefix \textsl{tetra} (that denotes 4)
could lead to confusion.}.
We do not claim full generality and focus on the families $p_n$ and $q_{\ell,n}$ defined above.

\medskip
To achieve this result without brute force, the main idea is a new splitting algorithm.
Roughly speaking, the classical Newton's method consists in choosing~\cite{HSS2001} a mesh of initial points
that encircles the roots and to iterate the Newton map until the trajectories reach the roots; this method results in
a dense mesh of most of the disk (see Fig.~\ref{fig:idea}).

\begin{figure}[h!]
\begin{center}
\includegraphics[width=\textwidth]{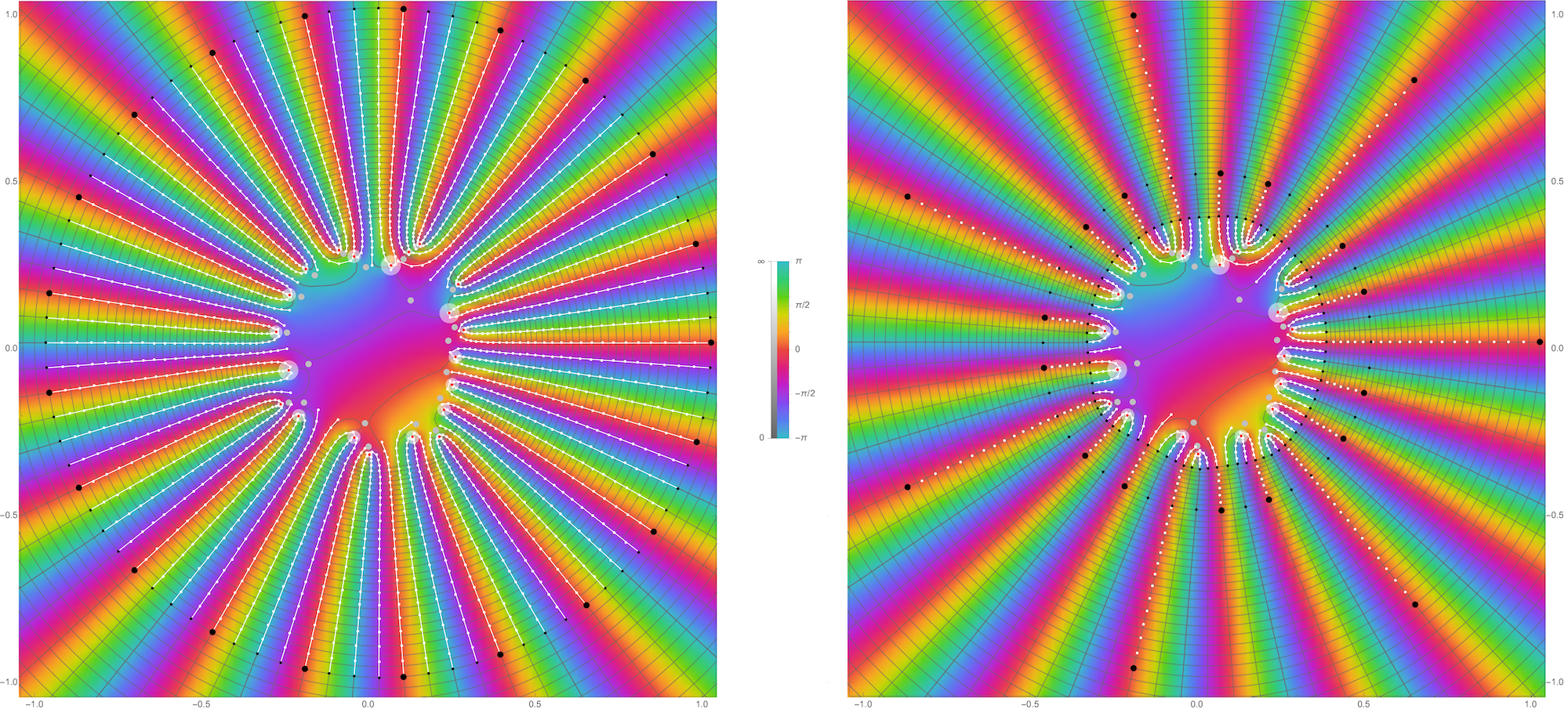}
\caption{\label{fig:idea}\small\sl
Newton's method to split a polynomial of degree $d=20$ with 4 starting points (black) per root.
The trajectories (white) are Newton's iterations, which are stoped in case of divergence;
the gray dots mark the critical points; the white disk are contracted to the roots (red).
The classical choice for the starting points (left) leads to $O(d^2)$ Newton steps.
A better choice, proposed in this article (right) optimizes the position of the starting points
by combining their computation with that of a level curve and brings, in practice, the total number of steps to~$O(d)$.
See also \cite[Fig.~2]{RSS2017} for a similar idea, though implemented differently.
}
\end{center}
\end{figure}

\noindent
Instead, we propose to choose the starting points of the Newton method on a level curve that hugs the set of roots tightly.
To produce this curve, one starts on an extremely lean mesh to perform most of the ``descend'' with as little computations as possible.
When needed, one can then densify the mesh at a certain level set (horizontally) and then continue the descend (vertically) with more points.
Ultimately, we catch up with the dense mesh of the classical method, having only a few Newton steps left until we find all the roots.

\medskip
We are releasing an exhaustive and proven list \citelib{MLibData}
of all the roots of $p_n$ (called hyperbolic centers) for $n$ up to period $41$
and all the roots of $q_{\ell,n}$ (called pre-periodic or Misiurewicz-Thurston points) of order $\ell+n\leq 35$.
By proven, we mean that each point $z_j$ in our database comes with a mathematical statement that guaranties two things:
an exact root $\tilde{z}_j$ lies within a certain maximal tolerance from our numerical value,
and a standard refinement technique (Newton's method) applied to $z_j$ converges to $\tilde{z}_j$.
To ensure a strict control of the computational errors that is robust enough to withstand intensive iterations,
we have revisited and extended the theory of disk arithmetic.

As a companion to this article, we also release an implementation of our algorithm as a standard \texttt{C} library~\citelib{MLib}. 
This library can be used to check and explore our database \citelib{MLibData}
and contains all the tools necessary to perform extractions,
generate images, compute arbitrary precision refinements, etc.
This library is optimized on all fronts: it can run on a single core for the smallest periods
or can take advantage of the massively parallel architectures of modern HPC centers for the higher ones. 
For example, the splitting of~$p_{21}$, which is reported as requiring at least 18.8 hours-core in~\cite{SS2017}
takes only 10~seconds with our algorithm, on a single core of a consumer-grade personal computer.
Similarly, we can split $p_{29}$ in only 1 hour-core and $p_{33}$ in about 1 day-single-core.
The computations automatically switch to high-precision, using the \citelib{MPFR} library, but only when necessary.
For example, when $n\geq 30$ the roots of $p_n$ near $-2$ are so close that they cannot be resolved in the standard
64 bits floating point arithmetic. Up to $n\leq 41$, we use numbers with 128 significant bits.
Throughout the library, custom data structures ensure near optimal memory usage, data access and computation times.

The case of the polynomials $q_{\ell,n}$ is interesting because, contrary to the $p_n$, they have multiple roots,
which slow down drastically the dynamics of the Newton map. However, our algorithm remains robust in this case.
We stopped at the order $\ell+n\leq 35$. The volume of data generated by the roots of all $(\ell,n)$ combinations
is of the same order of magnitude to that of the roots of all $p_n$ for $n\leq 39$ (see Section~\ref{par:scale}).

\medskip
Our database and its companion library constitute a powerful \textit{numerical microscope}
whose resolution extends far beyond the current state of the art and that can assist the mathematical
community in the exploration of the properties of $\M$.
While an in-depth analysis will take time, we already present in this first article
a few properties, like the minimal distance between the roots of a given type (see Section~\ref{par:HardwareLimit}).

\subsection{Structure of this article}\label{intro:structure}

Our article is organized as follows.
Section~\S\ref{par:mandel} contains a brief self-contained introduction to the Mandelbrot set.
Section~\S\ref{par:roots} presents a quick review of the most common splitting (or at least root finding) algorithms for polynomials.
A detailed presentation of our new algorithm is done in Section~\S\ref{par:split}, along with a general discussion of its algorithmic
complexity. The main statements, Theorems~\ref{thm:practicalFTA} and \ref{thm:levLineFlow}, are based on Theorem~\ref{thm:target}
proved in Appendix~\ref{par:mathBackground}.
As we do not claim full generality, the specifics of working with polynomials tailored to the Mandelbrot set is discussed too.
Readers interested in polynomials of high degrees may appreciate, on the side, our improvement of the standard evaluation scheme~\cite{AMV21},
along with its companion library~\cite{FPELib}.

\medskip
Section~\S\ref{par:certif} is focused on the question of providing certifiable results. The backbone of this section
is a theory of a disk arithmetic (Theorem~\ref{thm:diskArithmetic}),
which, contrary to interval arithmetic, tolerates a high number of iterations of holomorphic maps.

\medskip
Our implementation~\citelib{MLib} is presented in Section~\S\ref{par:num_results}. 
The computer assisted proof of the $3\times 10^{12}$ instances of the generic Theorem~\ref{thm:certif}
relies on Section~\S\ref{par:certif},
in particular Theorems \ref{thm:proofRoots}, \ref{thm:proofContract}.
We also discuss the practical cost of our algorithm as we scale up to splitting the tera-polynomial $p_{41}$.
This analysis confirms, in practice, the complexity claims of Section~\S\ref{par:split}.
A listing of the tasks is given in Appendix~\ref{MTasks}.

\medskip
A proof of our Theorem~\ref{thm:fact_mis} stated in Section \S\ref{par:mandel} on the factorization of $q_{\ell,n}$
is given in Appendix~\ref{par:factorization}.
A few background auxiliary mathematical results used in section \S\ref{par:certif} are recalled
briefly in Appendix~\ref{par:mathBackground}; Theorem~\ref{thm:target} and its corollary are new.
For the sake of convenience, the notations introduced throughout the article are listed in Appendix~\ref{par:notations}.

\medskip
Theorems are numbered independently (up to~\ref{thm:JUNG1985}).
All other statements and remarks follow a distinct continuous numbering.

\subsection{Thanks}

We are grateful to \textit{Romeo}, the HPC center of the University of Reims Champagne-Ardenne,
for hosting our project graciously even though its scale (near 50TB of data and 1~century-core / 870 kh of CPU usage
for both the development and production phase) is quite substantial for this regional structure.

\smallskip
It is our pleasure to thank our families who supported us in the long (5+ years) and, at times stressful,
process that gave birth to the \texttt{Mandelbrot} library.
We also thank Sarah Novak for a thorough proof-reading of this article.

\section{About the Mandelbrot set}
\label{par:mandel}
In this section, we will collect the strict minimum on the Mandelbrot set and fix notations
that will be needed subsequently. See Appendix~\ref{par:notations} for standard notations.

\bigskip
The \textit{Mandelbrot} set $\M$ is composed of the parameters~$c\in\C$ for which the sequence~$(p_n(c))_{n\in\N}$ (defined in \eqref{def:hyp_poly}) remains bounded (Fig.~\ref{fig:mandelbrot}).
This sequence is the orbit by iterated compositions of the only critical point $z=0$ of the map $f_c(z) \defequal z^2 +c$, \ie
\begin{equation}\label{eq:defFundMap}
\forall n\in\N,\qquad
f_c^n (0)\defequal \underset{\text{$n$ times}}{\underbrace{\,f_c \circ \ldots\circ f_c}}(0)  = p_n(c).
\end{equation}
For short, the sequence $(p_n(c))_{n\in\N}$ is called the \textit{critical orbit} of $f_c$.
$\M$ is a connected set, whose complement in $\CC$ is simply connected \cite{DOUHUB82}, \cite{Sib}.  
One has
\[
\M\subset \disk(0,2)\cup\{-2\}.
\]
There are many tools online to explore~$\M$, both for scientific purposes or just for
its intrinsic beauty; see \eg\cite{CHE10} or \cite{EXPL1}.

\medskip 
For $c\in\C$, the dynamics of $f_c$ splits $\CC$ in two complementary sets.
The \textit{Fatou} set~$\fatou_c$ is the open subset of $z\in\CC$ in the neighborhood of which, the sequence $(f_c^n)_{n\in\N}$ is a normal family
\ie precompact in the topology of local uniform convergence.
On the contrary, on the \textit{Julia} set~$\julia_c=\CC\backslash\fatou_c$, the dynamics appears to be, loosely speaking, chaotic.
Both~$\fatou_c$ and~$\julia_c$ are fully invariant (\ie invariant sets of the forward and backwards dynamics)
and that $c\in\M$ if and only if $\julia_c$ is connected.
For a review of the properties of Fatou and Julia sets, see~\eg\cite{CARLESON}, \cite{MILNOR} for polynomial and rational maps
and~\cite{Ber93}, \cite{MRW22} for entire and meromorphic functions.

\subsection{Hyperbolic parameters}

For $n\geq1$, the roots of $p_n$ are parameters $c\in\M$, called \textit{hyperbolic centers},
whose critical orbits are periodic of period $n$.
The roots of $p_n$ are simple~\cite{DOUHUB82, BUFF2018}
and the polynomial $p_n(z)$ is divisible by $p_k(z)$ for any divisor $k$ of $n$.
A.~Douady and J.H.~Hubbard \cite{DOUHUB82, DOUHUB84} have shown that the hyperbolic centers are interior points of~$\M$,
with at most one hyperbolic center per connected component of the interior.
The set of all hyperbolic centers is also dense in the boundary of $\M$ in the sense that
its closure in $\C$ contains $\partial\M$.

\medskip
Let us define the set of hyperbolic centers of \textit{order} $n\geq1$
as the subset of $p_n^{-1}(0)$ whose minimal (or fundamental) period is exactly $n$; in other terms:
\begin{equation}\label{def:hyp}
\hyp(n) \defequal \left\{ z\in p_n^{-1}(0) \,\big\vert\, \forall k\in\div(n)^\ast, \: p_k(z)\neq 0 \right\}
\end{equation}
where $\div(n) \defequal \{ k\in\N^\ast \,;\, \exists k'\in\N, \enspace \enspace k k' = n\}$
and $\div(n)^\ast=\div(n)\backslash\{n\}$ is the set of strict divisors of $n$.
For example,~$\hyp(1) = \{0\}$ and $\hyp(2) = \{-1\}$. The reduced hyperbolic polynomial, also known as Gleason's polynomial, is:
\begin{equation}\label{eq:hyp_red}
h_n(z) \defequal \prod\limits_{r\in\hyp(n)}(z-r).
\end{equation}
The polynomials $p_n$ and $h_n$ have integer coefficients.
While expected, the irreducibility of $h_n$ over $\Z[z]$ remains conjectural (see \cite[last remark of~\S3]{HT15}, \cite[p.155]{SS2017}).

\begin{thm}[folklore, included in \cite{HT15}]\label{thm:countingHyp}
The complete factorization of $p_n$ is
\begin{equation}\label{eq:HypFactor}
p_n(z) = \prod_{k\vert n}h_k(z).
\end{equation}
Moreover, the cardinal of $\hyp(n)$ is given by
\begin{equation}\label{eq:HypCount}
\left|\hyp(n)\right| = \sum_{k\vert n} \mu(n/k)2^{k-1}
\end{equation} 
where $\mu$ is the M\"obius function, \ie
\[
\mu(n)=\begin{cases}
(-1)^{\nu} & \text{if $n$ is square free and has $\nu$ distinct prime factors,}\\
0 & \text{if $n$ is not square free.}
\end{cases}
\]
\end{thm}
\noindent
The notation $k\vert n$ means $k\in\div(n)$ and $|S|$ denotes
the cardinal of a finite set.
The identity~\eqref{eq:HypFactor} is well known and is key in the count of hyperbolic centers.
The Online Encyclopedia of Integer Sequences~\cite{CARDHk} attributes~\eqref{eq:HypCount} to 
Warren D. Smith and Robert Munafo (2000), sadly with no published reference.
For a more general result that applies to iterated maps of the form $z^d+c$, see~\cite{HT15}.

\begin{figure}[H]
\captionsetup{width=.95\linewidth}
\begin{center}
\includegraphics[width=\textwidth]{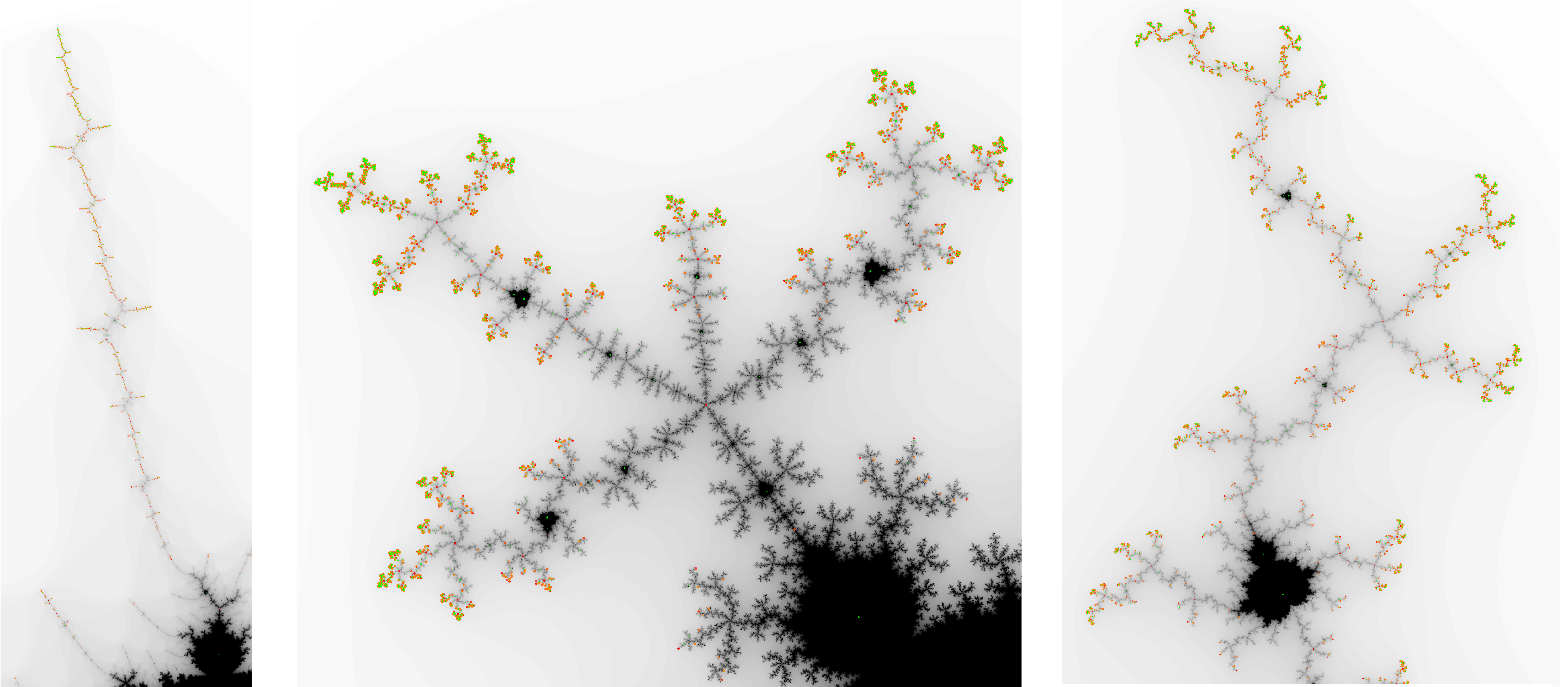}
\caption{\label{fig:HypMis}
The hyperbolic points $\hyp(n)$ for $n\leq18$ in green and
the Misiurewicz-Thurston parameters $M_{\ell,n}=\mis(\ell,n)$ with $\ell+n\leq16$ in red
in different parts of the Mandelbrot set~$\M$, with $\Re c$ increasing from left to right
between images.
}
\end{center}
\end{figure}

\subsection{Pre-periodic or Misiurewicz-Thurston parameters}

For integers $n\geq1$ and $\ell\geq2$, the roots of $q_{\ell,n}$ (defined in \eqref{def:mis_poly}) are parameters $c\in\partial\M$, called \textit{pre-periodic centers}
or \textit{Misiurewicz-Thurston} points, whose critical orbits are pre-periodic, that is it becomes periodic of period $n$ after the first $\ell$ steps.
For $\ell\in\{0,1\}$, one can check immediately that
\begin{equation}\label{eq:hypCase}
q_{0,n}(z) = p_n(z) \qquad\text{and}\qquad q_{1,n}(z) = p_n^2(z).
\end{equation}
The dynamical reason for the second identity in~\eqref{eq:hypCase} is that $0$ is the only pre-image of~$c$ under $f_c$
so a pre-periodic orbit of type $(1,n)$ starting at zero actually loops back to zero, so is periodic of period $n$.
The polynomial $q_{\ell,n}(z)$ is divisible by $q_{\ell,k}(z)$ for any divisor $k$ of $n$ and by $q_{\ell',n}(z)$ for any $\ell'\leq\ell$.
Some of the roots of $q_{\ell,n}(z)$ have multiplicity; however, the polynomial
\begin{equation}\label{eq:polySimple}
s_{\ell,n}(z)=\frac{q_{\ell,n}(z)}{q_{\ell-1,n}(z)} \in \Z[z]
\end{equation}
has simple roots \cite{BUFF2018}, \cite[Lemma~3.1]{HT15}. 
Douady-Hubbard \cite{DOUHUB82, DOUHUB84} have shown that the set of all pre-periodic points is dense in the boundary of $\M$.
Visually, those points are either branch tips, centers of spirals or points where branches meet (see Fig.~\ref{fig:HypMis}).

\medskip
By analogy with the hyperbolic case, let us define the set of \textit{Misiurewicz points of type $(\ell,n)$}
as the subset of $q_{\ell,n}^{-1}(0)$ whose dynamical parameters are exactly $\ell$ and $n$; in other terms:
\begin{equation}\label{def:mis}
\mis(\ell,n) \defequal \left\{ z\in q_{\ell,n}^{-1}(0)
\left|\begin{array}{l}
q_{\ell-1,n}(z)\neq 0,\\[3pt]
\forall k\in\div(n)^\ast, \: q_{\ell,k}(z)\neq 0
\end{array}\right.\right\}.
\end{equation}
We call $\ell+n$ the \textit{order} of $\mis(\ell,n)$ because $q_{\ell,n}$ is a polynomial of degree $2^{\ell+n-1}$.
The reduced polynomial whose roots are exactly $\mis(\ell,n)$ is denoted by
\begin{equation}\label{eq:mis_red}
m_{\ell,n}(z) \defequal \prod\limits_{r\in\mis(\ell,n)}(z-r) \in \Z[z].
\end{equation}
Note that $\mis(0,n)=\hyp(n)$ and $\mis(1,n)=\emptyset$ because of~\eqref{eq:hypCase}.
The count of the Misiurewicz points has been established by B.~Hutz and A.~Towsley \cite[Cor.~3.3]{HT15}:
\begin{equation}\label{eq:MisCount}
\left|\mis(\ell,n)\right| = \Phi(\ell,n)\left|\hyp(n)\right|
\end{equation}
where
\[
\Phi(\ell,n) = \begin{cases}
1 & \text{if } \ell = 0,\\
2^{\ell-1} -1 & \text{if } \ell \neq 0 \text{ and }n\vert \ell-1,\\
2^{\ell-1} & \text{otherwise.}
\end{cases}
\]

\medskip
To get a complete factorization of $q_{\ell,n}(z)$ in terms of~\eqref{eq:hyp_red} and~\eqref{eq:mis_red},
one needs to understand the multiplicity of the hyperbolic factors. We claim the following result.

\begin{thm}\label{thm:fact_mis}
For $\ell\in\N$ and $n\in \N^\ast$, one has
\begin{equation}\label{eq:MisFactor}
q_{\ell,n}(z) = \prod_{k\vert n} \left( h_k(z)^{\eta_\ell(k)} \prod_{j=2}^{\ell} m_{j,k}(z) \right)
\end{equation}
where the multiplicity $\eta_\ell(k)$ is given by\footnote{In~\eqref{eq:defMultHyp}, the floor function is denoted by $\lfloor x \rfloor = n$ if $n\in\Z$ and $x\in [n,n+1)$.}
\begin{equation}\label{eq:defMultHyp}
\eta_\ell(k) \defequal \left\lfloor \frac{\ell -1}{k} \right\rfloor + 2.
\end{equation}
In other words, the roots of $q_{\ell,n}$ are composed of the points $\hyp(k)$ for any divisor $k$ of $n$,
which are roots with multiplicity $\eta_\ell(k)$, and of the points $\mis(j,k)$ for $2\leq j\leq \ell$, which are simple roots.
\end{thm}
\noindent
This result appears to be new. We propose a direct proof in Appendix~\ref{par:factorization}.

\subsection{Harmonic measure}\label{par:harmonicMeasure}

The sets of all hyperbolic-centers and of all Misiurewicz-Thurston points are respectively denoted by:
\[
\hyp =\bigcup_{n\in\N^\ast} \hyp(n) \subset\overset{\circ}{\M}
\qquad\text{and}\qquad
\mis=\bigcup_{\substack{\ell\geq2\\[1pt]n\in\N^\ast}} \mis(\ell,n) \subset\partial\M.
\]
As mentioned above, $\overline{\mis} = \partial\M \subset \overline{\hyp}$.
The density measures of $\hyp$ and $\mis$ (\ie the limit of the uniform counting measure on finite subsets of increasing order) both coincide with the \textit{harmonic measure} for $\partial(\M^c)$,
which is the image of the uniform measure on the unit circle under the Riemann map
$\disk(0,1)^c\to \M^c$.
Intuitively, the harmonic measure expresses the asymptotic density of the external rays as they land on the boundary
of the Mandelbrot set.
The support of the harmonic measure is the only subset of $\partial\M$ that is
``visible'' to a random brownian motion starting at $\infty$ and is of Hausdorff dimension~1.
For further details, see~\cite{Mak84}, \cite{LEVIN1990},  \cite{BelSmi04}, \cite{FG2013}, \cite{GV2016b}, \cite{GV2016a}
and Remark~\ref{rmkRiemannMap} below.

\subsection{Scale of the computational challenge}\label{par:scale}

Using the library~\citelib{MLib} associated with this article,
we have completely determined the sets $\hyp(n)$ for all $n\leq 41$ and $\mis(\ell,n)$  for $\ell+n\leq 35$.
Both problems consists in solving polynomial equations, namely~\eqref{def:hyp_poly} or~\eqref{def:mis_poly}, of respective degrees $2^{n-1}$ or $2^{\ell+n-1}$.
Overall, we have sifted through about $3.3\times 10^{12}$ polynomial roots and identified~$2\,725\,023\,424\,662$ distinct parameters of the Mandelbrot set.
About 80\% of them are hyperbolic centers.

\medskip
The tera-polynomial $p_{41}$ has degree $1\,099\,511\,627\,776$ and, as $41$ is a prime number,
\[
p_{41}^{-1}(0) = \hyp(41) \cup \{0\}.
\]
More generally, the deflation of $\hyp(n)$ due to strict divisors of~$n$ constitutes an asymptotically negligible subset
of the roots of $p_n$.

\medskip
The situation is not exactly the same for $\mis(\ell,n)$.
There are $33$ non-hyperbolic pre-periodic polynomials $q_{\ell,n}$ of order $\ell+n=35$, each of degree $17\,179\,869\,184$.
However, the exact count $271\,590\,481\,057$ of $\bigcup\mis(\ell,n)$ for all types of order 35 represents only $47.9\%$ of the total number of roots of
those polynomials. More generally, \eqref{eq:MisCount} ensures that, for large $N$:
\[
\sum_{\substack{\ell+n = N\\\ell\neq 0}}|\mis(\ell,n)| \simeq \frac{N-2}{2} \: |\hyp(N)|.
\]
The deflation of $\mis(\ell,n)$ due to divisors is therefore, asymptotically, 50\% as $\ell+n\to\infty$.

\medskip
To convey a sense of the orders of magnitude involved for storing and processing the roots of polynomials of giga- and tera-scale,
let us underline that a computation that takes, on average, only 1 milli-second per root requires about $12$ days-core for a polynomial of degree $10^9$.
This requirement jumps to $32$ years-core ($\sim 280$k hours-core) for a polynomial of degree $10^{12}$.

\medskip
Our actual computation times are discussed in Section~\S\ref{par:num_results}.
The success of our entreprise is due to the efficiency of our new splitting algorithm for $p_n$ and $q_{\ell,n}$ (see Section~\ref{par:split}), which is bounded by $O(d\log d)$.


\section{A brief review of root finding algorithms}\label{par:roots}

In the rest of the article,
one will assume that $P\in\C[z]$ is a polynomial of degree $d=\deg P\geq1$ and that its roots
$\mathcal{Z}=P^{-1}(0)$ are localized in the disk $\disk(0,r)$.
See Theorem~\ref{thm:rootLoc} in Appendix~\ref{par:mathBackground} to estimate $r$ from the coefficients of $P$.

\bigskip
A \textit{splitting algorithm} is an algorithm that provides the list of all the roots
of $P$ through convergent numerical approximations.
Splitting algorithms can be classified in three broad families: 
methods based on Newton's iterations, root isolation methods
and eigenvalue methods. Let us review them briefly.

\bigskip
We denote by $\eval_d$ the \textit{arithmetic complexity} of evaluating a polynomial of degree $d$.
Horner's method ensures that, in general, $\eval_d  =  O(d)$ when the coefficients are known.
Some evaluation schemes offer a similar complexity with better balanced intermediary computations,
like Estrin's method, which is implemented in the Flint library \citelib{FLINT}, or some variants~\cite{Mor13}.
Other schemes take advantage of multipoint evaluations (see \eg\cite{KS16}).
In a separate article~\cite{AMV21}, we propose a \textit{Fast Polynomial Evaluator} algorithm and its practical implementation~\cite{FPELib}
that brings the average cost of evaluating real and complex polynomials with a fixed precision down to $\eval_d = O(\sqrt{d\log d})$
after a preprocessing phase that can be performed independently of the precision\footnote{In the pre-processing phase of the FPE algorithm,
only the integer part of the base-2 logarithm of the coefficients is needed.} and that costs $O(d\log d)$.

For iteratively defined polynomials (\eg for $p_n$, which is of degree $d=2^{n-1}$),
one obviously has $\eval_d=O(\log_2 d)$.
To keep comparisons fair between algorithms, we will therefore explicitly factor out the cost of evaluations whenever possible.

\begin{remark}\label{rmk:costCheckList}
If a polynomial $P$ has simple roots,
the computational cost of checking that a sorted list of numbers is indeed a list of all
the roots of $P$ is  $O(d \eval_d)$.
To check that a root $z$ is of multiplicity $k$, one computes $P(z)$, $P'(z)$,
\ldots, $P^{(k-1)}(z)$ instead of computing $k$ distinct values of $P$ so the
contribution to the total cost remains of order~$k \eval_d$ provided that each derivative can
also be evaluated in time $O(\eval_d)$. The cost of checking roots thus remains $O(d \eval_d)$.
\end{remark}

\subsection{A word on numerical instabilities}\label{par:wilkinson}

A naive splitting algorithm consists in finding a root of $P$ by any available method, and then reduce
the degree of $P$ by performing a euclidian division.

\bigskip
In practice, this method is plagued by numerical instabilities, illustrated on Wilkinson's polynomial \cite{W84}
\[
W_{20}(x) = \prod_{n=1}^{20}(x-n)
\]
whose largest root is $20^{19}$ times less stable than its smallest root when the coefficients
are perturbed in the canonical basis (though other basis may behave better).
In general, determining the roots of a polynomial from the list of its coefficients is a highly ill-conditioned problem.

Consequently, for very large polynomials, computing the coefficients is not always advisable.
For example, as $p_3(z)=z^4+2z^3+z^2+z$, the largest coefficient of $p_n(z)$ for $n\geq3$ exceeds $2^{2^{n-3}}$
while the smallest one remains $1$.
Storing the coefficients of $p_{41}(z)$ exactly is, at best, impractical and any reasonable approximation
of those coefficients would not be sufficiently precise to compute the roots accurately.

\subsection{Root isolation methods}

Bracketing (or root isolation) consists in identifying an interval of the real line or a domain of the complex plane
where the number of roots of $P$ is prescribed.

\bigskip
Methods based on this principle are usually developed on the
real line where they constitute a textbook application of the intermediary
value theorem (bisection method, or the more involved ITP method).

In the complex plane, the  splitting circle method and the Lehmer–Schur algorithm 
are based on the residue theorem, which implies that, for all $n\in\N$
\[
\sum_{\substack{z\in\Omega\\[2pt] P(z)=0}} m_P(z) z^n = \frac{1}{2i\pi} \int_{\partial\Omega} \frac{P'(z)}{P(z)} z^n dz
\]
for any smooth domain $\Omega$ whose boundary avoids the zeros of $P$ and where $m_P(z)$ denotes the multiplicity
of $z$ as a root of $P$. Successive
refinements of $\Omega$ will either isolate single roots or clusters of close roots that can
be identified using Newton's identities.
Aside from the cost of computing the integral to a sufficient precision,
those methods are plagued by the geometrical problem of finding a ``good'' circle,
\ie one which is proper to initiate a sequence of divide and conquer iterations.

\medskip
Of similar nature is Graeffe's method, where the sequence
\[
P_{n+1}(z^2)=(-1)^d P_n(z)P_n(-z)
\]
is computed from $P_1=P$. The roots of $P_n$ are $\{z^{2^n} \,;\, P(z)=0\}$ and
are thus, generically, exponentially separated from each other. In that case Vieta's formulas can easily
be inverted approximately, which gives the roots of $P_n$ and ultimately, those of $P$.

\bigskip
Root isolation methods excel in finding roots of a given polynomial, but they are not
a splitting algorithms per se. We will not insist on comparing their complexity.

Our algorithm (see \S\ref{par:split} below) presents a strong familiarity with the splitting
circle method: the level lines that we will use provide a tight enclosure of the roots of $P$ by a Jordan curve.
Tighter level sets would enclose the roots in a finite union of disjointed connected sets (see Fig.~\ref{fig:idea} and~\ref{fig:MVSplitAlgo}). 

\subsection{Eigenvalue algorithms}
Splitting a polynomial $P$ can always be restated as the question of finding all the eigenvalues
of the companion matrix $A_P\in\mathcal{M}_d(\C)$, at least provided that the coefficients of $P$ are known.
\[
A_P =  
\begin{pmatrix}
0 & 1 & 0 & \ldots & 0\\
0 & 0 & 1 & \ddots & \vdots\\
\vdots &   & \ddots & \ddots & 0\\
0 & \cdots & \cdots & 0  & 1\\
c_0 & c_1 & \cdots & \cdots & c_{n-1} 
\end{pmatrix}
\qquad\text{if}\qquad
P(x)=x^n - \sum_{k=0}^{n-1} c_k x^k.
\]
The vector $(1,x,x^2,\ldots, x^{n-1})$ is an eigenvector of $A_P$ if and only if $P(x)=0$.

\bigskip
The complexity of eigenvalue algorithms usually depends on the structure of the matrix under consideration.
For example, the modern version of the QR algorithm for
Hessenberg matrices\footnote{A Hessenberg matrix has zero entries above the first superdiagonal,
which is the case for $A_P$.}
costs $O(d^2)$ per iteration \cite{BDG04}, \cite{CGXZ07}.
The sequence starts with $A_0=A$, the matrix whose spectrum is being computed.
The principle of each step consists in an orthogonal similarity transform $A_{k+1} = Q_k^T A_k Q_k$ where
\[
A_{k} = Q_k R_k \quad\text{with}\quad \begin{cases}
Q_k \text{ orthogonal}\\
R_k \text{ upper triangular}
\end{cases}
\]
As $Q_k^T Q_k= I$, it boils down to computing $Q_k$, $R_k$ and then forming  $A_{k+1} = R_k Q_k$.
The number of iterations used in practice until $A_k$ reaches the Schur form of $A$
(upper triangular) may be $O(d)$ or higher.
Companion matrices belong to the Hessenberg class, so, as
a splitting algorithm, the QR method costs at least $O(d^3)$.

\medskip
Note that the polynomial $P$ is not explicitly evaluated here,
so the factor~$\eval_d$ is not included in the arithmetic complexity.
However, it appears implicitly in each step as a cost of matrices operations.

\begin{remark}
Many applications that require polynomial splitting arise naturally
as an eigenvalue problem. For example, in structural mechanics, the asymptotic stability analysis
of an engineering design calls for the determination of the spectrum of some PDE operator,
which can be approximated by a large band-limited finite-elements matrix.
Most linear algebra algorithms are tailored for the needs of the specific applications, like finding
the extreme eigenvalues or their distribution \cite{ACE09}, or determining the eigenvectors.
They do not always constitute proper splitting algorithms, even though they can occasionally
be repurposed as such.
\end{remark}


\subsection{Newton's iterations}\label{par:Newton}
The core of many iterative splitting algorithms are Newton's iterations:
\begin{equation}\label{eq:Newton}
z_{n+1}=N_P(z_n)
\qquad\text{where}\qquad N_P(z) = z-\frac{P(z)}{P'(z)}\cdotp
\end{equation}
The fixed points of the dynamics of $N_P$ are the roots of $P$. 

\medskip
In the vicinity of each simple root, this method converges bi-exponentially.
More precisely, if $z_\ast=\lim z_n$ with $P'(z_\ast)\neq0$, Taylor's formula  for $P$ can be rewritten
\[
z_{n+1}-z_\ast = \frac{(z_n-z_\ast)^2}{P'(z_n)}\int_0^1 (1-t) P''(t z_\ast + (1-t)z_n) dt.
\]
For $\rho>0$, one defines $C_\rho \defequal \frac{1}{2} \sup |P''(z)| / \inf |P'(z)|$, each extremum
being computed for $z\in\disk(z_\ast,\rho)$. One can always assume that $\rho C_\rho<1$
because $C_\rho \to \frac{1}{2}\left|\frac{P''(z_\ast)}{P'(z_\ast)}\right|\in\R_+$ as $\rho\to0$, so $\rho C_\rho\to 0$.
In that case, if $z_{n_0}\in\disk(z_\ast,\rho)$ then
\begin{equation}\label{newtonQuad}
\forall n\geq n_0,\quad
|z_{n+1}-z_\ast| \leq C_\rho |z_{n}-z_\ast|^2 \quad\ie\quad |z_{n}-z_\ast|\leq C_\rho^{-1}(\rho C_\rho)^{2^{n-n_0}}.
\end{equation}

\medskip
When the root~$z_\ast$  is of multiplicity $\mu\geq2$, then the convergence of Newton's algorithm 
is only exponential. Indeed, one has
\[
N_P'(z) = \frac{P(z)P''(z)}{P'(z)^2} \underset{z\to z_\ast}{\longrightarrow} 1-\frac{1}{\mu} \in \left[\textstyle{\frac{1}{2}},1\right)
\]
so for any $z_0$ in the bassin of attraction
\[
\attract(z_\ast) = \{ z\in\C \,;\, \lim_{n\to\infty} N_P^n(z) = z_\ast\}
\]
the estimate~\eqref{newtonQuad} is replaced, using Taylor's formula for $N_P$, by
\begin{equation}\label{newtonLin}
\frac{z_{n+1}-z_\ast}{z_n-z_\ast} \to 1-\frac{1}{\mu}\cdotp
\end{equation}
In that case, the accelerated scheme
\begin{equation}\label{newtonAcc}
z_{n+1}'=z_n'-\mu\cdot\frac{P(z_n')}{P'(z_n')}
\end{equation}
could be used instead to restore an ultimately bi-exponential behavior in the vicinity of multiple roots.
In practice, however, $\mu$ may be unknown and, even if it was known, this accelerator will completely destroy the
attraction bassins of the simple roots, as illustrated on Figure~\ref{fig:newtonAccel}.
A better accelerator may be K\"onig's method \cite{BufHen03} that replaces~$P/P'$ in~\eqref{eq:Newton}
by ratios of consecutive higher-order derivatives of $1/P$. 

\begin{figure}[H]
\captionsetup{width=.95\linewidth}
\begin{center}
\includegraphics[width=\textwidth]{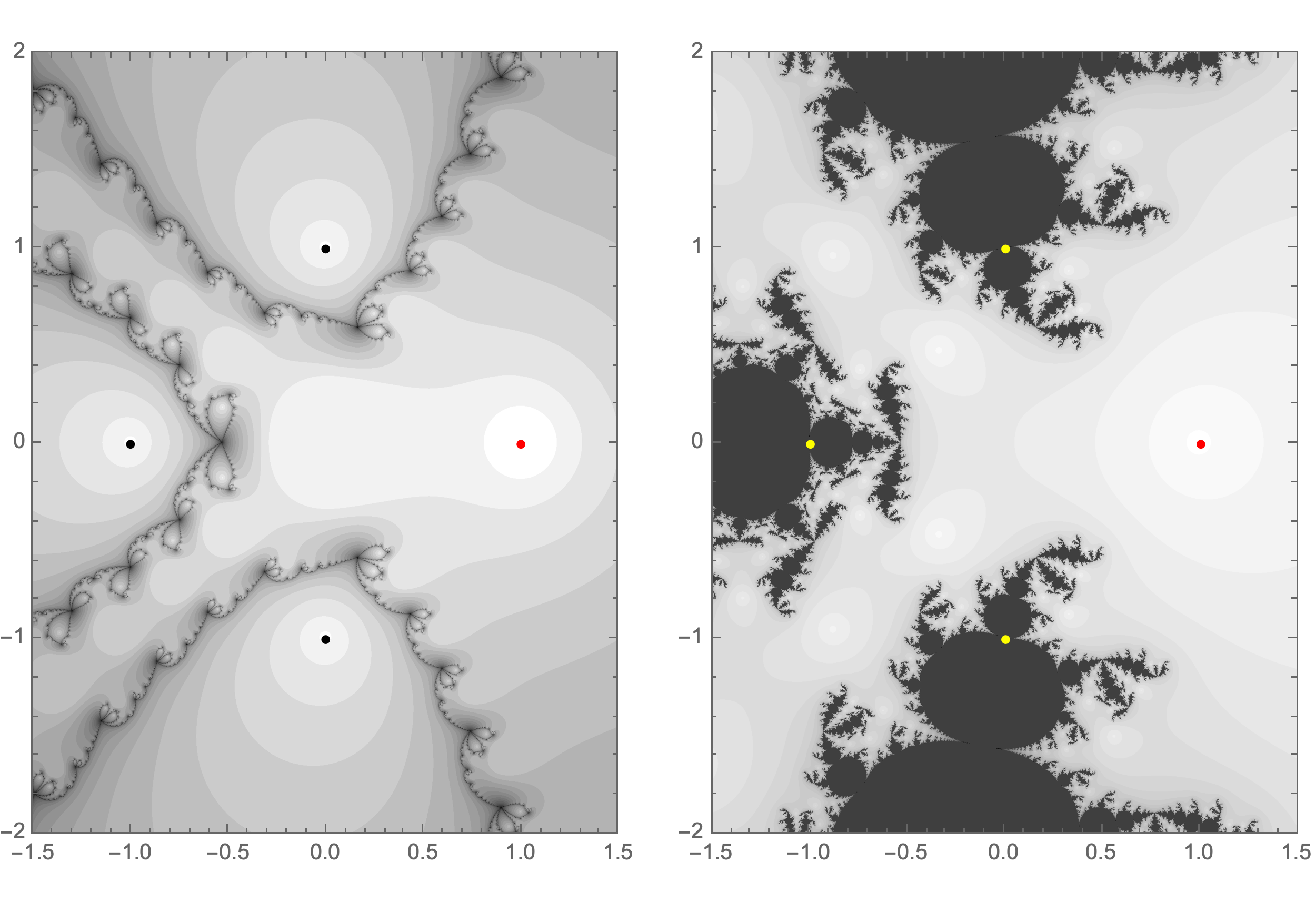}
\caption{\label{fig:newtonAccel}
Attraction bassins and Julia set of the Newton method (left) and of the accelerated method with $\mu=2$ (right)
for $P(z)=(z-1)^2(z+1)(z^2+1)$.
}
\end{center}
\end{figure}

The so called, quasi-Newton variants (Muller, Steffensen) avoid computing the derivative for numerical purposes,
in particular if the method is applied to a general function whose derivative is not known or not easily
computable. A typical example is the secant method.
Another class of quasi-Newton methods involve higher-order derivatives (Halley, Householder,
Laguerre, Jenkins-Traub,\ldots): they give a higher-order of convergence in exchange for a tighter
requirement on the smoothness of the function and higher-order norms in the estimates.
See \eg\cite{Cam19} for a parallel implementation of Laguerre's method.

\bigskip
Transforming Newton's method into a splitting algorithm requires choosing a set of starting points
for which the dynamics will visit every root. The naive approach of findind one root at
a time and factoring it out is both unpractical and numerically unstable (see \S\ref{par:wilkinson}).
Choosing starting points at random may miss some roots entirely.

For a given $z_0\in\C$ for which the sequence converges to~$z_\ast$,
the number of steps required before the sequence enters  the bi-exponential regime \ie $\min\{n \,;\, z_n \in \disk(z_\ast,\rho)\}$
with $\rho$ as in~\eqref{newtonQuad} may be astronomical.
Occasional wild excursions from the vicinity of critical points of $P$ towards~$\infty$
are not the main problem: those trajectories are easily identified and abandoned.
The less understood regime happens when the sequence crosses the mid-range
$\disk(0,r)\backslash \disk(z_\ast,\rho)$ where the sequence may almost stall if the Newton's bassin forms a narrow
gorge, as described in~\cite{HSS2001}; see also Figure~\ref{fig:idea}.
Any theoretical advance on this matter like~\cite{BAS2016} are of interest to improve the
bounds on the complexity of Newton's method as a global splitting algorithm.

The stoping criterion is another critical issue: P.~Henrici~\cite[Chap.~6]{HEN1974}
states that one may stop the Newton method when its steps get smaller then $\varepsilon/d$
where $\varepsilon$ is the precision sought after, in which case the proximity of a root can be proved.
In our algorithm, we stopped much earlier and relied instead on an a-posteriori analysis,
performed in a higher precision and using disk arithmetic to prove that the sequence has entered
the quadratic convergence region (see~\S\ref{par:certif}).

\bigskip
Aberth-Ehrlich method \cite{A73} is a generalisation of Newton's method.
It attempts to find all the roots simultaneously by computing a fixed point 
for $z_n=(z_{n,j})\in \C^d$~:
\[
z_{n+1,j} = z_{n,j} - \frac{ 1 }{
\frac{P'(z_{n,j})}{P(z_{n,j}) } - \sum\limits_{k\neq j} \frac{1}{z_{n,j}-z_{n,k}} }\cdotp
\]
The idea is that each root repels  the others with an electrostatic potential,
which prevents the Newton's dynamics to converge multiple times on a single
root while missing another one completely.
The components of the starting vector $z_0\in\C^d$ are chosen within~$\disk(0,r)$
and presumed to lack any symmetries that would accidentally match a symmetry of the set of roots,
which could prevent convergence.

The Aberth method is at the heart of the implementation of \texttt{MPSolve} \cite{BG2000,BR2014},
which is the reference software for splitting polynomials using arbitrary precision arithmetic.
A massively parallel implementation \cite{GSKCx} of this algorithm
on multiple GPUs can run up to a few millions of roots.

At each step, computing the electrostatic potential requires $O(d^2)$ operations,
while the evaluation of all $P'(z_{n,j})/P(z_{n,j})$ costs $O(d\eval_d)$.
The typical number of steps to reach a given approximation depends  strongly on the choice
of initial points \cite{G86}, \cite{F89}, \cite{B96}; in practice, it takes at least $O(d)$ steps,
with an overhead of $O(d \log d)$ operations to compute a ``good'' initial vector $z_0\in\C^d$.
The overall cost of Aberth's method is thus~$O(d^3)$ operations
regardless of the actual value of $\eval_d$.

In practice, the Aberth method has good global convergence properties.
However, a theoretical foundation for this global convergence remains an open problem.

\bigskip
Weierstrass method (also known as Durand–Kerner) is the following variation:
\[
z_{n+1,j} = z_{n,j} -\frac{P(z_{n,j})}{\prod\limits_{k\neq j} (z_{n,j}-z_{n,k})}\cdotp
\]
On a theoretical level, it was recently proven \cite{RST2020}, \cite{KI04} that the Weierstrass method
may fail generically, \ie on an open set of polynomials for an open set of intialization points.
This is a sharp reminder that one must pay close attention to the bassin of attraction of
each numerical method.

\subsection{The algorithm of Hubbard, Schleicher and Sutherland}\label{par:HSS}

The global structure of the bassins of attraction of Newton's method is quite intricate.
It has been extensively studied in~\cite{SU1989}, \cite{HSS2001}, \cite{SS2017},
where a remarkable universal algorithm for splitting $P$ is explained.
Let us consider a root $z_\ast$ of $P$
and denote by $\attract(z_\ast)\subset\CC$ the basin of attraction for Newton's method.

\bigskip
The connected component of~$z_\ast$ in~$\attract(z_\ast)$ is called the immediate basin of attraction and is denoted by $\mathcal{U}(z_\ast)$.
It admits as many \textit{access} to infinity (\ie homotopy classes of simple arcs connecting~$z_\ast$ to~$\infty$)
as the number of critical points of $N_P$ within it, counted with multiplicity (including the root itself as critical point if $z_\ast$ is a multiple root). The \textit{girth} (or \textit{channel width}) at
$z\in \mathcal{U}(z_\ast)$ is  the length of the connected component of $z$ in
\[
\{z' \in \mathcal{U}(z_\ast) \,;\,  |z'| = |z|\}.
\]
The girths admits a universal bound from below,
which implies that taking about $2.47 d^{3/2}$ points uniformly along a wide enough circle ensures
that there is at least one point per basin of attraction.
This property turns Newton's method into a trivially parallelizable global splitting algorithm. 

\medskip
A more involved analysis based on the conformal modulus of the access channels
(roughly speaking: the ratio between the girth of the channel to the amplitude
of one Newton step from that point) allows both a reduction in the number of points involved,
and better estimates.
A universal starting mesh, given in~\cite{HSS2001},
is composed of about $\alpha\simeq 4.16 d\log d$ points taken on $\beta \simeq 0.27\log d$ concentric circles.
\begin{thm}[Splitting algorithm of \cite{HSS2001}]\label{thm:HSS}
Given a polynomial $P(z)$ of degree $d$ whose roots are contained in $\disk(0,r)$ and
integers $\alpha\geq 4.16 d\log d$, $\beta \geq 0.27\log d$, the mesh
\[
\mesh_r(d)= \bigcup_{1\leq\nu\leq\beta} r (1+\sqrt{2})\left(1-\frac{1}{d}\right)^{\frac{2\nu-1}{4\beta}} \U_{\alpha}
\]
contains a point in each bassin of attraction of the roots of $P$, for Newton's method.
Here, $\U_\alpha = \left\{ e^{2ik \pi/\alpha} \,;\, 0\leq k<\alpha\right\}$ denotes the roots of unity of order $\alpha$.
\end{thm}
\noindent
Overall, $| \mesh_r(d) |\simeq 1.11 d \log^2 d$.
For $d\simeq 2^{32}$, the mesh is build upon $\beta=6$ circles.
The computational improvement over a one-circle approach, which requires $O(d^{3/2})$ starting points,
becomes advantageous when $d\geq 2^{14}$ (see Fig.~\ref{fig:gridToPointRatio}).

\subsection{Practical considerations regarding the algorithm of \cite{HSS2001}}

In this subsection, we discuss some practical aspects of the implementation of Theorem~\ref{thm:HSS},
seen as a splitting algorithm, that will lead (for the polynomials $p_n$ and $q_{\ell,n}$) to our
improved algorithm presented in Section~\ref{par:split}.

The overall complexity of J.H.~Hubbard, D.~Schleicher and S.~Sutherland splitting algorithm 
is $\alpha\beta\gamma$ Newton steps (\ie $\alpha\beta\gamma V_d$ arithmetic operations),
where $\alpha\beta=| \mesh_r(d)|$ is the number of starting points and $\gamma$ is the average
number of Newton iterations that are actually necessary to compute the roots with a precision $\varepsilon$.
The question of estimating~$\gamma$ is two-fold.
It obviously depends on the minimal separation of the roots of $P(z)$,
which dictates the precision $\varepsilon$ of the (final) computations.
In \cite{HSS2001}, no upper bound on $\gamma$ is given. However, a practical number of iterates is recommended:
\begin{equation}\label{eq:reccomendedGamma}
\gamma \simeq d \log\left(\frac{r}{\varepsilon}\right)
\end{equation}
assuming that roots are simple.
On the other hand, $\gamma$ is also profoundly impacted by the mid-range dynamics of the Newton map,
\ie the region $\disk(0,r(1+\sqrt{2}))\backslash \disk(z_\ast,\rho)$ with $\rho$ as in~\eqref{newtonQuad},
where the method can spend a substantial amount of time.
The key observation that leads to our algorithm is that,
in that region, Newton's method appears to be computing level lines
over and over in a very inefficient way (see Fig.~\ref{fig:idea}).

\begin{figure}[H]
\captionsetup{width=.75\linewidth}
\begin{center}
\includegraphics[width=.85\textwidth]{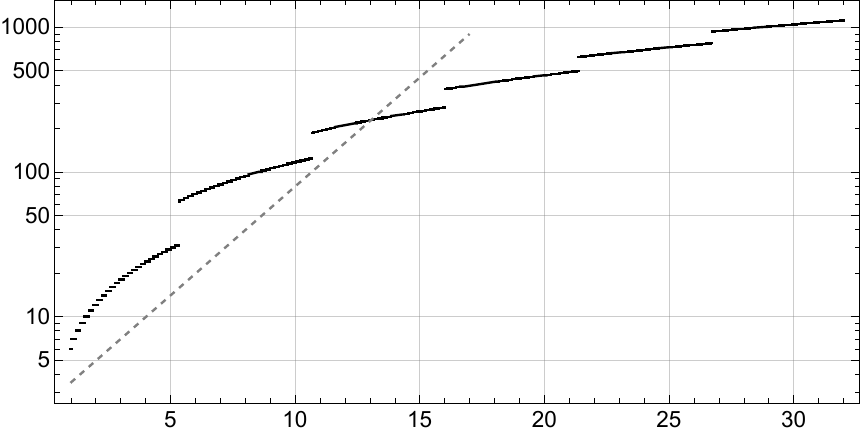}
\caption{\label{fig:gridToPointRatio}\small
Gridpoints-per-root ratio for $\mesh_r(d)$ as a function of $\log_2(d)$.
Comparison with the one-circle grid (dashed line).
}
\end{center}
\end{figure}

\subsubsection{On the scalability of the number of gridpoints per root}

The algorithm of~\cite{HSS2001} is remarkably efficient when $d\leq 2^{20}$ \ie for polynomials of degree up to a few millions.
However, in the giga- and tera-scale, the growth of the gridpoint-per-root ratio, \ie $\alpha\beta/d$,
may add a substantial overhead to the cost of the computation (see Fig.~\ref{fig:gridToPointRatio}).
In other terms, having the certainty to catch all the roots comes at a cost of at least $500$ redundant finds for each root, on average.
Moreover, eliminating those redundant finds adds, in practice, additional search and sorting costs.

\noindent
This objection is, however, not completely fair: in practice, one can start the search on a dyadic thinning of $\mesh_r(d)$
and gradually step up to using the full mesh if roots are still missing after each pass (see \cite[\S9]{HSS2001}).

\begin{remark}
If $P$ is known to only have real roots, then the mesh $\mesh_r(d)$ can be reduced to
$1.3d$ points on a single semi-circle, which gives a constant gridpoints-per-root ratio.
\end{remark}

\subsubsection{On root separation} 

Let $\mathcal{Z} \defequal P^{-1}(0)$. For any~$z \in \mathcal{Z}$, let us introduce
\begin{equation}
\delta_P(z) = \min\left\{ |z-z'| \,;\, z'\in\mathcal{Z} \text{ and } z'\neq z \right\}.
\end{equation}
In order to be able to distinguish the root $z$ from its closest neighbor, 
it is necessary to compute it with a precision $\varepsilon< \delta_P(z)/2$.
The minimal separation between the roots of the polynomial $P$ is
\begin{equation}
\Delta_P = \min\left\{ \delta_P(z) \,;\, z\in \mathcal{Z} \right\},
\end{equation}
For example (see Fig.~\ref{fig:rootSeparation} and~\S\ref{par:HardwareLimit}), for hyperbolic polynomials, one has
$\Delta_{p_n} \propto 4^{-n}$ which means that, lacking any further information on the local fluctuations of $\delta_{p_n}(z)$, one should
request a precision $\varepsilon$ of order~$O(4^{-n})$ for all the roots, which, according to~\eqref{eq:reccomendedGamma}
requires $\gamma =O(n 2^n)$ iterations of the Newton map. The cardinality of the starting mesh
is at least $d=2^{n-1}$ and may skyrocket up to $|\mesh_2(d)| = O(n^2 2^n)$ should
the finest mesh be required. The total number of Newton steps thus ranges from $O(n 2^{2n})$ to $O(n^3 2^{2n})$.
As the evaluation of $p_n$ requires $V_d=O(n)$ arithmetic operations, applying~\cite{HSS2001} to split $p_n$
requires $O(n^\kappa 2^{2n})$ arithmetic operations with $2\leq \kappa\leq 4$.

\medskip
We have no reason to believe that $p_{41}$ would behave especially badly. However, the previous discussion gives a range 
of $10^{27}$ to $10^{31}$ arithmetic operations, which means that, even if we disregard the need to compute at a
precision much higher than 64 bits and forget all the extra operations required to handle the large volume of data,
it would take, optimistically, 31 years to split $p_{41}$ on a exa-scale computer like \textit{Frontier}.
This optimistic estimate is astronomical and marks the splitting of $p_{41}$ as a problem clearly out of reach of the current state of the art.

\subsubsection{On the mid-range dynamics of the Newton map}

A pertinent quantity for the analysis of the Newton map is the number of Newton steps  that are necessary
to bridge the level set $A$ to the level set $B$ (with $A>B$):
\begin{equation}
\gamma_{A,B} (P,z_0)= \min \{ k \,;\, |P(z_k)|\leq B \} -  \max \{ k \,;\,|P(z_k)|\geq A \}.
\end{equation}
For each root $z_\ast$, the dynamics of the Newton map develops in two stages.
The first stage (called mid-range) connects the starting mesh $\mesh_r(d)$ to the edge of $\disk(z_\ast,\rho)$, which requires
$\gamma_1 \simeq \gamma_{A,B}(P,z_0)$ Newton steps
where $A=|P(z_0)|$ and $B=\max\{ |P(z)|  \,;\, z\in \disk(z_\ast,\rho)\}$ with $\rho$ as in~\eqref{newtonQuad}.
In practice,~$\gamma_1 \geq O(d)$ because,
according to \cite[Lemma 3]{HSS2001}, Newton steps satisfy
\begin{equation}
\frac{|z|-r}{d}  < |N_P(z)-z| < \frac{|z|+r}{d}
\end{equation}
 for $|z|\geq r>0$.
A second stage connects the edge of $\disk(z_\ast,\rho)$ to
$\disk(z_\ast,\varepsilon)$ where the desired precision is reached.
As it happens entirely in the quadratic regime~\eqref{newtonQuad},
this phase requires and additional $\gamma_2=O(\log \log \frac{\rho}{\varepsilon})$ refining steps,
or $O(\log \frac{\rho}{\varepsilon})$ for multiple roots.
Overall $\gamma = \gamma_1+\gamma_2$.

\begin{remark}
As discussed in \cite[\S9]{{HSS2001}},
starting instead from a rescaled mesh or radius $\lambda r$ for some $\lambda >1$
would allow one to slightly decrease the upper bound on the cardinality of the mesh if $\lambda$ is chosen close to 1,
but it would dramatically increase $\gamma_1$ by $O(d\log\lambda)$.
\end{remark}

\medskip
Computing $\gamma_{A,B}$ precisely in full generality seems hopeless (see however \cite{BAS2016} and
Theorem~\ref{thm:levLineFlow} below).
However, one may experiment along the real line to get a feeling of the orders of magnitude.
For the hyperbolic polynomials $p_n$,
the Newton dynamics starting from  the point
\[
z_0 = \frac{1}{4}+\theta_n\in \R_+
\quad\text{such that}\quad
p_n\left(\frac{1}{4}+\theta_n\right) = 5
\]
always converges to $z_\ast = 0$. The first stage of the dynamics takes
\[
\gamma_1= \inf \{ k\in\N \,;\, z_k\in\disk(0,1/4) \}
\]
steps because for $k\geq \gamma_1$, the convergence of the sequence $(z_k)_{k\in\N}$ becomes quadratic.
We checked numerically that $\gamma_1\leq 6$ for any $n\leq 38000$.
On the other hand, starting another sequence of Newton iterations from $\widetilde{z_0}=1$ (as would be the case if we start from $\mesh_r(2^{n-1})$), one
would get
\[
\widetilde\gamma_1 =  \inf \{ k\in\N \,;\, p_n(\widetilde z_k)\leq 5 \} \geq O(2^n)
\]
because most Newton steps are of order $2^{-n}$.
This simple observation suggests that, instead of starting from a universal and too far
mesh $\mesh_r(2^{n-1})$,
it may be more appropriate to start from a mesh along a \textit{level line}
\[
|p_n(z)|=L.
\]
This idea is the key to the new algorithm that we describe in the next section.


\section{A new splitting algorithm based on level-lines}\label{par:split}

In this section, we expose our new splitting algorithm at the general level,
and illustrate it on the search for hyperbolic centers and Misiurewicz points of $\M$.
Again, throughout this section $P$ denotes a polynomial in $\C[z]$ of degree $d=\deg P$
whose roots~$\mathcal{Z}=P^{-1}(0)$ are located within the disk~$\disk(0,r)$.
The set of critical points is denoted by~$\crit{P}$.
For simplicity, we will also assume that $P$ is monic. 

\medskip
Brute force is not a viable option when $d\gtrsim 10^{12}$.
For example, one observes that the minimal separation between roots of $p_k$ changes
dramatically with the position in~$\M$ and that only a few roots contribute to the realisation
of the overall minimum (see Fig.~\ref{fig:separationHistogram}). The corresponding Newton bassins
are narrow channels. As illustrated in the previous section,
using a uniform starting grid fine enough to capture all of those
narrow channels is an algorithmic dead-end.

\subsection{General principle}

The solution we propose is that of an adaptative mesh refinement.
Trying to reverse-engineer the dynamics of the Newton map on the fly to guess where to add points
seems mostly out of reach and has only been attempted by~\cite{RSS2017}-\cite{Sch2023} who use a clever
geometric estimator of the deformation of the trajectories for that purpose; however, if a few roots end up
being missed, it is not clear how to refine the computation.

What we need instead is a systematic way to put more grid points around regions of high root density
and spend less time visiting the regions of lower root density.
We want to parametrise a curve encircling the roots of~$P$ such that uniform subdivisions
of its parameter induce the desired adaptative mesh refinement.

\bigskip
The answer is given by the interaction between the two ordinary differential equations that rule Newton's flow and the level lines.
\textit{Newton's flow} is defined by the ODE:
\begin{equation}\label{eq:NewtonFlow}
\zeta'(t) = -\frac{P(\zeta(t))}{P'(\zeta(t))}
\end{equation}
The evolution is iso-angle (\ie $\arg P$ is constant along trajectories) and the modulus is an exponentially decaying Lyapunov function
because any solution satisfies:
\begin{equation}\label{eq:NewtonFlowLyapunov}
P(\zeta(t)) = e^{-(t-t_0)} P(\zeta(t_0)).
\end{equation}
The stationary solutions are the roots.
We will call  \textit{Newton's rays} or \textit{iso-angle rays} the non-stationary solutions of the Newton flow.
The non-critical iso-angles, \ie
\begin{equation}
\arg P(\zeta(t_0))\notin \arg \left[ P(\crit{P})\backslash\{0\} \right]
\end{equation}
stem global solutions, both forwards and backwards.
Forwards, non-stationary maximal trajectories on $[t_0,T^\ast)$
converge to the roots if the trajectory is global \ie $T^\ast=+\infty$, or to critical points if $T^\ast<+\infty$.
Note that, when $T^\ast<+\infty$, the argument of the critical value at the end-point must
match the iso-angle of the trajectory; also, the end-point cannot be a multiple root because~\eqref{eq:NewtonFlowLyapunov}
implies that
\[
\lim_{t\to T^\ast} P(\zeta(t)) = e^{-(T^\ast-t_0)}P(\zeta(t_0))\neq0.
\]
Backwards, maximal trajectories on $(T_\ast,t_0]$ either go to $\infty\in\CC$ if $T_\ast=-\infty$
or to non-root critical points with the same critical iso-angle if $|T_\ast|<\infty$.
Bounded maximal trajectories are possible between two critical points whose values share the same argument.
Note that $0$ is a critical value if and only if $P$ admits multiple roots. In that case, if $\zeta(t)$ converges to a root of multiplicity $\mu$,
L'Hopital's rule ensures that
\[
\zeta'(t) \underset{t\to+\infty}{\sim} - \frac{P^{(\mu-1)}(\zeta(t))}{P^{(\mu)}(\zeta(t))}
\]
so, locally, the profile of Newton's flow is radial and coincides with that of a simple root of~$P^{(\mu-1)}$.
The properties of Newton's flow where used in~\cite{AMV2023a} to produce a short proof of the fundamental Theorem of Algebra.

\bigskip
Level lines $|P(z)|=L$ are also caracterized by an ODE, namely:
\begin{equation}\label{eq:LevelLinesFlow}
\lambda'(t) = i\frac{P(\lambda(t))}{P'(\lambda(t))}\cdotp
\end{equation}
The solutions satisfy
\begin{equation}\label{eq:LevelLinesFlowSol}
P(\lambda(t)) = e^{it} P(\lambda(0)).
\end{equation}
By construction, the level-lines are orthogonal to iso-angles.
Again, stationary solutions are the roots of $P$.
The modulus $|P(\lambda(t))|$ is constant along each level-line;
in particular, the level-lines are bounded. If $|P(\lambda(t))| > \max |P(\crit(P))|$,
there is only one maximal solution, which is a Jordan curve (as the image of a circle by the Riemann map of the outside).
Below the maximal critical value, there are multiple connected components that are either compact
or join two (possibly identical) critical points.
The function $\lambda(t)$ is $2\ell\pi$-periodic where $\ell\leq d=\deg P$ because $P(\lambda(2k\pi))=P(\lambda(0))$ for any $k\in\Z$ and $P$ can
take this value at most $d$ distinct times; once the value loops, the unicity in the Cauchy-Lipschitz theorem for the first-order ODE~\eqref{eq:LevelLinesFlow}
ensures periodicity. For a level line of large enough modulus, $d$ is the smallest period.

\bigskip
\begin{thm}\label{thm:practicalFTA}
Given a monic polynomial $P$ of degree $d$, if $z_0\in\C$ satisfies
\begin{equation}
|P(z_0)| > \max |P(\crit(P))|
\qquad\text{and}\qquad
\arg P(z_0)\notin \arg [ P(\crit(P))\backslash\{0\}],
\end{equation}
then the points $z_k=\lambda(t_0+2k\pi)$ for $0\leq k< d$ obtained by solving~\eqref{eq:LevelLinesFlow} with~$\lambda(t_0)=z_0$
are distinct. Each $z_k$ is associated to a global forward Newton flow~\eqref{eq:NewtonFlow}, denoted by $\zeta_k(t)$.
The points $z_k^\ast = \lim\limits_{t\to+\infty}\zeta_k(t)$ are all the roots of $P$ and this enumeration is consistent with the multiplicities, \ie
\begin{equation}
P(z) = (z-z_0^\ast)(z- z_1^\ast)  \ldots (z-z_d^\ast).
\end{equation}
\end{thm}

\begin{proof}
From the discussion above, we know that the starting points $\lambda(t_0+2k\pi)$
for $0\leq k < d$ such that $\arg P(\lambda(t_0))\notin \arg [ P(\crit(P))\backslash\{0\}]$
have a global forward Newton flow and that each of these flows converges to a root of $P$.
The $d$ starting points are distinct because $|P(\lambda(t_0))| > \max |P(\crit(P))|$ so the level line has
only one connected component.
The question is wether one could be missing a root.
Near each root $z_\ast$, the polynomial~$P$ takes  all possible arguments exactly as many times as the multiplicity because $P(z)\sim \alpha (z-z_\ast)^m$
as $z\to z_\ast$, for $\alpha\in\C^\ast$ and $m\in\N^\ast$. 
If we missed a root $z_\ast$, we consider a nearby point where $\arg P(z) = \arg P(\lambda(t_0))$. The Newton ray from $z$ is
global both forwards and backwards and escapes (backwards) at infinity. Thus, this Newton ray intersects the level line 
of modulus $|P(\lambda(t_0))|$ and, as the angle is invariant, it is one of the $\lambda(t_0+2k\pi)$ points. Therefore, no root could have been missed.
The same argument shows that no multiplicity can be misrepresented, by following back the $m$ Newton rays near $z_\ast$ whose
angle is $\arg P(\lambda(t_0))$.
\end{proof}


If one chooses $d$ distinct values $(t_j)_{0\leq j < d}$ in $[0,2\pi)$, then at least one of them must satisfy $\arg P(\lambda(t_j))\notin \arg [ P(\crit(P))\backslash\{0\}]$ because $\arg P(\lambda(t_j)) =  \arg P(\lambda(t_0)) + t_0-t_j \mod 2\pi$ and there are, at most, $d-1$ distinct critical values.
We have therefore a naive way of choosing $O(d^2)$ starting points $(\lambda(t_j+2k\pi))_{0\leq j,k< d}$ whose Newton trajectories are guarantied to reach all the roots.
Generically however, the number of starting point drops to $O(d)$ as it is very unlikely that more than a few choices will hit critical arguments.
For example, in practice, we were able to split the polynomials $p_n$ and $q_{\ell,n}$ by Newton's method
using only $4d$ starting points of iso-angles $0$, $\pi/2$, $\pi$ and $3\pi/2$.

\bigskip
As illustrated on Figure~\ref{fig:idea},
the next advantage of using iso-angle rays and level lines instead of using unstructured parallel iterations of the Newton map, is that one can
perform most of the descend using a lean mesh, and densify it at a lowest level with no risk of missing iso-angle rays.

\medskip
\begin{figure}[H]
\captionsetup{width=.95\linewidth}
\begin{center}
\includegraphics[width=\textwidth]{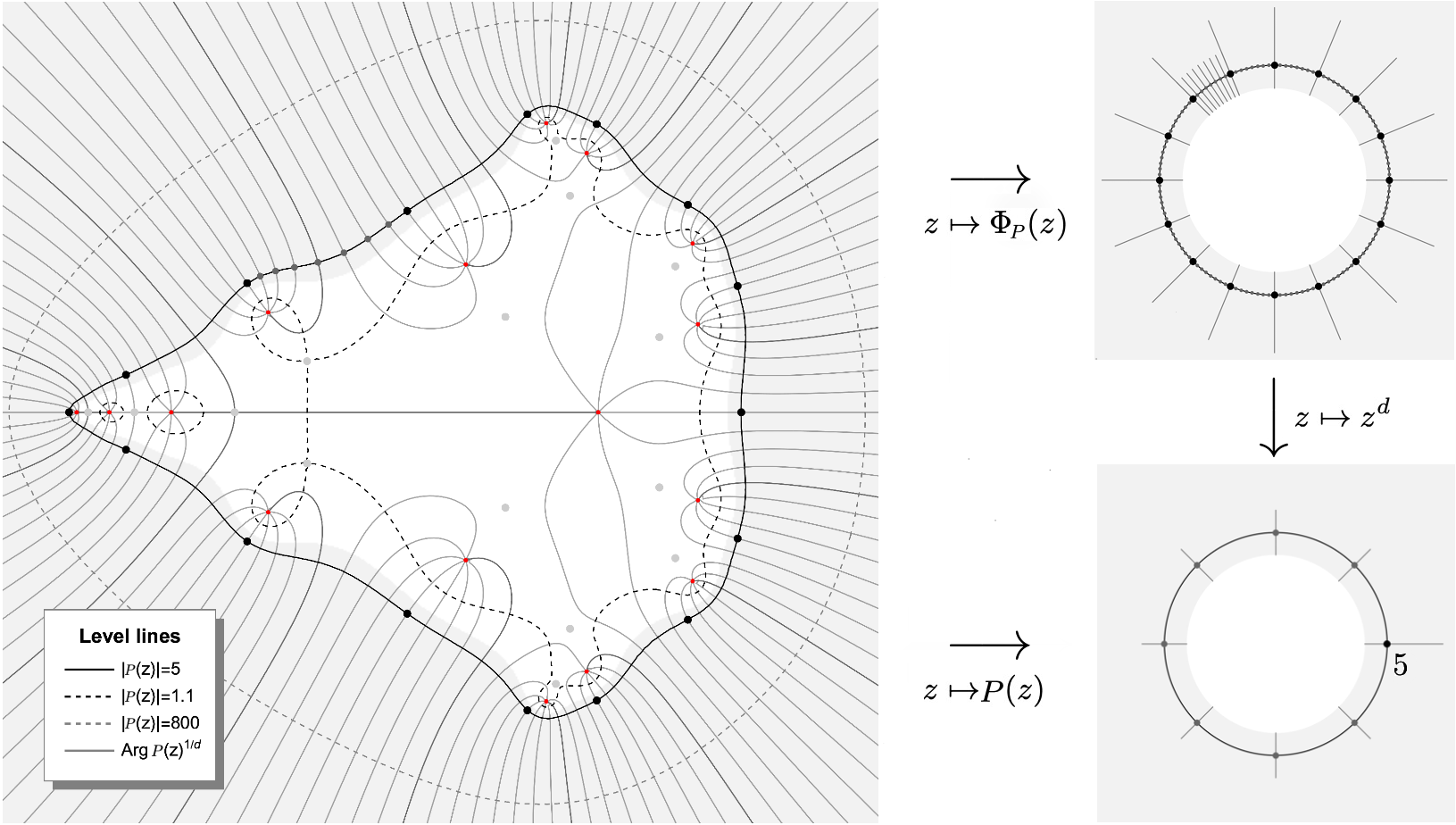}
\caption{\label{fig:MVSplitAlgo}\small
Initial mesh $\lev_{\zeta_0,d}(p_5)$ as black points and part of $\lev_{\zeta_0,8d}(p_5)$ with $d=16$ and $\zeta_0=5$.
Two other level lines are depicted along with Newton's flow (iso-angles).
The map $\Phi_P(z)=P(z)^{1/d}$ maps large level lines to circles (top right)
while $P$ folds the level line $\zeta_0$ onto itself  $d$ times (bottom right). This figure is a commutative diagram.
}
\end{center}
\end{figure}

For $\zeta_0\in\C$ such that $|\zeta_0|> \max |P(\crit(P))|$ and $N\in\N^\ast$, let us consider the mesh
\begin{equation}\label{eq:defDiscreteLev}
\lev_{\zeta_0,N}(P) = \left\{ \lambda(2k d\pi/N) \,;\, 0\leq k < N\right\}
\end{equation}
where $\lambda$ is the  $2d\pi$-periodic maximal solution of~\eqref{eq:LevelLinesFlow} defined by a choice\footnote{To lift the ambiguity regarding that choice, one can denote by $\lev^\bullet_{z_0,N}(P)$ the discrete level
that \textit{starts} at the point $\lambda(0)=z_0\in\C$; then $\lev^\bullet_{z_0,N}(P)=\lev_{\zeta_0,N}(P)$ with
the value $\zeta_0=P(z_0)$.}
of $\lambda(0)$ as one arbitrary root
of~$P(\lambda(0))=\zeta_0$. In the rest of this section, we assume that this choice is made consistently;
it does not interfere with the set itself because $\lambda$ is periodic.
The formula~\eqref{eq:LevelLinesFlowSol} implies that:
\begin{equation}\label{eq:LevelLinesFlowPhaseShift}
P(\lambda(2k d\pi/N)) = e^{2 i k d\pi/N} \zeta_0.
\end{equation}
In particular, the mesh $\lev_{\zeta_0,N}$ is a finite subset of cardinality $N$ of the (smooth)
level curve defined by $|P(z)|=|\zeta_0|$.
It is called a \textit{discrete level set} of order $N$.
As $\lev_{\zeta_0,N}(P) \subset \lev_{\zeta_0,2N}(P)$, each mesh $\lev_{\zeta_0,2^m}(P)$
is a refinement of $\lev_{\zeta_0,2^{m-1}}(P)$. 
More generally, increasing the value of $N$ creates a finer mesh. 
On the value side, $P(\lev_{\zeta_0,N}(P))$ is a uniform mesh (roots of unity of higher order);
on the pre-image side, the density of $\lev_{\zeta_0,N}(P)$
the mesh increases naturally around points of higher root density, as illustrated on Figure~\ref{fig:MVSplitAlgo}.
Note that the previous considerations on the solutions of~\eqref{eq:LevelLinesFlow} imply that
\[
P^{-1}(\zeta_0) = \lev_{\zeta_0,\deg P}(P)
\]
as long as $|\zeta_0|> \max |P(\crit(P))|$.

\medskip
The discrete level line $\lev_{\zeta_0,N}(P)$ is computed recursively using the following result.
One defines the Newton map with target $\zeta$ as the limit, if it exists:
\begin{equation}
L(P ; z_0,\zeta) = \lim_{n\to\infty} N_{P-\zeta}^n (z_0)
\end{equation}
where $N_Q$ is the Newton map~\eqref{eq:Newton} of a polynomial $Q$ and powers denote, as usual, composition.
By construction, $P(L(P ; z_0,\zeta)) = \zeta$.
\begin{prop}\label{prop:computeLevelLine}
Given $z_0$ such that $|P(z_0)| > \max |P(\crit(P))|$ and
an arbitrary integer~$N\in\N^\ast$, there exists $M\in \N^\ast$ large enough given by~\eqref{eq:nextStepV3}
such that the finite sequence
\begin{equation}\label{eq:auxSequence}
z_{k+1} = L\left(P;z_k, P(z_k) e^{\textstyle \frac{2 i \pi d}{MN}}\right)
\end{equation}
satisfies $\lev_{P(z_0),N}=\{z_0,z_M,z_{2M},\ldots,z_{(N-1)M}\}$.
\end{prop}
\begin{remark}
The auxiliary finite sequence $(z_{\ell M+j})_{\ell\in\ii{0}{N-1}; \,j\in\ii{0}{M-1}}$ satisfies
\[
P(z_{\ell M+j}) = P(z_0) e^{2i\pi d \,(\ell M+j)/(MN)}.
\]
When $N=d=\deg P$, the phase shift is simpler. 
The phase of $P(z)$ changes $d$ times along $|P(z)|=\lambda$.
Each time the phase advances by $2\pi$ (one turn), we have a new point of $\lev_{\zeta_0,d}(P)$.
Each discrete turn is completed with $M$ intermediary steps,
as illustrated on Figure~\ref{fig:MVSplitAlgo}. 
The value $M$ is called the \textsl{number of points per turn} for the computation of $\lev_{\zeta_0,d}(P)$.
When $N$ divides $d$, roots of $P(z)=\zeta_0$ are being skipped.
\end{remark}
\proof 
Let $C \defequal \max |P(\crit(P))|$.
Applying the same argument as in \cite{JUNG1985}, one can claim the following:
there is a unique analytic determination of $\Phi_P(z)=P(z)^{1/d}$
that satisfies the constraint $\Phi_P(z)\sim z$ as $|z|\to\infty$.
Moreover, on $\{z\in\CC \,;\, |P(z)|>C\}$, the map $\Phi_P$ is univalent (see Fig.~\ref{fig:MVSplitAlgo}).
Given $\zeta_0=P(z_0)\in\C$ such that $|\zeta_0|> C$ and $\lambda$ the continuous level line from $z_0$,
the first step of the proof consists in choosing $M\in\N^\ast$ large enough to ensure that $z_1 = \lambda(\frac{2\pi d}{MN})$
satisfies
\begin{equation}\label{eq:nextStepV1}
\lim_{n\to\infty} \mathcal{N}^n(z_0)=z_1 \quad\text{where}\quad   \mathcal{N}(z) = z - \frac{P(z)-P(z_1)}{P'(z)}\cdotp
\end{equation}
Up to a translation $\mathcal{N}(z)=\mathcal{N}_\ast(z-z_1)+z_1$ the desired statement is equivalent to:
\[
\lim_{n\to\infty} \mathcal{N}_\ast^n(z_0-z_1)=0 \quad\text{where}\quad   \mathcal{N}_\ast(z) = z - \frac{P(z+z_1)-P(z_1)}{P'(z+z_1)}\cdotp
\]
Let $\zeta_1=P(z_1)$ and the half-line $\mathcal{D}=\R_-\zeta_1$
antipodal to $\zeta_1$. Note that $|\zeta_0|=|\zeta_1|$.
The injectivity of $\Phi_P$ ensures that $P^{-1}(z)$ is well defined, injective and bi-holomorphic
on $\disk(0,C)^c\backslash\mathcal{D}$ and in particular on $\disk(\zeta_1,|\zeta_0|-C)$.
Next, writing $R=|\zeta_0|-C$, we apply Theorem~\ref{thm:target} to the injective map $g:\disk(0,1)\to\C$ defined by
\[
\forall w\in \disk(0,1),\qquad
g(w) = \frac{P'(z_1)}{R} \left( P^{-1}(\zeta_1 + R w) - z_1\right).
\]
By construction, $g(0)=0$ and $g'(0)=1$. Its reciprocal  $f=g^{-1}$ is given by
\[
f(z)=\frac{P\left(z_1+ \textstyle\frac{R z}{P'(z_1)}\right)-P(z_1)}{R}\cdotp
\]
Then $g(\disk(0, \beta_0))$ is contained in the bassin of attraction of $0$ for the Newton map
\[
N_f(z)= \frac{P'(z_1)}{R} \: \mathcal{N}_\ast\left(  \frac{Rz}{P'(z_1)} \right).
\]
The original question~\eqref{eq:nextStepV1} gets an positive answer if $\frac{P'(z_1)}{R}(z_0-z_1)\in g(\disk(0, \beta_0))$, which easily boils down to
\begin{equation}\label{eq:nextStepV2}
|P(z_0)-P(z_1)|<\beta_0(|P(z_0)|-C).
\end{equation}
As $P(z_1)=P(z_0)e^{\frac{2i\pi d}{MN}}$ and $|1-e^{i\theta}|<|\theta|$ it is therefore sufficient to take
\begin{equation}\label{eq:nextStepV3}
M \geq \frac{2\pi d}{N\beta_0 \left(1-\frac{C}{|P(z_0)|}\right)}
\end{equation}
to ensure that $L\left(P;z_0, P(z_0) e^{\textstyle \frac{2 i \pi d}{MN}}\right) = \lambda(\frac{2\pi d}{MN})$.
As the constants are uniform by rotation,
the rest of the statement follows by re-indexing the next points.
\endproof

\begin{remark}
For more precise estimates along the level line, see \cite{MVCoeffs},
where we exploit the fact that $P^{-1}$ is injective not only on the disk $\disk(\zeta_1,|\zeta_0|-C)$
but also on the half-plane tangent to $\disk(0,C)$ that contains $\zeta_1$.
\end{remark}


\bigskip
Once one is able to compute $\lev_{\zeta_0,N}(P)$ efficiently,
then, the splitting algorithm consists in iterating either Newton's flow or the
Newton map from each point of $\lev_{\zeta_0,N}(P)$
until roots are found up to the precision desired. Obviously, this step is highly parallelizable.
Duplicates are eliminated from the list and if some roots are missing, $N$ can be increased.

\medskip
The strategy for splitting $P$ can therefore be formalized as follows.

\medskip
\renewcommand{\algorithmcfname}{Algorithm $\mathcal{S}(P)$~}
\renewcommand{\thealgocf}{}
\begin{algorithm}[H]
\caption{Splitting algorithm for $P(z)$ with simple roots in $\disk(0,1)$}
\label{algo:split}
\DontPrintSemicolon
\SetKwFor{ForEach}{for~each}{do}{endfch}
\SetKwRepeat{Do}{do}{while}
\smallskip
\KwData{$\lambda_{+} \geq \max |P(\crit(P))|$ and $\lambda_{-} \leq \min |P(\crit(P))|$.}
\cod{Compute the discrete level line $\mesh=\lev_{2\lambda_{+}, c_L d}(P)$ \tcc*{$O(d)$}}
\ForEach{$z\in \mesh$}{
  	\cod{ Iterate time-1-flow to $\lambda_-/11$ \tcc*{$(5+\log \frac{\lambda_{+}}{\lambda_{-}}) \times T_\text{flow}$}}
	\cod{ Iterate $N_P$ to the desired precision $\varepsilon$ \tcc*{$O(\log|\log\varepsilon|)$}}
	\cod{ Add root to appropriate data structure $R$ \tcc*{$O(\log^2 d)$}}
}
\end{algorithm}

\bigskip
An upper bound for~$\lambda_+$ is given by Theorem~\ref{thm:rootLoc} applied to $P'$; on the other hand, 
no simple universal a-priori lower bound for $\lambda_-$ is known.

\medskip
We have already explained how to compute the discrete level line.
The constant $c_L$ is chosen in accordance with Proposition~\ref{prop:computeLevelLine}. In practice, $c_L=4$ (see~\S\ref{par:levAlgoMandelSpecifics})

\medskip
In the previous algorithm, Newton's descend is split in two phases.
In the first phase, one  follows the iso-angle rays instead of taking discrete Newton steps.
We call $T_\text{flow}$ the arithmetic cost of following the continuous flow in time 1 and factor it out.
An exact upper-bound of the duration of transit is available for the continuous flow.

\begin{prop}\label{prop:computeLevelLineBis}
The number of iterations of the time~1 Newton flow that connect two level lines is
logarithmic in the ratio of the levels.
Precisely, given a point $z_0$ such that $|P(z_0)|=\lambda_1$ and $\lambda_2<\lambda_1$,
the Newton flow $\zeta(t)$ of $P$ defined by~\eqref{eq:NewtonFlow} with $\zeta(0)=z_0$ satisfies
\[
|P(\zeta(t))|\leq\lambda_2 \qquad\text{if and only if}\qquad t\geq \log \frac{\lambda_1}{\lambda_2}\cdotp
\]
\end{prop}
\proof According to~\eqref{eq:NewtonFlowLyapunov}, each time~1 iteration of Newton's flow divides the modulus
of the value of the polynomial by~$e$.
\endproof

In the second phase, our choice for $\lambda_-$ ensures that $N_P$ converges quadratically.
\begin{prop}\label{prop:quadraticBassinGuarantied}
Assuming $P$ has simple roots, let $r_c=\min |P(\crit(P))|$. The set
\[
\Omega \defequal P^{-1}\left( \disk(0,r_c) \right)
\]
is contained in the bassin of attraction of the roots of Newton's flow~\eqref{eq:NewtonFlow}.

Moreover, the set $P^{-1}(\disk(0,r_c/11))$ in included
in the union of the quadratic bassins of the Newton's method.
If the roots or $P$ belong to $\disk(0,1)$, it takes $O(\log|\log\varepsilon|)$ Newton steps to go 
from any level line $\lambda \le r_c/11$ to the roots, up to the desired precision $\varepsilon$.
\end{prop}
\proof
The set $\Omega$ has $d$ connected components, one for each root of $P$.
Indeed, by construction, $\Omega$ does not contain any critical point so the
forwards Newton's flow~\eqref{eq:NewtonFlow} is global on $\Omega$ and connect continuously any point in $\Omega$ to a root of $P$,
without leaving $\Omega$ because of~\eqref{eq:NewtonFlowLyapunov}. Every component contains at least one root and
if two roots where in the same component, the flow would have to split and $\Omega$ would contain a critical point, which is impossible.
Let us now take $\lambda\leq \lambda_- /11 < \beta_2 r_c$ and $w_0$ such that $|P(w_0)|=\lambda$.
The function $P$ maps each connected component of $\Omega$ to $\disk(0,r_c)$ bi-holomorphically and, in particular,
is injective on each. Choosing a component $\Omega_j$ of $\Omega$ and the corresponding root $z_j$,
we apply Corollary~\ref{cor:targetQuad} to
\[
g(y)= \frac{P'(z_j)}{r_c} \left[ P^{-1}(r_c y) - z_j \right] : \disk(0,1) \to
\tilde{\Omega}_j \defequal \textstyle \frac{P'(z_j)}{r_c}\left( \Omega_j -z_j\right)
\]
with target $\disk(0,\epsilon_j)$, where $\epsilon_j=\frac{|P'(z_j)|}{r_c} \,\varepsilon$.
If $|P'(z_j)|/r_c>\frac{1}{8\varepsilon}$, we already have $w_0\in\disk(0,\epsilon_j)$.
Otherwise, the Newton sequences
\[
w_{k+1} = N_P(w_k)\qquad\text{and}\qquad \zeta_{k+1}=N_{g^{-1}}(\zeta_k)  
\]
are related by
\[
w_k = z_j + \frac{r_c \zeta_k}{|P'(z_j)|}
\]
and for $n\geq3+\log_2\left|\log_2 \frac{|P'(z_j)|}{r_c}\,\varepsilon\right|$, one has 
$\zeta_k\in\disk(0,\epsilon_j)$, \ie $w_k\in\disk(z_j,\varepsilon)$.

\medskip
If one assumes that the roots of $P$ belong to the unit disk, then K\"obe's lemma implies that
$\disk(z_j,\frac{1}{4}r_c/|P'(z_j)|) \subset \Omega_j$. Moreover, $\Omega_j\subset \disk(z_j,\delta_j)$
where $\delta_j$ is the distance to the nearest critical point, which cannot exceed $2$ because $\operatorname{Crit} P$
belongs to the convex envelope of the roots and therefore to $\disk(0,1)$. Consequently, $|P'(z_j)|/r_c > 1/8$.
\endproof

\begin{remark}
Proposition~\ref{prop:quadraticBassinGuarantied}
is optimal for the Newton flow: $\overline{\Omega}$ contains at least one critical point of $P$ so
it is not contained in the bassin of attraction of the roots.
\end{remark}

The appropriate data structure that allows sorting and insertions of $d$ points with total complexity of
at most~$O(d \log^2 d)$ will be presented in Section~\ref{par:num_results}.

\bigskip

The previous statements about the algorithm can be reformulated as follows.
\begin{thm}\label{thm:levLineFlow}
There are universal constants $\rho_2$ and $\beta_2$ given by Corollary~\ref{cor:targetQuad}, such that, for any polynomial
$P$ with simple roots $(z_j)$:
\begin{itemize}
\item The Newton map converges bi-exponentially to $z_i$ on the
disk $\disk(z_j,\rho_2 r_c/|P'(z_j)|)$.
\item Newton's flow~\eqref{eq:NewtonFlow} in time $T=\log \frac{1.1}{\beta_2}+\log\frac{\max |P(\crit(P))|}{\min |P(\crit(P))|}$
maps the level line
\[
\mathcal{L}=\{z\,;\, |P(z)|=1.1\max |P(\crit(P))|\}
\] into the union of the quadratic bassins of attraction
of the roots for the Newton map~$N_P$, \ie the zone where the convergence of the iterates of  $N_P$ towards the roots is bi-exponential.
Note that $\log \frac{1.1}{\beta_2}\leq 2.5$.
\end{itemize}
\end{thm}

\subsection{Specifics for polynomials related to $\M$}\label{par:levAlgoMandelSpecifics}

In full generality, $\mathcal{S}(P)$ and Theorem~\ref{thm:levLineFlow} are not yet ready to rival the algorithm of \cite{HSS2001}. However,
in practice, our strategy behaves extremely well (see Section~\ref{par:num_results}).

\medskip
In our implementation, the outmost level line is $\lambda_0=5$ for $p_n$ and $\lambda_0=100$ for $q_{\ell,n}$.
For the smallest degrees, there is no need for parallelism and it is not necessary
to refine the level curve: one applies simply $\mathcal{S}(P)$ with $4d$ points, \ie $c_L=4$.

\bigskip
To split $p_n$  and $q_{\ell,n}$ for large $n$, a massively parallel implementation is welcome and
can be formalized as follows. Upper bounds on the arithmetic complexity
are given as commentaries. The mention \textit{para} indicates that the operation can be performed in a trivially parallel  fashion (\ie as $J$ independent jobs)
on multiple computing units ; on the other hand, the mention \textit{mono} is a step that should be performed on a single computing unit.

\medskip
One chooses a constant $I_N$ that will cap the number of Newton iterates to $I_N\log d$. In practice, $I_N=1$
for $p_n$ or $I_N=2.6$ for $q_{\ell,n}$.

\medskip
\renewcommand{\algorithmcfname}{Algorithm $\mathcal{S}_{\newparallel}$~}
\begin{algorithm}[H]
\caption{Splitting algorithm for $p_n$ and $q_{\ell,n}$ (parallel \& certified version)}
\label{algo:splitParallel}
\DontPrintSemicolon
\SetKwFor{ForEach}{for~each}{do}{endfch}
\SetKwRepeat{Do}{do}{while}
\smallskip
\Begin(\textbf{compute level lines}){
  \cod{coarse level line $\lev_{\lambda_0,N_0}(P)$ \tcc*{$O(1)$ mono}}
  \cod{refine level line to $\lev_{\lambda_0,d c_L}(P) = \bigcup \mesh_j$ \tcc*{$O(d)$ para}}
} 
\Begin(\textbf{root finding}){
 \ForEach{\textbf{parallel job} $j\in\{1,\ldots,J\}$}{
  \ForEach{$z\in \mesh_j$}{
  	\cod{ iterate $N_P$ up to $\max\{20; I_N \log d\}$ \tcc*{$O(\log d)$ para}}
 } 
 \cod{sort roots $R_j$ in appropriate data structure \tcc*{$O(1)$ para}}
} 
 \cod{merge roots from all jobs $R=\bigcup R_j$ \tcc*{$O(d/J)$ mono}}
} 
 \Begin(\textbf{certification}){
 \ForEach{$z\in R$}{
 \cod{check $P(z)=0$ using disk arithmetic \tcc*{$O(\log d)$ para}}
 }
 } 
\end{algorithm}

\bigskip
Let us go over each step in details and justify the bounds on the arithmetic complexity
claimed above.

\bigskip
The mesh refinement is possible because the $p_n$ and $q_{\ell,n}$ are defined recursively (see below).
Each point $\lev_{\lambda_0,N_0}(P)$ can then be used as a starting point of an independent job so one can choose $J\leq  N_0$.
By construction, there are $d c_L/N_0-1$ points of $\lev_{\lambda_0,d c_L}(P)$ between each pair of consecutive points
of $\lev_{\lambda_0,N_0}(P)$. According to Theorem~\ref{thm:practicalFTA}, it means that the algorithm will produce
$| R_j | \sim d/J$ roots per job.

For large degrees (typically $n\geq 17$), the coarse level set corresponds to $N_0 = 2^{15}$ which means that 16\,385 starting points such that $|P(z)|=\lambda_0$ are computed in
the upper plane $\Im z \geq 0$, with $\arg P(z) = 2k\pi/N_0$ and $k\in\ii{0}{N_0}$.
For the intermediary periods, we found it more practical to group
consecutive jobs together in order to ensure an approximately constant size for $| R_j |$ (and thus consistent file sizes across the project).

\medskip
Let us now explain how $\lev_{\lambda_0,c_L d}(P)$ can be computed efficiently using a parallel algorithm.
As the cardinality of this mesh is $c_L d$, the huge volume of data prohibits to compute the mesh
ahead of time or to store it on disk when $d$ reaches the tera-scale. The pertinent subsection of
the mesh has to be generated on the fly by each computing process.

\medskip
For the Mandelbrot set, one uses the fact that the polynomials $(p_k)_{k\geq3}$  form a nested family whose
level lines encircle~$\M$ tightly (see Fig.~\ref{fig:meshRefine}).
One computes an initial mesh $\lev_{\lambda_0,2^{15}}(p_{16})$.
Next, one observes that for any $z_j \in \lev_{\lambda_0,2^{15}}(p_{16}) = p_{16}^{-1}(\lambda_0)$, one has
\[
p_{17}(z_j)=p_{16}(z_{j})^2+z_{j} =  \lambda_0^2+z_{j}.
\]
Proceeding as in Proposition~\ref{prop:computeLevelLine},
we can find $z_j'$ close to $z_j$ such that $p_{17}(z_j')=\lambda_0^2$.
This means that $z_j'\in \lev_{\lambda_0^2, 2^{15}}(p_{16})$ with the correct index.
Iterating this method, one can then use $z_j'$ as starting points to solve iteratively $p_{17}(z)=r_k$ where
the radius $r_k$ is gradually reduced from $\lambda_0^2$ down to $\lambda_0 $.
In the end, one obtains $\lev_{\lambda_0, 2^{15}}(p_{17})$.
In a similar fashion, one can construct the mesh $\lev_{\lambda_0, 2^{15}}(p_{k+1})$ from $\lev_{\lambda_0, 2^{15}}(p_{k})$
in negligible time. This method is the key to the parallelization as each point of $\lev_{\lambda_0, 2^{15}}(p_{k})$ is the seed
of a parallel task.

\begin{figure}[H]
\captionsetup{width=.85\linewidth}
\begin{center}
\includegraphics[width=.9\textwidth]{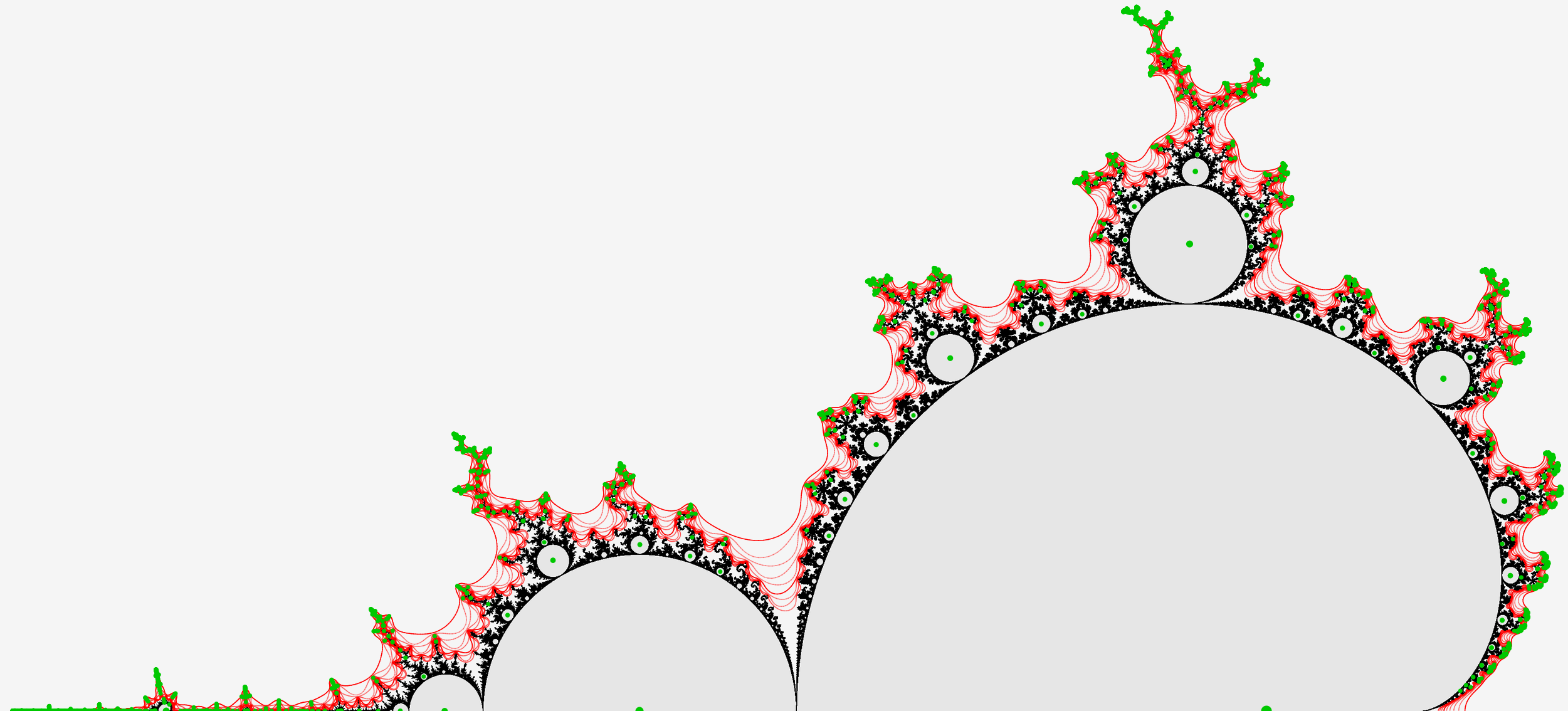}
\caption{\label{fig:meshRefine}\small
Level sets of $p_k$ for $10\leq k\leq 18$ (red) and hyperbolic centers $\hyp(n)$ for $n\leq 12$ (green).
The level lines $\lev_{\lambda,N}$ tend to self-refine around regions where $\arg P(z)$ cycles most,
which reflects a higher intrication of the bassins of attraction for the Newton flow and the Newton map.
}
\end{center}
\end{figure}

In practice, the highest critical value (in module) of $p_n$ is reached along the real axis,
at the left-most tip of the Mandelbrot set and is asymptotically equal to $2$ (see also~\cite{JUNG1985}).

\bigskip
The number of Newton steps necessary to jump  from the starting level line $\lambda_0>2$
to a neighbourhood $\disk(z,\varepsilon)$ of a root $z$ is almost constant: on average,
no more then 12 steps where necessary to split $p_{41}$, most of which can even be performed in a lower precision
(see Table~\ref{table:newtonSteps} in Section~\ref{par:num_results}).
 One can show that the 1st Newton step is of order the distance to the Mandelbrot set. What happens later is less clear.

\subsection{About multiple roots}

As the level line is chosen above the highest modulus of critical values, the algorithm is not affected by multiple roots. A root of multiplicity $m$ will be reached along $m$ distinct iso-angles.
In practice, our benchmarks (see  Section~\ref{par:BenchNewtonSteps})
confirm that our algorithm is indeed robust to high multiplicities.


\section{The certification process}
\label{par:certif}

As illustrated in \cite{RSS2017}, there are various techniques to convince oneself that one has found all roots.
For example, one can use Viete's identities to check that a few sums of the powers of the roots
are in concordance with the first coefficients of the polynomial.
Occasionally, this technique can provide some help in finding a few missing roots.
However, for very large degrees, the cost of this checkup may become prohibitive for a result that, in the worst case,
could be no more than a lucky coincidence.

\medskip
Each of our results comes instead with a numerical proof, based on disk arithmetic.
The complexity of checking that $p_n(z)=0$ using our method is $O(\log d)$.

\subsection{How to prove the localization of a root}

Let us consider $f \in \Hol(U)$ and assume that one has found a point~$z_0\in U$ such that $f(z_0) \in \disk(0,\varepsilon)$
for some $\varepsilon>0$ and $f'(z_0) \in \disk(\lambda,\eta)$ with~$|\lambda|>\eta>0$.
If $\varepsilon/(|\lambda|-\eta)$ is small enough, one can expect
the existence of an exact root of $f$ in the immediate vicinity of~$z_0$.
This idea is at the heart of Newton's method. We can combine it with the spirit of Rouch\'{e}'s theorem~\ref{thm:rouche} to produce a quantifiable statement.

\begin{thm}\label{thm:proofRoots}
Given a holomorphic function $f\in\Hol(U)$, a disk $B=\disk(z_0,R)$ such that $\bar{B}\subset U$
and $B'$ a second disk  such that $f'(B)\subset B'$, we assume that
\begin{equation}\label{locRootAssumpt}
R \dist(0,B') > |f(z_0)|.
\end{equation}
Then there exists a unique point $z_\ast \in B$ such that $f(z_\ast)=0$.
\end{thm}

\begin{remark}
Note that, $z_\ast$ is necessarily a simple root of $f$.
In practice, disk arithmetic (see \S\ref{par:controlNum}) provides an upper bound of $|f(z_0)|$
and, for a given $R>0$, also  of $|f'(z)-z_1|$ when $z\in B$, where $z_1$ is a numerical
approximation of $f'(z_0)$. One can then construct $B'$ and check if~\eqref{locRootAssumpt}
is indeed satisfied.
\end{remark}

\proof
Let us introduce the arc $\gamma(t) = z_0 + R e^{2i\pi t}$ for $t\in [0,1]$.
The fundamental theorem of calculus and the convexity of $B'$ imply:
\[
f(\gamma(t)) = f(z_0) + R e^{2i\pi t} \int_0^1 f'(z_0+s R e^{2i\pi t}) ds \in f(z_0) + R e^{2i\pi t} \cdot B'.
\]
Assumption \eqref{locRootAssumpt} thus implies that $f\circ \gamma$ is valued in an annulus centered at $f(z_0)$ that encircles zero. Its winding number with respect to zero is thus the same than that with respect to $f(z_0)$.
Moreover,~$f\circ\gamma$ is homotopic, within that annulus,
to the path $t\mapsto f(z_0)+R\lambda e^{2i\pi t}$ where $\lambda$ is the center of $B'$ so the
winding number is 1. Because of~\eqref{locRootAssumpt}, $0\notin f'(B)$ so all possible roots within $B$ are simple
and
\[
|f^{-1}(0) \cap B | = \frac{1}{2i\pi}\int_{\gamma} \frac{f'(z)}{f(z)} dz = 
\frac{1}{2i\pi}\int_{f\circ\gamma} \frac{d\zeta}{\zeta} = \frac{1}{2i\pi}\int_{f\circ\gamma} \frac{d\zeta}{\zeta-f(z_0)} = 1
\]
\ie $f$ admits exactly one root in $B$.
\endproof

\subsection{How to prove the convergence of Newton refinements}

The practical limitation of the classical estimates presented in Section~\ref{par:Newton}
is that the radius on which Newton's method behaves bi-exponentially or exponentially
is not, usually, an explicit one.
The following result adresses this issue, at least for simple roots.
For further results, see~\cite[Chap.~6]{HEN1974}.

\begin{thm}\label{thm:proofContract}
Let us consider a simple root $z_\ast$ of $P\in\C[z]$ and $z_0 \in \disk(z_\ast,\eta)$.
For $\varepsilon>3\eta$, if there exists a disk $B'$ satisfying
\begin{equation}\label{locNewtonAssumpt}
P'(\disk(z_0,\varepsilon)) \subset B' \qquad \text{with}\qquad \disk(B',0) > 2\operatorname{diam}(B'),
\end{equation}
then one has $N_P(\disk(z_0,\varepsilon)) \subsetneq \disk(z_0,\varepsilon)$ and thus $\disk(z_0,\varepsilon)$
is contained in the attraction bassin $\attract(z_\ast)=
\{ z\in\C \,;\, \lim\limits_{n\to\infty} N_P^n(z) = z_\ast\}$.
\end{thm}

\begin{remark}
The typical application case of Theorem~\ref{thm:proofContract} concerns an approximate root $z_0$
that is known, thanks to Theorem~\ref{thm:proofRoots}, to be in $\disk(z_\ast,\varepsilon_1)$.
To save resources, we want to publish $z_0'$, which is an approximate value of $z_0$ with a reduced precision $\varepsilon_2\gg\varepsilon_1$.
Theorem~\ref{thm:proofContract} gives a condition that garanties that we can do so without losing any important information:
an approximation of $z_\ast$ to an arbitrary high precision can be retrieved
from the sole knowledge of $z_0'$ by applying iterates of $N_P$
The assumption~\eqref{locNewtonAssumpt} can be checked using disk arithmetic (see \S\ref{par:controlNum}).
Note that, in the proof,  the loss of $\frac{\varepsilon-3\eta}{2}$ on the radius suggests that the Lipschitz constant
of $N_P$ is approximately Lipschitz of order $1-\frac{3\eta}{2\varepsilon}$,
which hints that the next iterations of $N_P$ will converge at least exponentially.
\end{remark}

\proof
Let us consider $B=\disk(z_0,\varepsilon)$ and $P'(B)\subset B'$
where $B'=\disk(d,\varepsilon')$ is a disk such that $|d|>5\varepsilon'$.
In particular $B'$ does not contain zero.
Given $z_0+h \in \disk(z_0,\varepsilon)$, one has:
\begin{align*}
N_P(z_0+h) &= z_0+h - \frac{P(z_0+h)-P(z_0)+P(z_0)-P(z_\ast)}{P'(z_0+h)}\\
&= z_0 + h\left(\int_0^1 1-\frac{P'(z_0+t h)}{P'(z_0+h)} dt\right) + (z_0-z^\ast) \int_0^1 \frac{P'((1-t)z_\ast + t z_0)}{P'(z_0+h)} dt
\end{align*}
The ratio of two values $d+\vartheta_1$ and $d+\vartheta_2$ in $B'$ satisfies:
\[
\left| 1 -\frac{d+\vartheta_1}{d+\vartheta_2} \right| =
\frac{|\vartheta_2-\vartheta_1|}{|d+\vartheta_2|} \leq \frac{2\varepsilon'}{|d|-\varepsilon'}<\frac{1}{2}
\quad\text{and}\quad
\left| \frac{d+\vartheta_1}{d+\vartheta_2} \right| \leq \frac{|d|+\varepsilon'}{|d|-\varepsilon'}<\frac{3}{2}.
\]
One thus has $|N_P(z_0+h)-z_0| < \frac{1}{2}|h|+\frac{3}{2}|z_0-z_\ast|$ \ie
\[
N_P(z_0+h) \in \disk\left(z_0,\frac{\varepsilon}{2}+\frac{3\eta}{2}\right) 
= \disk\left(z_0,\varepsilon-\frac{\varepsilon-3\eta}{2}\right)
\subsetneq \disk(z_0,\varepsilon)
\]
provided $\varepsilon>3\eta$. 
As $\disk(z_0,\varepsilon)$ is forward invariant under $N_P$, this disk is contained in its Fatou set. 
The iterates of $N_P$ on $\disk(z_0,\varepsilon)$ thus converge to $z_\ast$.
\endproof

\subsection{Results on how to control numerical errors}\label{par:controlNum}

Lastly, we need a theoretical background for handling all the errors that occur within the numerical computations that are
necessary to prove that the Theorems~\ref{thm:proofRoots} and~\ref{thm:proofContract} can indeed be applied.
It is often (wrongly) believed that integer arithmetic is the only one apt for formal verification.
We intent to show here that proofs can also be carried in floating point arithmetic, even with complex numbers.

\medskip
The certification of a computation is only possible if the implementation choices respect a universal norm,
as for example the one defined by the IEEE standards~\cite{IEEE}.
We use the library MPFR~\citelib{MPFR} because it offers a reliable implementation
of arbitrary precision that has been extensively tested. One of its main feature is the guaranty of proper
handling of roundings, which is essential for the certification process described below.
However, any other compliant implementation could equally be used as a foundation.

Working with complex numbers requires special care because, contrary to the real case,
multiplications in $\C$ are not an elementary operation. In what follows, we will restrict our attention to
the fields operations: $+$, $\times$.

\subsubsection{Finite precision arithmetic}
An ideal model of arithmetic with a finite precision $N\in\N^\ast$ consists in considering
the following discrete subset of real numbers
\begin{equation}\label{eq:fprecDefRhat}
\hat\fprec_N = \{0\} \cup \bigcup\limits_{e\in\Z} 2^e \fprecRef_{N},
\end{equation}
where
\begin{equation}\label{eq:fprecRefZ}
\fprecRef_{N} = \left\{
\pm\Big(2^{-1} + \sum_{j=2}^N b_j 2^{-j} \Big)
\text{ with } b_2,\ldots,b_N \in \{ 0,1 \}
\right\} = \pm2^{-N} \llbracket 2^{N-1}, 2^N-1\rrbracket.
\end{equation}
When $x\in2^e\fprecRef_{N}$, one defines $e=e(x)$ as the \textit{exponent}
of $x$.
The $(b_j)_{1,\ldots,N}$ are called the \textit{bits} of $x$,
with the convention that $b_1=1$.
Let us point out that if $N'\geq N$ is a larger precision, then $\hat\fprec_N \subset \hat\fprec_{N'}$.

Practical implementations of finite precision arithmetic are restricted
by physical contingencies; one then considers instead the  finite set:
\begin{equation}\label{eq:fprecDefR}
\fprec_N = \{\pm0\} \cup \bigcup_{e=e_{\min}}^{e_{\max}} 2^e \fprecRef_{N}
\cup \{\pm \infty, \operatorname{NaN} \},
\end{equation}
where $-e_{\min}$ and $e_{\max}$ are (configurable but fixed) large positive integers.
The additional symbols allow one to handle numerical
exceptions without producing errors ($\operatorname{NaN}$ stands for \emph{not a number}). 
For efficient processing by the hardware,
bits are packed in groups of 8 called a \textit{byte},
and packs of bytes (typically 4 or~8) are called \textit{limbs}.

\begin{figure}[H]
\captionsetup{width=.95\linewidth}
\begin{center}
\includegraphics[width=.9\textwidth]{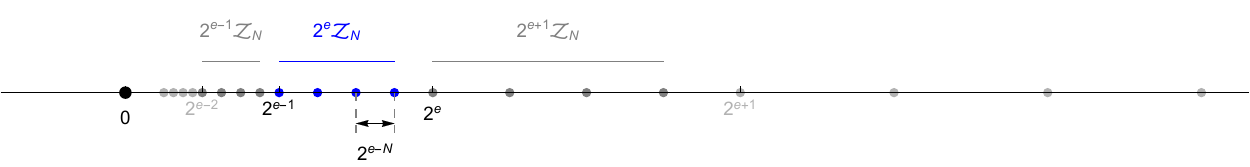}\\
\caption{\label{fig:nbLine}\small
Representation of $2^e\fprecRef_{4}\cap\R_+$ for five consecutive values of $e$.
Note how the gap between adjacent vertices varies with the exponent~$e$,
and the special role of zero as an accumulation point of $\hat\fprec_N$.
}
\end{center}
\end{figure}

\medskip
The arithmetic operations $+_N$ and $\times_N$ are naturally defined on $\hat\fprec_N$ by
\[
\forall x,y \in\hat\fprec_N,\qquad
x +_N y = \hat\round_N (x+y)
\quad\text{and}\quad 
x \times_N y = \hat\round_N (x \times y),
\]
where $\hat\round_N:\R\to\hat\fprec_N$ is the rounding operator that rounds to the nearest lattice point.
Let us also define the operators $\hat\round_N^\pm$ that round respectively always up and always down,
and the practical ones $\round_N$, $\round_N^\pm$, that are valued in~$\fprec_N$. By a convention, called \textit{flush to zero},
which  is not an IEEE standard but is the choice made in the MPFR library, $\round_N(x)=0$ if $|x|< 2^{e_{\min}}$
even though $0$ is not the nearest lattice point if $|x|>2^{e_{\min}-1}$.
Similarly, $|\round_N(x)|=\infty$ if $|x|>2^{e_{\max}}(1-2^{-N})$.

\medskip
For a given number $x\in\hat\fprec_N\backslash\{0\}$, the gap that separates it from its farthest immediate neighbors is
\begin{equation}\label{eq:defULP}
\widehat\ulp(x) = 2^{e(x) - N}.
\end{equation}
By convention, $\widehat\ulp(0)=0$.
The name of this operator is \textit{unit on last position} because it reflects the metric effect of a change of one unit on the least significant bit $b_N$.
If $x\in\fprec_N$ is a regular number  \ie if~$e_{\min}~\leq~ e(x)~\leq~ e_{\max}$, one sets $\ulp(x) = \widehat\ulp(x)$.
By convention, $\ulp(\pm 0) = 2^{1+e_{min}}$ and $\ulp(\pm \infty)~=~\infty$, which ensures the  following statement.
Note that the usual implementation choice is $e(0) = e_{\min}-1$.

\begin{prop}\label{prop:rounding}
For $ x\in \R$, one has
\[
\left| \hat\round_N(x) - x\right| \leq  {\textstyle\frac{1}{2}} \widehat\ulp(\hat\round_N(x))
\qquad\text{and}\qquad
\left| \round_N(x) - x\right| \leq {\textstyle\frac{1}{2}} \ulp(\round_N(x)).
\]
\end{prop}

\begin{remark}
IEEE standards require gradual underflow instead of flush to zero. This means that $\fprec_N$ should be complemented with
a set of denormalized numbers
\[
\{ k 2^{e_{min}-N} \,;\, k\in\Z, \enspace |k|<2^{N-1}\},\]
called subnormals, which ensures that
the lattice $\fprec_N$ is regular near zero.
In that case, $\ulp(x)=2^{e_{min}-N}$ for all those additional points, including $x=0$. With this alternate convention,
Proposition~\ref{prop:rounding} still holds.
The MPFR library offers the possibility of emulating the norm, but it is not the default behavior.
\end{remark}

\subsubsection{Interval arithmetic}
Interval arithmetic is a standard topic in numerical analysis \cite{MATHEMAGIX}, \cite{ROKNE2001} whenever dependable results are critical.
Reliable and fast libraries exist, like MPFI \cite{MPFI} or Arb \citelib{ARB}.

\medskip
Given a continuous numeric function $f$,
the goal of interval arithmetic is to bound, as accurately as possible,
the set to which $f(x)$ belongs when the prior knowledge on $x$ is limited to a set of inequalities, \eg $a\leq x\leq b$. 
The main challenge is to take dependency into account: for example, $f(x)=x^2+x$ maps $[-1,1]$ to $[-\frac{1}{4},2]$ while
$g(x,y)=xy+x$ maps $[-1,1]^2$ to $[-2,2]$.

\medskip
The following statement is an immediate but essential consequence of Proposition~\ref{prop:rounding}.
\begin{prop}\label{prop:intervalArithmetic}
For $x_a,x_b\in\fprec_N$ and $r_a,r_b\in\fprec_M$, one has:
\begin{gather*}
\interval(x_a,r_a) + \interval(x_b,r_b) \subset \interval\left(x_a+_N x_b \,,\,
\round_M^+ \big( r_a + r_b +  {\textstyle\frac{1}{2}}\ulp(x_a +_N x_b)\big)
\right),
\\[1ex]
\interval(x_a,r_a)\times \interval(x_b,r_b) \subset \interval\left(x_a \times_N x_b \,,\,
\round_M^+ \big( r_a  r_b  + r_a|x_b|+ r_b|x_a| +  {\textstyle\frac{1}{2}}\ulp(x_a \times_N x_b)\big) \right)
\end{gather*}
where $\interval(x,r) = (x-r,x+r)$ denotes open intervals along the real line.
\end{prop}\noindent%
In practice, the new radius is computed using the $\round_M^+$ operator for each intermediary computation
to ensure that one gets a certifiable upper bound.

\subsubsection{Naive rectangle arithmetic}
For our purpose in complex dynamics, a naive use
of a tensorized interval arithmetic faces the following shortcoming (see Fig.~\ref{fig:compIntervalDisk}):
if~$a$ is in~$z+[-r,r]+ i [-r,+r]$ with $|z|=1$, then $a^2$ belongs to $z^2+(2r |z|_1+r^2) \, [-1,1] + i (2r |z|_1+2r^2) \, [-1,1]$,
with $|z|_1=|\Re z|+|\Im z|$,
which means that, when $\operatorname{Arg} z\simeq \pi/4$, the size of the uncertainty box is roughly multiplied by $2\sqrt{2}$
instead of the factor 2 imposed by the derivative.
If this happens at many iterations of the square function, the resulting precision loss is catastrophic. 
The average value of $\log |e^{i \theta}|_1$ on the unit circle is $\beta = 0.2365$.
By the Birkhoff ergodic theorem, for a typical starting point $a$ with $|a| = 1$, computing $n$ iterates of the map $z \mapsto z^2$
starting at~$a$ induces a cumulative multiplicative loss of precision of about $e^{\beta n}$.
For example, for $n=200$, a loss of precision of order of $10^{20}$ should be expected.

\subsubsection{Disk arithmetic}
The idea of disk arithmetic consists in studying what optimal outcome can be deduced from the prior knowledge
that $a\in \disk(z,r)$. While the idea is not new \cite{COMPLEXDISKARITH}, it has remained, up to now, fairly uncommon.

\medskip
The practical gain of substituting disks to squares is illustrated on Figure~\ref{fig:compIntervalDisk}. If $a\in \disk(z,r)$,
then~$a^2$ belongs to $\disk(z^2,2r|z| + r^2)$. If $r$ is small enough to ensure that $r^2\ll 2r|z|$ but without
being necessarily tiny,  the first observation is that each iteration of the square function multiplies the radius
of the uncertainty disk by roughly a factor~$2|z|$ instead of $2|z|_1$;
one thus gains a casual $\sqrt{2}$ factor over the naive use of interval arithmetic.

The second observation is more profound and pertains to small values of $r$:
for any conformal map~$f$, the image of a small disk is almost a disk. 
Naturally, a nearly circular shape can be enclosed within a disk of a barely larger diameter.
In practice, $r$ is infinitesimally small and the derivative of the map on the disk of radius $r$ is almost constant. In comparison to the unavoidable action of the differential~$f'(z)dz$ at the center point,
the additive loss on the bounding radius of $f(\disk(z,r))$ is then of order $r^2$ and the multiplicative loss is thus about $1+r/|f'(z)|$.
For example, if we perform $k=10^5$ iterations with bounded derivatives $|f'(z)|\geq1$ along the orbit, starting from a radius $r=10^{-20}$,
then the cumulative multiplicative error of the disk arithmetic method is about $(1+r)^k \simeq 1 + kr = 1 + 10^{-15} \simeq 1$.

\medskip
Using disk arithmetic is the gateway to getting bounds that are  both proven and almost optimal, even after a high number of iterates.
It is a first essential step towards computational proofs in complex dynamics.

\begin{figure}[H]
\captionsetup{width=.93\linewidth}
\begin{center}
\includegraphics[width=.4\textwidth]{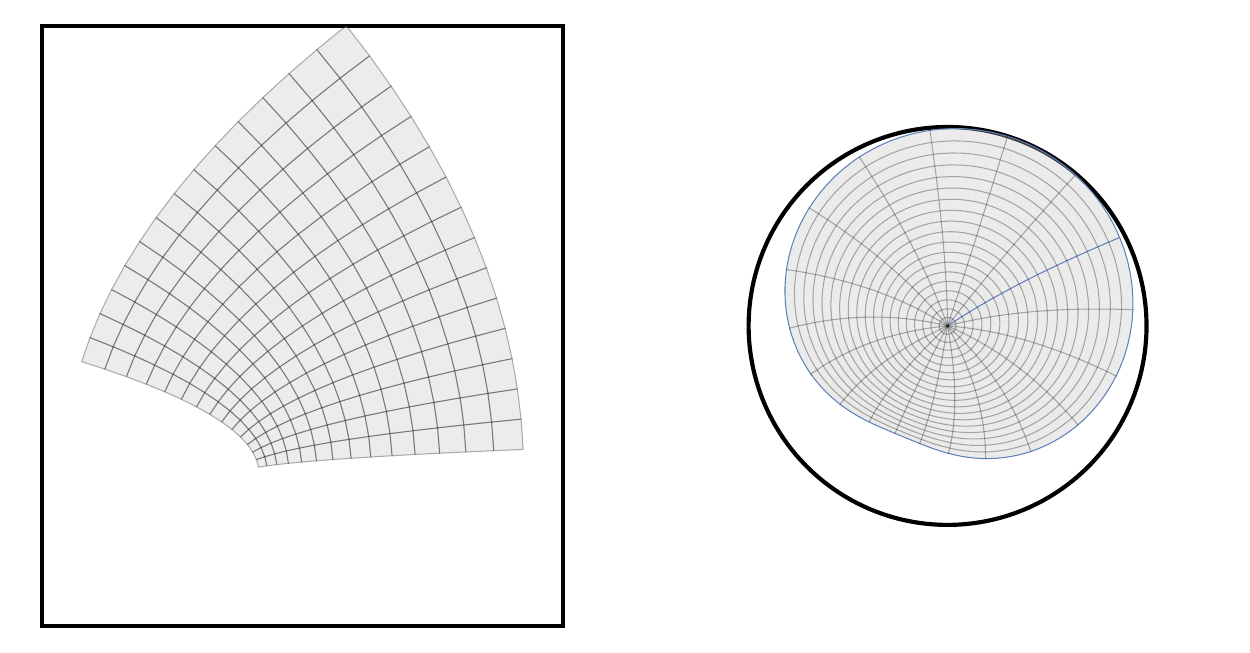} \quad\rule{.1pt}{100pt}\qquad
\includegraphics[width=.4\textwidth]{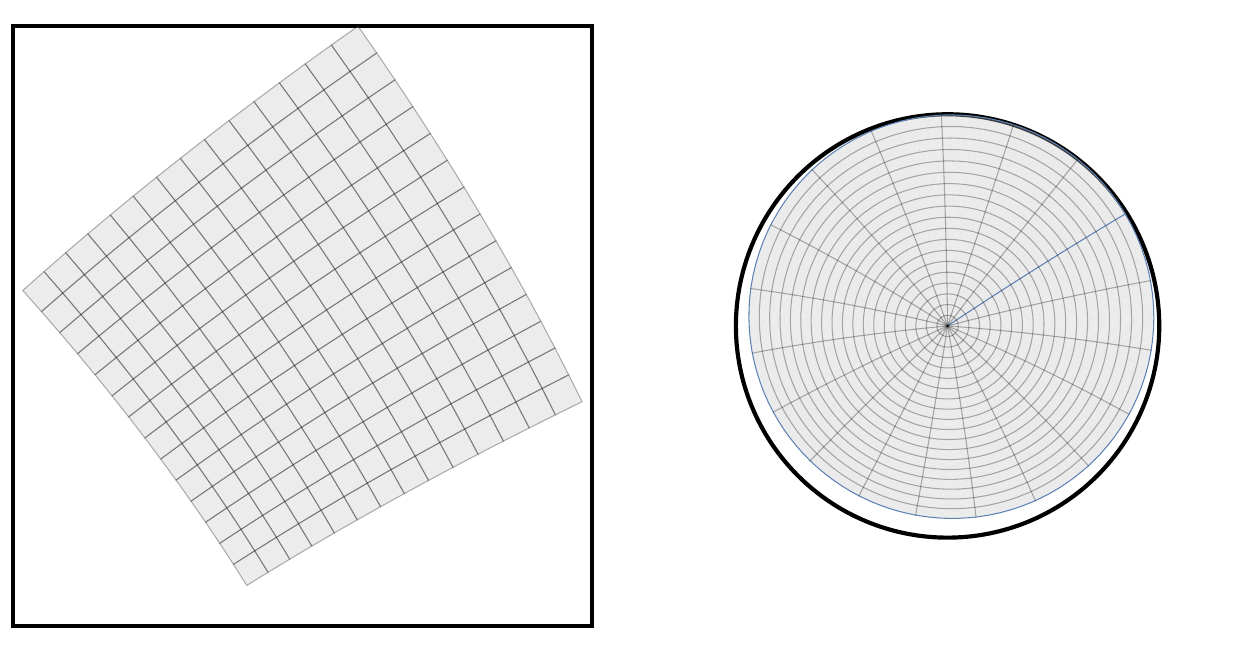}\\
\caption{\label{fig:compIntervalDisk}\small
Comparison between the image by $f(z)=z^2$ of a square centered at $z=e^{3i\pi/16}$ of side $2r$,
and that of a disks of center $z$ and diameter $2r$ for $r=0.5$ (left) and $r=0.1$ (right).
The images of the square and of the disk are drawn at the same scale, and compared respectively
to an enclosing box or disk centered around $z^2$, which is of optimal size, provided that uniformity with
respect to $\operatorname{Arg} z$ is required.
}
\end{center}
\end{figure}

\medskip
For $z_a=x_a+iy_a$, $z_b=x_b+iy_b \in\hat\fprec_{N} + i \hat\fprec_{N}$, the sum $z_a +_N z_b$ is naturally defined in components
\[
z_a+_N z_b = (x_a +_N x_b) + i (y_a +_N y_b).
\]
The radical difference between interval and disk arithmetic is that the computation of the product of complex numbers requires
four exact multiplications on the real line, followed by two additions. Therefore, the product requires an intermediary precision $N'\geq N$:
\[
z_a \times_{N,N'} z_b = (x_a  \times_{N'} x_b) -_N  (y_a \times_{N'}  y_b)
+ i \left[ (x_a  \times_{N'}  y_b) +_N (y_a  \times_{N'}  x_b) \right].
\]
Intermediary products are guarantied exact only if $N'\geq 2N$.
Let us extended the definition of $\ulp$ to complex finite precision numbers by
\[
\ulp( x + i y) = \sqrt{ \ulp(x)^2 +  \ulp(y)^2}.
\]
The following statement generalises Proposition~\ref{prop:intervalArithmetic} and is at the heart of the implementation of our library~\citelib{MLib}.
\begin{thm}\label{thm:diskArithmetic}
For centers $z_a,z_b \in \fprec_{N} + i \fprec_{N}$, radii $r_a,r_b\in \fprec_{M}$, one has:
\[
\disk(z_a,r_a) + \disk(z_b,r_b) \subset \disk\left(z_a +_N z_b \,,\,
\round_M^+\big(  r_a + r_b + {\textstyle\frac{1}{2}}\ulp(z_a+_N z_b) \big) 
\right).
\]
For any intermediary precision $N'\geq N$, the product of exact centers satisfies:
\[
z_a z_b \in \disk\left( z_a \times_{N,N'} z_b, R_\ast \right),
\]
where
\[
R_\ast = \round_M^+\Big( {\textstyle\frac{1}{2}} \ulp(z_a \times_{N,N'} z_b) +
{\textstyle\frac{1}{2}}\sum_{\substack{u=x_a,y_a\\ v= x_b,y_p}} \ulp(u \times_{N'} v)
\Big).
\]
The product of disks satisfies:
\[
\disk(z_a,r_a) \times \disk(z_b,r_b) \subset \disk\left(
z_a \times_{N,N'} z_b \,,\,
\round_M^+\big(  r_ar_b + r_a |z_b| + r_b |z_a| + R_\ast\big)
\right).
\]
Finally,  if $x \in \fprec_N$ and $r\in \fprec_M$, 
 the scaling transform of the disk satisfies:
\[
\interval(x,r) \times \disk(z_a,r_a)\subset \disk\left(
x \times_N z_a \,,\,
\round_M^+\big(  (|x|+r) r_a  + r |z_a| + {\textstyle\frac{1}{2}}\ulp(x \times_N z_a) \big)
\right).
\]
\end{thm}

\begin{figure}[H]
\captionsetup{width=.8\linewidth}
\begin{minipage}[t]{.49\textwidth}\vspace{0pt}%
\begin{center}
\includegraphics[height=.65\textwidth]{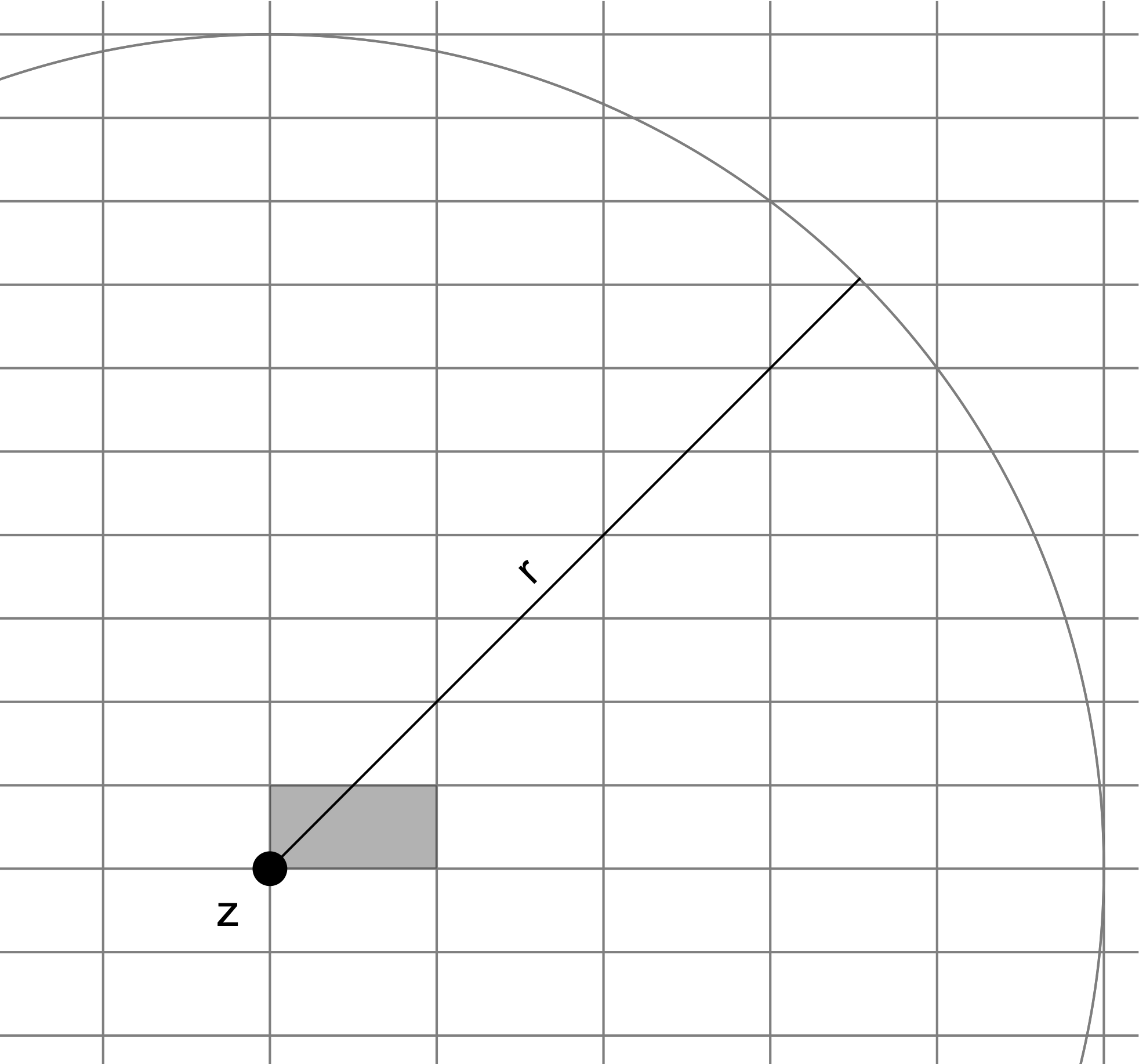}
\end{center}
\end{minipage}
\begin{minipage}[t]{.49\textwidth}\vspace{0pt}%
\begin{center}
\includegraphics[height=.65\textwidth]{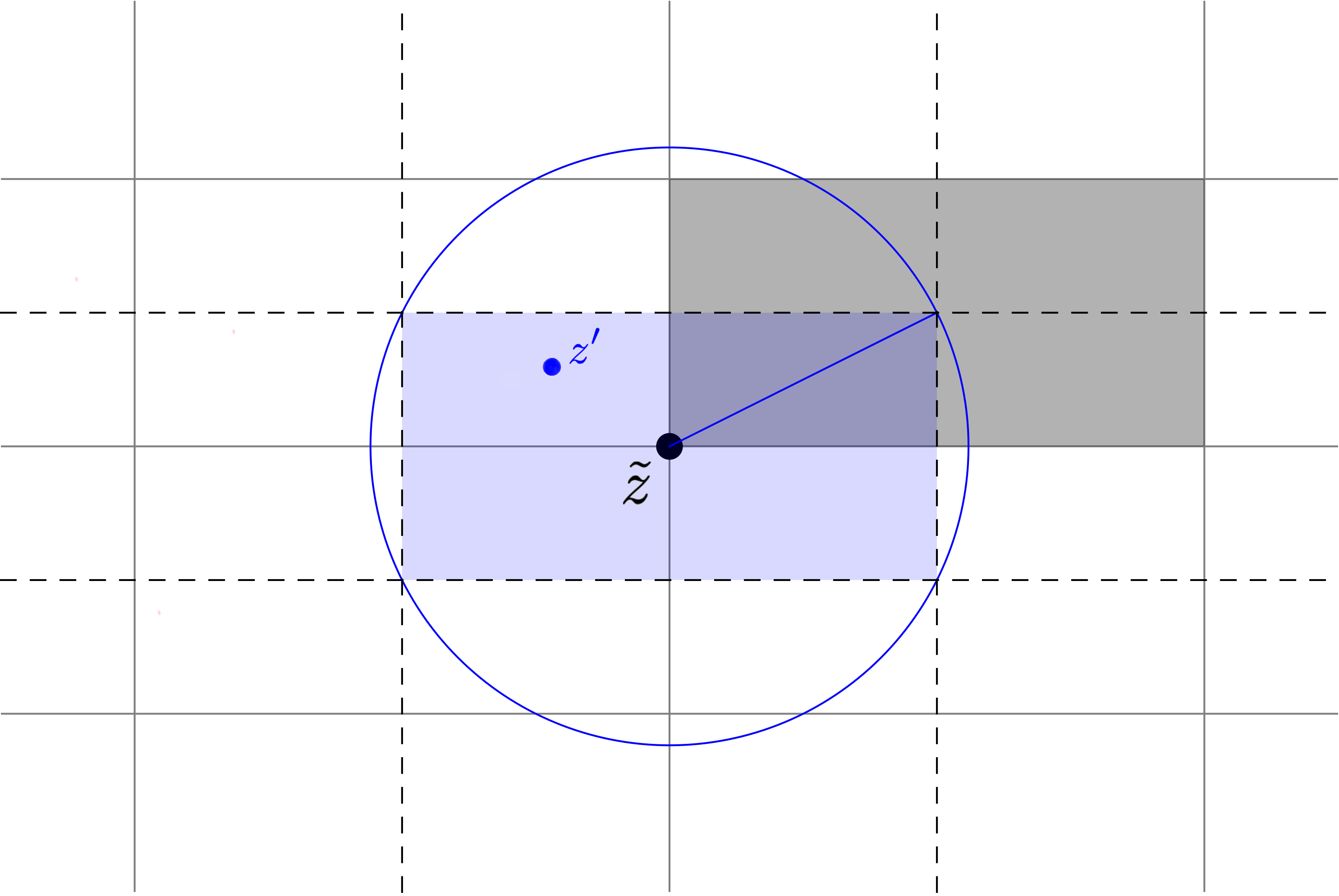}
\end{center}
\end{minipage}
\caption{\label{fig:disk}\small
The key points to the proof of Theorem~\ref{thm:diskArithmetic} :
$\fprec_{N}\times \fprec_{N}$ grid (left)
and the rounding strategy of $z'\in\C$ to $\tilde{z}\in \fprec_{N}\times \fprec_{N}$ within $\frac{1}{2}\ulp\tilde{z}$ (right).
}
\end{figure}

The proof of the theorem is straightforward and the complex extension of the definition of $\ulp$ is illustrated on Figure~\ref{fig:disk}.
In our library, the role of $R_\ast$ is implemented as an on-the-fly modification of $\ulp$.

To ease the tedious task of checking Theorem~\ref{thm:diskArithmetic}, let us recall a few
ground rules. Having Figure~\ref{fig:disk} in mind may help convey the key points.
The center of a disk $\disk_{z,r}$ is always considered exact;
it belongs to $\fprec_{N}\times \fprec_{N}$.
The corresponding $\ulp$ is thus a rectangle of dimensions $\ulp x \times \ulp y$ where $z=x+iy$. 
When computing the result of an operation, the new exact center $z'\in\C$ does not belong,
in general, to the grid $\fprec_{N}\times \fprec_{N}$.
The real and imaginary parts of $z'$ are rounded to their closest value, which gives
the new center $\tilde{z}\in \fprec_{N}\times \fprec_{N}$.
To compensate, one needs to increase the radius of the disk by, at most, half the diagonal of
the $\ulp$ rectangle. The claim that the center can be assumed exact is thus restored
and one is ready for the next operation.

\begin{remark}
The product of complex numbers is a multi-step operation over $\R$. As such, additional
$\ulp$ have to be added to bound the successive rounding errors in each intermediary step.
\end{remark}

\begin{remark}
There is a slightly tighter upper bound for the radius $R_\ast$, namely
\begin{align*}
R^\ast =  \round_M^+ \Big\{
{\textstyle\frac{1}{2}}\ulp\Big( z_a \times_{N,N'} z_b
& + \ulp(x_a \times_{N'}x_b) + \ulp(y_a \times_{N'}y_b)\\
& +i \left[ \ulp(y_a \times_{N'}x_b) + \ulp(x_a \times_{N'}y_b) \right]
\Big)\Big\}.
\end{align*}
However, the code complexity and computational cost of using $R^\ast$ are unreasonably high
compared to that of using $R_\ast$ and are not justified for the expected gain.
\end{remark}


\section{Implementation and numerical results}\label{par:num_results}

In this section, we present the library~\citelib{MLib} and the database~\citelib{MLibData}, which are
a companion to this article and the numerical results that we have obtained with it.
The main practical challenge is to preserve
the efficiency of computations as the scale of the problem spans 12 orders of magnitude.
The algorithm exposed in~\S\ref{par:split} is specifically designed for this.
However, all the logistics regarding process scheduling  and data handling have to keep up.
At this scale, the certification of all the data becomes crucial, not only though the procedure
described in~\S\ref{par:certif} but also, at the lowest level, to ensure that no bit corruption
occurs in the production and storage pipeline.
In its final state, our database takes 43TB of disk space
and we estimate (see Table~\ref{table:overallTime}) the overall cost to regenerate
it from scratch to 83 years-core (723\,000 hours-core) of computation time.
The actual time that was actually necessary to build this library from scratch
actually exceeds 1 million core-hours.

\subsection{Appropriate data structures}\label{par:dataStruct}

Computations are performed with the standard \texttt{FP80} format (\texttt{long double}) for low precision
or with \texttt{mpfr} and \texttt{mpc} numbers (a convenience wrapper for complex numbers)
for arbitrary higher precision. 
However, the foundation of scalable efficiency is the usage of appropriate data structures that are
efficient for storage (both for disk and memory usage\footnote{Because of memory alignement, the size of a structure
in RAM may not match its size on the disk.}), maintenance operations (sort, search, insertion) and
that are easily compatible with high-precision arithmetic.
In \citelib{MLib}, we have developed such structures specifically for this project.

\bigskip
Our core numerical format to store complex numbers in $[-2,2)+i [0,4)$ is the \texttt{u128} format.
It is composed of two unsigned 128-bits integers representing respectively the real and imaginary part with the two leading bits
reserved for the mantissa. A pair $(p,q)$ in \texttt{u128} format thus represents the complex number
\begin{equation}\label{eq:publishedValue}
z=(-2 + p \times 2^{-126}) + i q \times 2^{-126}.
\end{equation}
Roots of $p_n$ and $q_{\ell,n}$ are ultimately stored as \texttt{u128} numbers, which we call the \textsl{published value}
in the rest of this text.
The minimal absolute resolution of published values is therefore about $1.2\times 10^{-38}$.
This resolution is sufficient for our needs:
for example, the minimal distance between two distinct hyperbolic centers of period~41 is only~$2.45\times 10^{-23}$
so pairs of roots will differ by at least 50 bits.

\begin{figure}[H]
\captionsetup{width=.95\linewidth}
\begin{center}
\includegraphics[width=.85\textwidth]{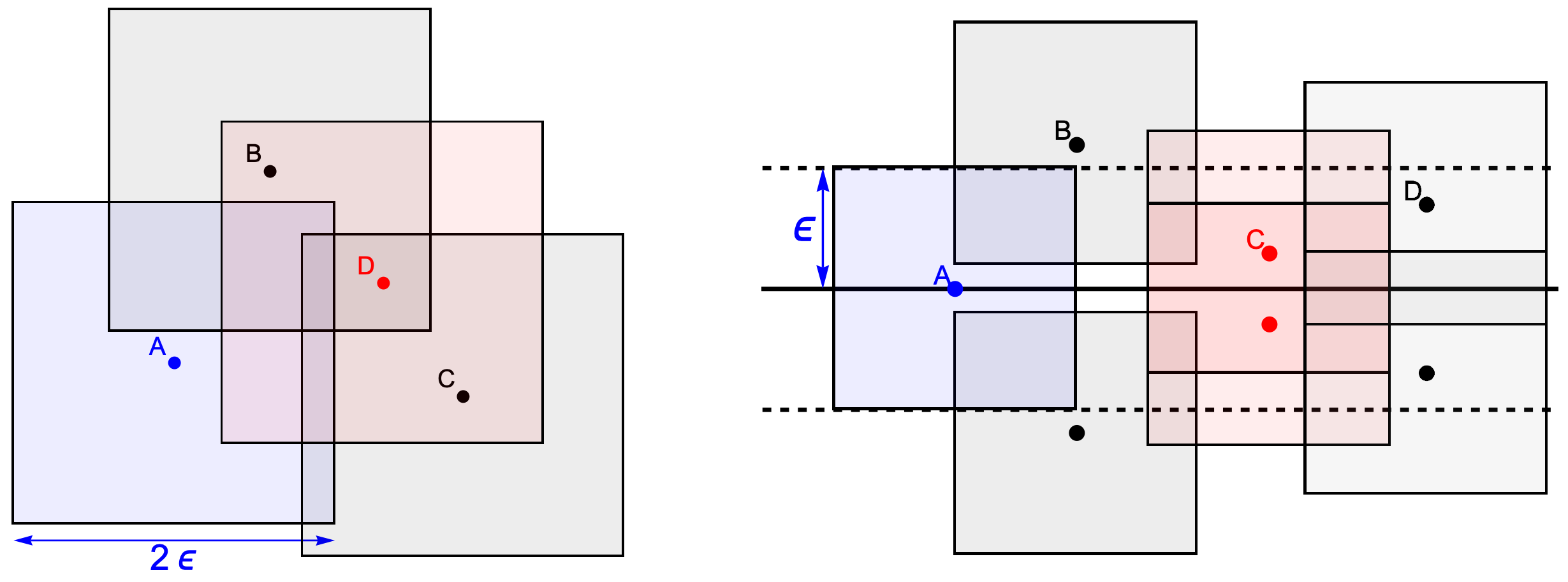}
\caption{\label{fig:nSet}\small
Insertion in \texttt{nset} is non commutative (left): inserting $A,D,B,C$ results in $\{A,D\}$ while
inserting $A,B,C,D$ results in $\{A,B,C\}$. Near the real axis (right), the points $A$, $B$ and $D$
respect the assumption~\eqref{nsetAsumption} so can be added to an \texttt{nset}.
The point $C$ is problematic; the absolute value of its imaginary part
does not exceed~$\epsilon/2$, where $\epsilon$ is the separation parameter,
and should be rounded to zero prior to insertion.
}
\end{center}
\end{figure}

The work-horse of our database is the \texttt{nset} structure. This structure implements the mathematical
idea of a finite set of points in the rectangle $[-2,2)+i (-4,4)$ that is symmetric with respect to the real axis.
Only the points in the upper-plane are stored.
Each stored point in the set is a \texttt{u128} complex number.
Each set comes with a separation parameter $\varepsilon>0$ and all operations on the \texttt{nset}
guaranty the following property: given two distincts points $z_1$, $z_2$ in the set or in its conjugate image:
\begin{equation}\label{nsetAsumption}
\max\{|\Re (z_1-z_2)|,|\Im(z_1-z_2)|\} \geq \varepsilon.
\end{equation}
In other words, points are considered as equal if one center belongs to the square ``pixel'' of side $2\varepsilon$
centered in the other one. Note that this is not an transitive relation (see Fig.~\ref{fig:nSet}).
Special care must be taken near the real axis to avoid collisions with the conjugate images ($z_2=\overline{z_1}$):
if $|\Im(z_1)|<\varepsilon/2$ then $z_1$ is undistinguishable from its conjugate and is therefore considered real;
the value $\Im z_1$ should be set to zero in order to respect~\eqref{nsetAsumption}.
In our database, we use a separation $\varepsilon_h = 3.23\times 10^{-27}$ for hyperbolic centers
and $\varepsilon_m = 8.08\times 10^{-28}$ for Misiurewicz points.
Note that this parameter serves a different purpose than the certification parameters (see Fig.~\ref{fig:rootCertif}).

\medskip
The points in an \texttt{nset} structure are totally ordered by the lexicographical order:
\begin{equation}\label{defLexi}
z_1 \preceq z_2  \qquad\Longleftrightarrow\qquad \Re z_1 < \Re z_2 \quad\text{or}\quad \begin{cases}
\Re z_1 = \Re z_2,\\
\Im z_1 < \Im z_2.
\end{cases}
\end{equation}
To optimize memory management, an \texttt{nset} can exist in two internal states, either locked or unlocked.
In its unlocked state, an \texttt{nset} contains a family of sub-sets (called bars) whose sizes form, ideally, a geometric
progression. Each bar is ordered with~\eqref{defLexi}. Bars are independent from one another (no ordering).
Insertion is performed on the last (smallest) bar unless it is full, in which case the bar is merged with the
previous one (recursively if necessary).
This means that insertion can be performed in $O(1)$ to $O(N)$ operations, where $N$ is the size of the \texttt{nset}.
However, large numbers of operations are exponentially rare: the $k^\text{th}$ bar is only merged into the bar $k+1$
when all the previous bars are full, so once in $O(2^{k})$ insertions. 
If we denote by $B$ is the size of the smallest bar and $K=\log_2 \left(1+\frac{N}{B}\right)$ the total number of bars,
then, on average, an insertion costs
\[
\frac{\displaystyle \sum_{k=1}^K k B 2^{k-1}}{\displaystyle \sum_{k=1}^{K} B 2^{k-1}} = \frac{1+(K-1)2^K}{2^K-1} = O(\log N)
\]
operations of memory management.
Conversely, a search on an unlocked \texttt{nset} takes $O(\log^2 N)$ operations because each
of the $K=O(\log N)$ bars must be searched independently.
Note that in order to ensure~\eqref{nsetAsumption}, an insertion can only be performed after a search
for the new point has confirmed that it is indeed appropriate to add it.
Points near the real line (like the point $C$ on the right-hand side of Figure~\ref{fig:nSet}) must be treated as exceptions:
it is up to the user to decide wether it is appropriate to correct the imaginary part to zero or to refuse to add the point.

In its locked state, all the memory bars of an \texttt{nset} are merged, which means that the set becomes fully ordered.
This is a one-time $O(N)$ cost.
Insertion within a locked \texttt{nset} is not authorized anymore.
Note that unlocking remains possible and, in that case, the previous logarithmic costs can be restored (using dyadically smaller
bars before the main data chunk).
In a locked \texttt{nset}, the cost of a search drops to $O(\log N)$. A locked \texttt{nset} is essentially ready
to be written to the disk. Actions on an \texttt{nset} structure require only a $O(\log N)$ memory overhead
(pointers to the each memory bar), which turns out to be constant in practice\footnote{No more than 64 memory bars are ever necessary if one uses \texttt{unsigned long} for array indices.}. Note that we haven't implemented deletion
as a guaranty that no point will ever be added unless we know for certain that it should.

\bigskip
We have also implemented \texttt{pset}, \ie planar sets, which is a structure analogous to the \texttt{nset} but
based on \texttt{FP80} numbers (\texttt{long double}) instead of \texttt{u128}. The \texttt{pset} structure
implements the mathematical idea of a finite set of points in the complex plane, with no geometric restrictions.
It provides logarithmic search and insertion times with logarithmic memory overhead.
The theoretical optimum  \cite{FreSak89}, \cite{LAR13} for operations on ordered memory structures is $O(\log N/\log\log N)$.
The performances of our structures are therefore quite satisfactory and may be of general interest.
In particular, aside from improved RAM usage, our low memory footprint minimizes the cost of the conversions
to and from a disk-storage format, which is crucial when handling terabytes of data\footnote{A pointer requires typically 8 bytes
on a 64-bit architecture; implementing \eg a balanced tree of \texttt{FP80} with 2 pointers to the children would therefore induce
about 50\% of memory overhead.}.

\bigskip
Our format to handle vectors of high-precision numbers is called \texttt{mpv}.
This format is optimized for storage and large scale disk access. 
In a given vector, all the elements share a common precision. 
The vector can either be real or complex, with consecutive pairs of real/imaginary parts.
Individual operations on elements are possible, though awkward.
One important feature of this format is the possibility to concatenate multiple vectors into a single file (what we call mini-files)
and to perform parallel read-writes on each mini-file.

\bigskip
Our implementation also contains tools to import, export and compare \texttt{CSV} files (comma separated values)
of complex numbers of arbitrary precision
(using $a, b$ to represent $a+ib$). This format is extremely slow because of the binary to decimal conversion and requires
about 1.8 times more disk space.
It ensures however a human-readable output and a minimal compatibility with other computing or graphical tools.

\subsection{Certification of the results and protection against data corruption}\label{par:certif2}

Each published value (see~\eqref{eq:publishedValue} for a definition)
for a root of $p_n$ or $q_{\ell,n}$ comes with multiple levels certifications.
Let us underline that the following theorem contains about $3\times 10^{12}$ individual
statements, whose proof are computer assisted.
\begin{thm}\label{thm:certif}
There are constants $\varepsilon_R$, $\varepsilon_N$ and $\varepsilon_S$ given in Figure~\ref{fig:rootCertif}
for which the following holds.
In the \texttt{nset} database for the roots of $p_n$ or $q_{\ell,n}$, each published value $z_j\in 2^{-126}(\Z+i\N)$
can be paired with a unique
actual root $z_\ast \in \disk(z_j,\varepsilon_R)$. This pairing is bijective in the upper half plane $\Im z\geq 0$.
Each pair of published values $z_j,z_k$ satisfies $|z_j-z_k|\geq \varepsilon_S$ and either $z_j\in\R$ or $|z_j-\bar{z}_j|\geq \varepsilon_S$.
Lastly, the disk $\disk(z_j,\varepsilon_N)$ is entirely contained in the
attraction bassin of $z_\ast$ for the Newton method of the associated polynomial.
\end{thm}

\proof
Using the controlled rounding features of the MPFR library \citelib{MPFR},
we have implemented disk arithmetic \ie~Theorem~\ref{thm:diskArithmetic}.
This allows us to provide, for each root candidate $z_j$, a numerical proof that the assumptions
of Theorem~\ref{thm:proofRoots} on the localisation of each root can indeed be applied.
Thus, we can certify that the published coordinates of each root differ from an actual root by no more
than $10^{-30}$ for $p_n$ or $10^{-35}$ for $q_{\ell,n}$. 
Similarly, using disk arithmetic with $\varepsilon_N$, we can check the assumptions of Theorem~\ref{thm:proofContract}
which ensures the claim on the Newton bassin.
Next, each pair of two published values or any pair with conjugate values satisfies the separation
assumption~\eqref{nsetAsumption} because it is build into our storage structure (see \S\ref{par:dataStruct}).
It is slightly stronger than the one claimed here.
This separation guaranties that it is possible to count unambiguously all the real roots,
and all the non real roots in the upper half-plane. Comparing with the theoretical count given
by Theorem~\ref{thm:countingHyp} and~\eqref{eq:MisCount} ensures our bijection claim.
\endproof

\begin{figure}[H]
\captionsetup{width=.95\linewidth}
\noindent
\includegraphics[width=0.4\textwidth]{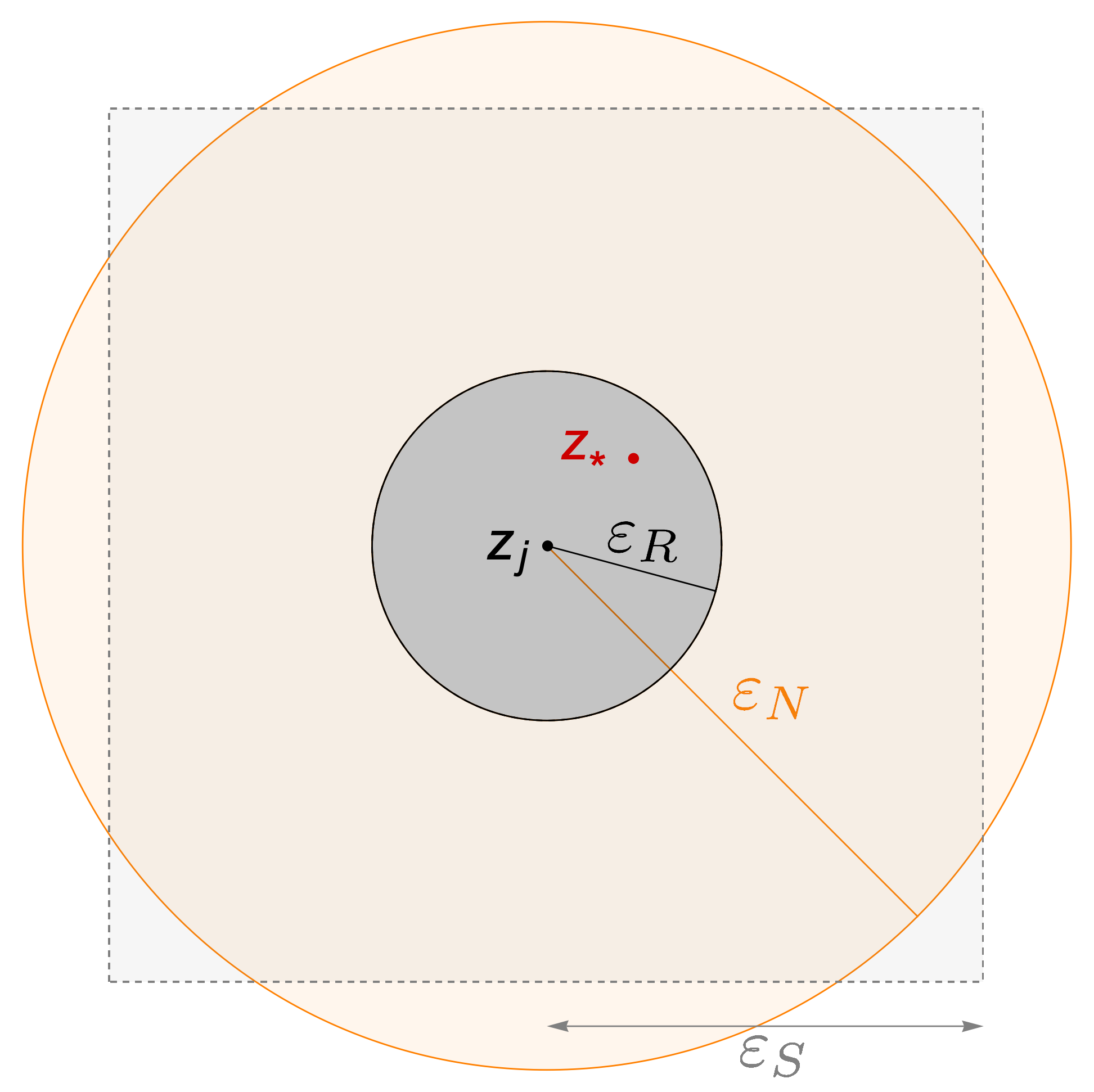}
\qquad
\raisebox{0.2\textwidth}{%
\begin{minipage}{.3\textwidth}
\vspace{0pt}
\[\arraycolsep=10pt\def\arraystretch{1.2}
\begin{array}{|c|c|c|c|}\cline{2-4}
\multicolumn{1}{c|}{} & p_{n\leq 40}  & p_{41} & q_{\ell,n}\\\hline
\varepsilon_R & \multicolumn{2}{c|}{10^{-30}} & 10^{-35} \\\hline
\varepsilon_N & 10^{-24} & 10^{-25} & 10^{-32} \\\hline
\varepsilon_S &  \multicolumn{2}{c|}{3.23\times 10^{-27}} & 8.08\times 10^{-28} \\\hline
\end{array}
\]
\end{minipage}
}
\begin{center}
\caption{\label{fig:rootCertif}\small
For each root $z_\ast$ with $\Im z_\ast\geq0$, the published value $z_j \in 2^{-126}(\Z+i\N)$ comes with different radii.
The radius of separation $\varepsilon_S$ guaranties proper counting,
the radius of certification $\varepsilon_R$ is such that $z_\ast\in \disk(z_j,\varepsilon_R)$  and
the radius of convergence $\varepsilon_N$
ensures that the $\disk(z_j,\varepsilon_N)$ is contained in the bassin of attraction of $z_\ast$
for the Newton method.
}
\end{center}
\end{figure}

\begin{remark}
We are confident that each component of the published values is exact up to
\[
\pm{\textstyle\frac{1}{2}}\ulp(\Re \operatorname{\tt u128}) = \pm2^{-127}\simeq 5.9 \times 10^{-39},
\]
which is better than the $\varepsilon_R$ radius claimed in Theorem~\ref{thm:certif}. Indeed, once we have a $z_j$ value
that passes the certifications for both disk arithmetic and Newton bassin, we can refine its value using Newton's method
until reaching a fixed point at the desired precision~\eqref{eq:publishedValue}, up to a final rounding error.
\end{remark}

At the terabyte scale, the possibility of bit corruption becomes a significant issue.
Besides the previous statement, practical precautions must therefore be taken to ensure the integrity of the data
along the whole production chain, from the initial computation that finds a root to its final storage in a file.
This protection must also extends to any subsequent use of the stored values.

\medskip
Our library~\citelib{MLib} incorporates a strict data certification procedure.
Our \texttt{nset} file format implements a header that contains the MD5 checksum of the data stored
in the rest of the file. When writing a file, a checksum of the original data is first computed in memory,
then the file is written onto the disk and the checksum is written in the file header.
To detect subsequent data corruption (due to an error in a file transfert or a random bit flip),
each time a file is loaded, the checksum of the data stored in the file is computed again
and checked against its original value stored in the header. This protocol ensures the
data integrity once the original MD5 stamp has been generated. 

Lastly, one needs to check that a random memory corruption has not affected the initial creation process, after the value was
computed but before the original MD5 checksum was generated. The only sensible way to guard against this problem
is to perform a new independent certification (count and proofs) of Theorem~\ref{thm:certif} of all
the data written on the disk, which we did using our applications \texttt{hypCount}, \texttt{hypProve}, \texttt{misCount} and \texttt{misProve}. To encourage independent verifications, the code of those applications has been kept as
minimalistic as possible.

The only uncertainty left is a data corruption that could happen after this second certification
but would remain undetected by the MD5 checksums. As MD5 is a 128 bit cryptographic hash function\footnote{Granted
that our database is not subject to a malicious attack that would actively seek a vulnerability.},
the probability of such an event is of order $2^{-128}\simeq 3\times 10^{-39}$.

\subsection{Quick single-core splitting of polynomials up to degree $10^9$}

\begin{figure}[t]
\captionsetup{width=.95\linewidth}
\begin{center}
\includegraphics[width=.95\textwidth]{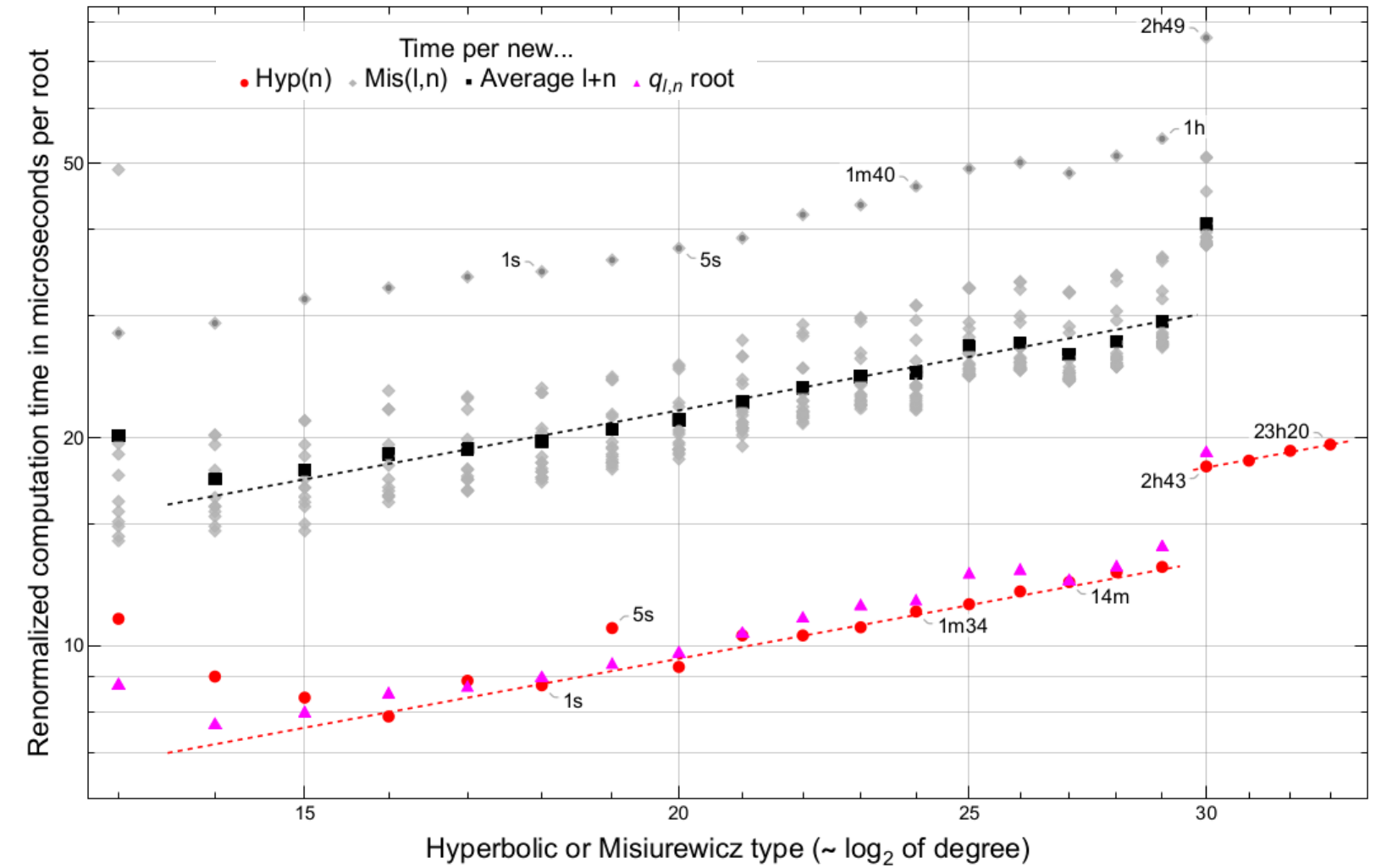}
\caption{\label{fig:quickTimes}\small
Computation time, in microseconds, per new parameter $\hyp(n)$ computed with \texttt{hypQuick} (red dots)
or $\mis(\ell,n)$ with \texttt{misQuick} (gray lozenges).
The black square give the average time at a given $\ell+n$. The slowest subtype corresponds to $q_{\ell,2}$.
Labels indicate the total splitting time of $p_n$ or $q_{\ell,2}$ on a single CPU.
The magenta triangles represent the average computation time per root of $q_{\ell,n}$,
regardless of wether the root is a hyperbolic divisor or not.
The dashed lines correspond to a slope $O(n^{0.8})$ or $O((\ell+n)^{0.8})$, and the overall fit suggests that the algorithm
behaves, in practice, as $O(d(\log d)^{0.8})$.
}
\end{center}
\end{figure}

Splitting the polynomials $p_n$ and $q_{\ell,n}$ for the first values of $n$ is a standard computational challenge.
Using the standard algorithms of Section~\ref{par:roots}, splitting $p_{25}$ (degree 1.7 million) is a task that used
to lie at the computational frontier and required days of core usage.
Using our new algorithm based on level lines (presented in Section~\ref{par:split}),
the polynomial $p_{25}$ can now be split in only a few minutes on a single core of consumer hardware.

As illustrated of Figure~\ref{fig:quickTimes}, we push the day-core limit back to $p_{33}$ (degree 4.3 billion).
As explained in \S\ref{par:HardwareLimit} below, this degree is the highest one that can be handled
using hardware arithmetic.

\subsubsection{Implementation of the level line algorithm}

In our implementation~\citelib{MLib}, the application \texttt{hypQuick} splits $p_n$ and computes a listing of $\hyp(n)$.
On Figure~\ref{fig:quickTimes}, the low degree scatter for processes that execute in less than 5 seconds
is driven by memory loading and is not significant.
The 40\% jump in the average processing time per root at~$p_{30}$ is due to the fact that we double
the number of points on the initial level line when $n\geq 30$. The code of this application is concise
and is largely independent of the rest of the library. It relies on the \texttt{pset} structure (see \S\ref{par:dataStruct})
to collect and sort the roots and prevent duplicates. The functions that implement the Newton's method
are common to our whole library.

\medskip
Let us comment on our implementation for the splitting of $p_n$.
The endpoints of the level curve~$|p_n(z)|=5$ along the real line are found by dichotomy.
The discrete level curve $\lev_{5, 2^{n-1}}(p_n)$ defined by~\eqref{eq:defDiscreteLev}
is computed on the fly using Proposition~\ref{prop:computeLevelLine} with $M=8$ points per turn
if $n\leq 29$ or $M=16$ points per turn if $n\geq 30$, \ie one computes
solutions of $p_n(z_k)=5 e^{2i k \pi/M}$,
each $z_k$ acting as the starting point for the Newton iterations towards $z_{k+1}$  (see Fig.~\ref{fig:MVSplitAlgo}).
For each $z_k$ such that $k\equiv 0 \mod 2$ (if $n\leq 29$) or $k\equiv 0 \mod 4$ (if $n\geq 30$), a Newton descend attempts to find a new root of $p_n$.
To prevent divergent trajectories, the number of iterations is capped by $O(\log d)$.
The new root is then searched in $O(\log d)$ within a \texttt{pset} structure that contains all the roots of the
known divisors of $p_n$ given by Theorem~\ref{thm:countingHyp}.
If the root does not belong to a divisor, it is added,  in constant time, into a \texttt{pset} structure that collects all new roots.
Then the computation of the discrete level curve restarts and the process continues.
Overall, the complexity of splitting $p_n$ and identifying $\hyp(n)$ using this implementation is $O(d\log d)$ with $d=2^{n-1}$,
which explains the record times shown on Figure~\ref{fig:quickTimes}.
All computations are performed using~\texttt{FP80} hardware arithmetic.
The memory requirement is $d+O(\sqrt{d})$ to store both the divisors and the new roots.
The application finally outputs the listing of the set $\hyp(n)$ in \texttt{CSV} format.

\medskip
We have also implemented two applications that split $q_{\ell,n}$ for $\ell+n\leq 30$.
The application  \texttt{misQuick} uses the original polynomials $q_{\ell,n}$, which have high-multiplicity roots,
as described in Theorem~\ref{thm:fact_mis}.
Figure~\ref{fig:quickTimes} synthesizes our benchmarks.
The computation time appears to obey a general rule $O(d(\log d)^{0.8})$ per new parameter, regardless of
multiplicity. The constant factor in front of this experimental law depends on the sub-family of polynomials
on which the benchmark is run and the parameters (like the ratio of starting points per root).
In theory, a $O(d \log d)$ scaling law should have been observed.
Let us point out that the renormalized computation time per root (regardless of wether it is a Misiurewicz type or not)
is equivalent to that of hyperbolic polynomials of the same degree. This means that high-multiplicities have,
at least in that case, a minimal effect on the efficiency of our algorithm.
To get rid of multiplicities, the application \texttt{misSimpleQuick} uses the simplified polynomials 
$s_{\ell,n}=p_{\ell+n-1}+p_{\ell-1}$ instead of $q_{\ell,n}$. This polynomial~\eqref{eq:polySimple}
plays a central role in the proof of
Theorem~\ref{thm:fact_mis} (see equation~\eqref{eq:polyMisSimple} in Appendix~\ref{par:factorization}),
it has simple roots that contain $\mis(\ell,n)$ and is a lot simpler to compute than the fully reduced
polynomial~$m_{\ell,n}$ given by~\eqref{eq:mis_red}.

\subsubsection{Reaching the limit of hardware arithmetic}\label{par:HardwareLimit}

The different \texttt{C} standard types of finite precision arithmetic are recalled in Table~\ref{table:FPxx}.
For example, \texttt{FP80} arithmetic is the highest precision available on common hardware and offers 64 significand bits;
as such, it is suitable to represent the list of all roots of a polynomial only when the minimal root separation
exceeds $2^{-64}\simeq 5\times 10^{-20}$.
Moreover, a margin of at least a few bits is necessary for the Newton dynamics to be meaningful and
converge to those roots.

\begin{table}[H]
\begin{center}\small
\begin{tabular}{ccccccccc}
\hline
Norm & Standard C 	& Java &  Sign & Exponent & Significand & RAM & Disk\\\hline
FP32	& \verb|float|		& \verb|float|	& 1 bit 	& 8 bit	& 1+23 bit	 & 4B & 4B \\\hline
FP64	& \verb|double|		& \verb|double|	& 1 bit 	& 11 bit	& 1+52 bit  & 8B & 8B	\\\hline
FP80 	&\multirow{2}{*}{\texttt{long double}}	 & -	& 1 bit 	& 15 bit	& 64 bit & 12-16B & 10B \\\cline{4-8}
FP128	& & -			& 1 bit 	& 15 bit	& 1+112 bit & 16B & 16B\\
\hline
\end{tabular}
\caption{\label{table:FPxx}\small
Standard types of finite precision arithmetic.
}
\end{center}
\end{table}

Let us consider the two left-most real roots $c_1=-2+\varepsilon_1$ and $c_2=-2+\varepsilon_2$
of $p_n$, with $0<\varepsilon_1<\varepsilon_2$. For large $n$, the corresponding dynamics (see Fig.~\ref{fig:leftmostRoot})
linger near 2 for most of the trajectory and separate from one another only at the last step,
with
\[
p_{n-1}(c_1) = \sqrt{2-\varepsilon_1} \simeq \sqrt{2} \left(1-\frac{\varepsilon_1}{4}\right)
\quad\text{and}\quad
p_{n-1}(c_2)= -\sqrt{2-\varepsilon_2} \simeq -\sqrt{2}\left(1-\frac{\varepsilon_2}{4}\right),
\]
hence the rough estimate:
\begin{equation}\label{eq:approxLeftRoots}
2\sqrt{2} \simeq |p_{n-1}(c_1)-p_{n-1}(c_2)| \simeq |p_{n-1}'(-2)| |c_1-c_2|.
\end{equation}

The computation of the derivative is classical and is based on the connection between the parameter space
and that of dynamics.
Indeed, as $f_c'(z)= 2z$ and
\[
(f_c^k)'(z)=\prod_{\ell=0}^{k-1} f_c'(f_c^\ell(z)) = 2^k z f_c(z)\ldots f_c^{k-1}(z),
\]
the recursive definition of $p_n$ gives
\[
p_{n}'(c)  = 1+\sum_{\ell=1}^{n-1} (f_c^\ell)'(p_{n-\ell}(c)) 
= 1+\sum_{\ell=1}^{n-1} 2^\ell p_{n-\ell}(c) p_{n-\ell+1}(c) \ldots p_{n-1}(c) 
\]
thus
\begin{equation}\label{eq:DerParamDyna}
\frac{p_{n}'(c)}{(f_c^{n-1})'(c)} = 
1+\sum_{\ell=0}^{n-2} \frac{2^\ell p_{n-\ell}(c) p_{n-\ell+1}(c) \ldots p_{n-1}(c) }{
2^{n-1} p_1(c) p_2(c)\ldots p_{n-1}(c) }
= 1+\sum_{k=1}^{n-1} \frac{1}{
(f_c^k)'(c) }
\end{equation}
For $c=-2$, one has $f^k_{-2}(-2)=2$ for all $k\geq1$ and $(f^k_{-2})'(-2) = -4^k$.
Substitution in~\eqref{eq:DerParamDyna} gives $p_{n}'(-2)= -\frac{4^{n}+2}{6}$
and thus, within the approximation~\eqref{eq:approxLeftRoots}:
\begin{equation}
|c_1-c_2|  \propto 4^{-n}. 
\end{equation}

\begin{figure}[H]
\captionsetup{width=\linewidth}
\begin{center}
\includegraphics[width=.8\textwidth]{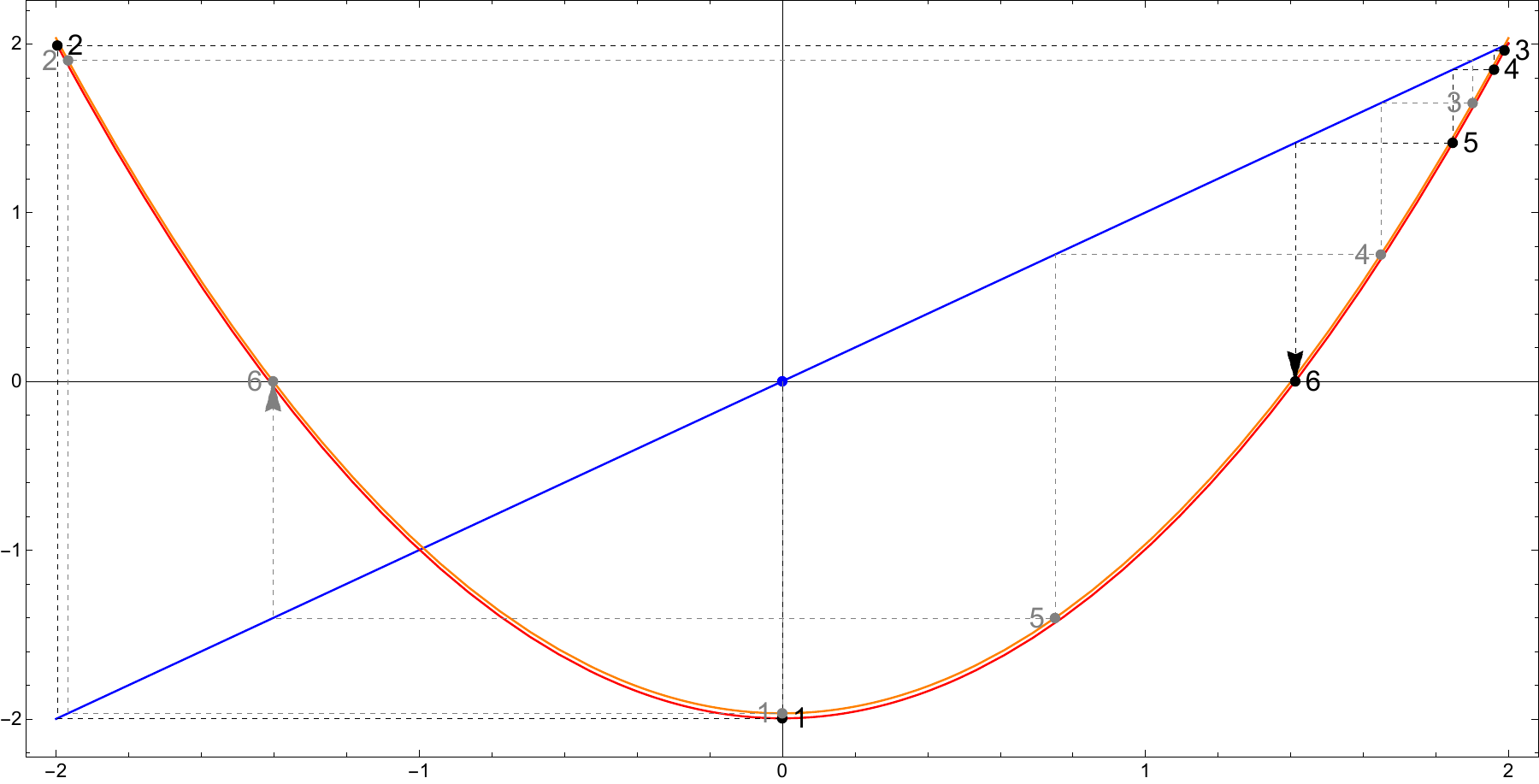}
\caption{\label{fig:leftmostRoot}\small
Dynamics associated with the left-most real roots of $p_n$ (here with $n=6)$.
}
\end{center}
\end{figure}

\begin{figure}[H]
\captionsetup{width=\linewidth}
\begin{center}
\includegraphics[width=.83\textwidth]{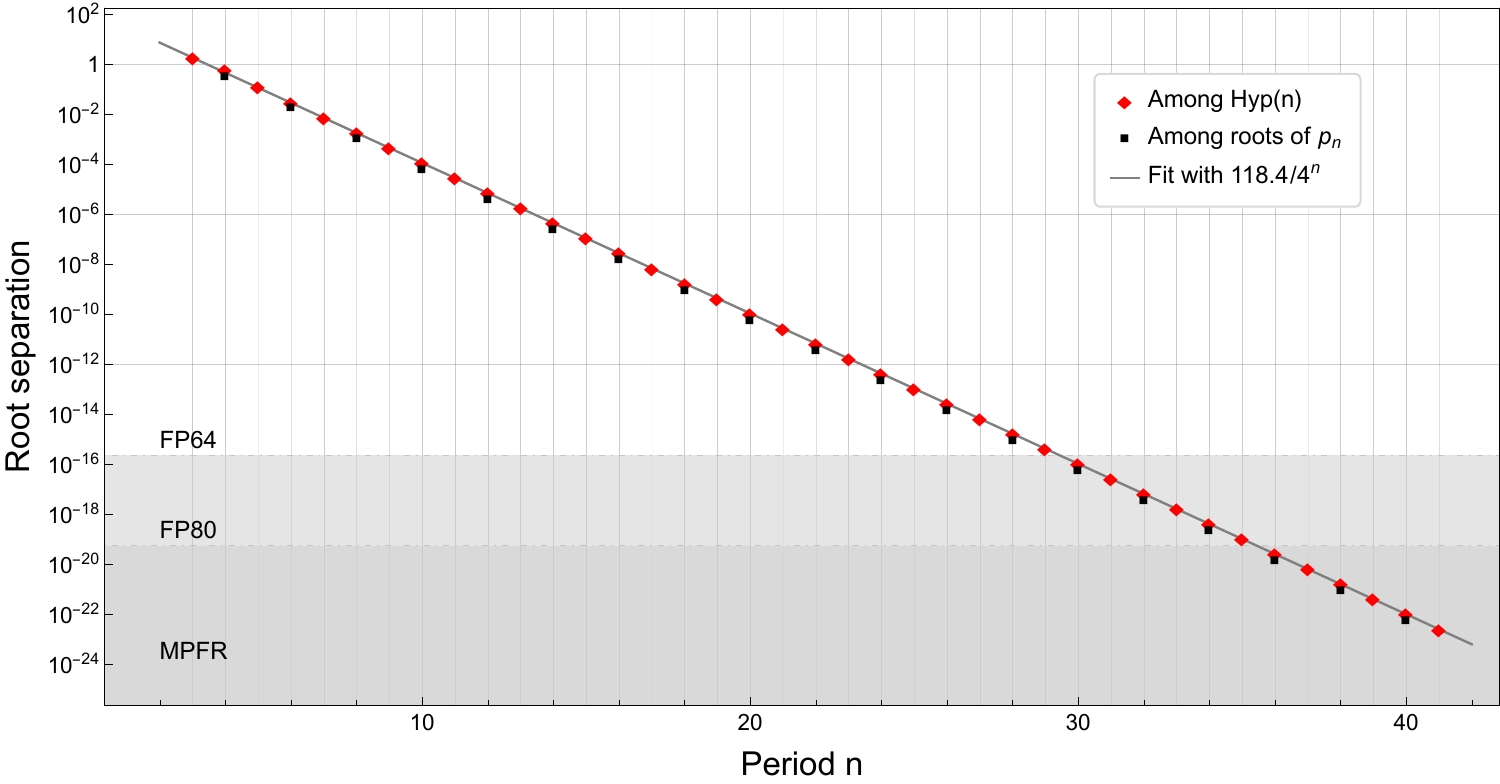}
\caption{\label{fig:rootSeparation}\small
Minimal distance among hyperbolic centers $\hyp(n)$ and among the roots of $p_n$  (which differs only for odd $n$).
The shaded regions  highlight the resolution of hardware finite precision arithmetics \texttt{FP64} and  \texttt{FP80}
and the region where a higher precision is required.
}
\end{center}
\end{figure}

\begin{figure}[H]
\captionsetup{width=\linewidth}
\begin{center}
\includegraphics[width=.82\textwidth]{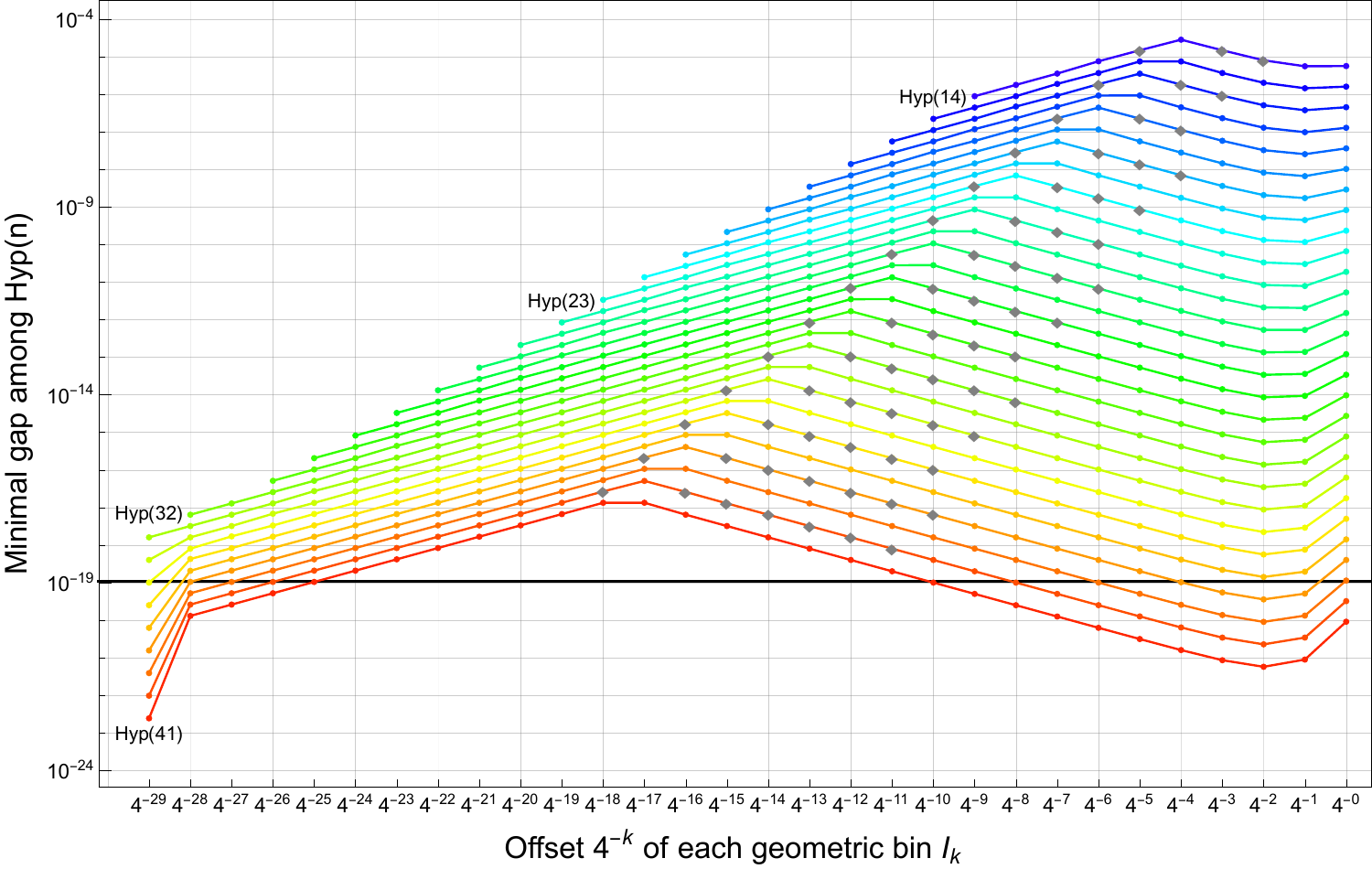}
\caption{\label{fig:separationHistogram}\small
Histogram of the minimal distance among $\hyp(n) \cap I_k$ with geometric bins $I_k=(-2+4^{-k-1}, -2+4^{-k}]\times\R$
illustrating that the closest gap happens near $c=-2$. The gray data points corresponds to bins where
the minimal distance among all roots of $p_n$ is strictly smaller (close proximity of a divisor).
}
\end{center}
\end{figure}

\begin{figure}[H]
\captionsetup{width=\linewidth}
\begin{center}
\includegraphics[width=.82\textwidth]{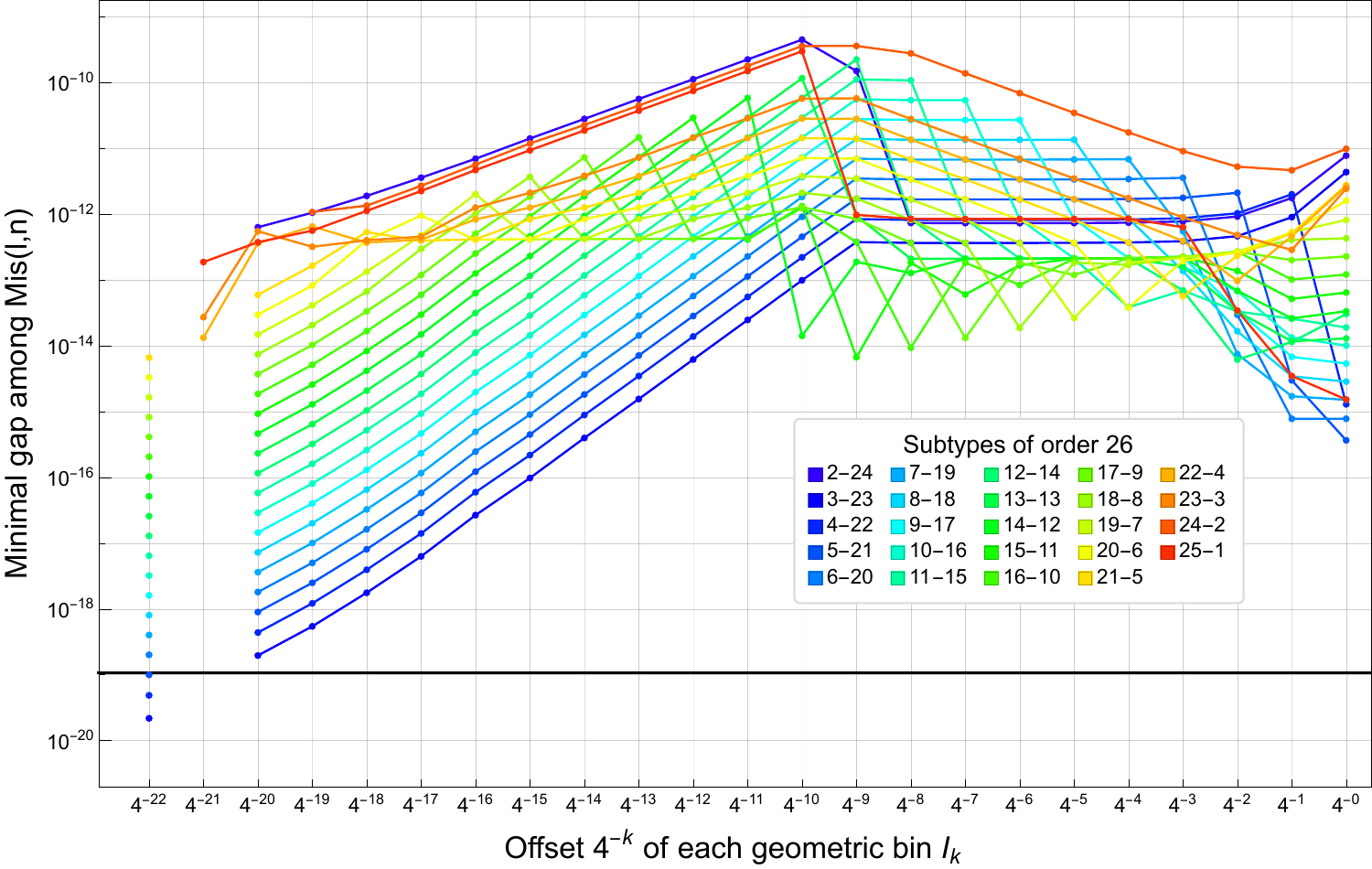}
\caption{\label{fig:separationHistogramMis}\small
Histogram of the minimal distance among $\mis(\ell,n) \cap I_k$ for $\ell+n=26$,
which is the first Misiurewicz order that cannot be integrally represented with \texttt{FP80} arithmetic.
The separation within $\mis(3,23)$ and $\mis(4,22)$ drops below the limit of the representation.
}
\end{center}
\end{figure}

The actual root separation computed on our high-resolution dataset is illustrated on
Figures~\ref{fig:rootSeparation}, \ref{fig:separationHistogram} and \ref{fig:separationHistogramMis}.
Within $\hyp(n)$, one can check numerically that the two left-most real roots are, indeed, the closest points.
The experimental scaling law matches the predicted behavior (relative error smaller than $0.1\%$ when $n\geq9$):
\begin{equation}
\min_{\substack{c,c'\in \hyp(n)\\ c\neq c'}}|c-c'|\sim \frac{118.4}{4^n}\cdotp
\end{equation}
For odd~$n$, both roots lie at the center of a primitive component of period~$n$.
For even~$n$, the minimal distance among roots of $p_n$ is slightly smaller than among $\hyp(n)$, due to the presence of 
a closer pair involving a root of a divisor (see Fig.~\ref{fig:separationHistogram}): one root is the center of
a primitive component of period~$n/2$ and the other one is the center of a satellite of period~$n$.

\bigskip
In the $\mis(\ell,n)$ family, the limitation of hardware arithmetic is first hit for the subtypes $\mis(3,23)$ and $\mis(4,22)$,
as illustrated on Figure~\ref{fig:separationHistogramMis}.
Both sets contain, at the leftmost end of $\M$, two distincts real parameters separated by less than~$4.86 \times 10^{-20} < 2^{-64}$.
They cannot be separated in \texttt{FP80} arithmetic.
More generally, we have observed that all sets $\mis(\ell,n)$ with $26\leq \ell+n\leq 32$
and $3\leq \ell\leq 3(\ell+n-25)+1$ contain, near the left tip,
parameters that cannot be represented in \texttt{FP80} arithmetic.
Obviously, our app \texttt{misQuick} is subject to those limitations and reports the missing roots, which is normal behavior.

However, to offer a user friendly interface, the apps \texttt{misQuick} and \texttt{misSimpleQuick} provide an option
to refine all roots near the real axis and left of $-2+4^{-16}$ using high-precision \texttt{MPFR} numbers.
This option does not sacrifice much of the performances and allows us to compute a complete list of all $\mis(\ell,n)$
parameters of type $\ell+n\leq 30$ in the time-frame illustrated on Figure~\ref{fig:quickTimes}.

When $\ell+n=30$, a second region (near $\Re c = -1$) contains close complex pairs (see Fig.~\ref{fig:rootsHardToFind})
that are separated by less than $2^{-62}$. Computing types of order $\ell+n\geq 31$ would require introducing
multiple high-precision exceptions. At that point, one should just switch to our HPC implementation described in Section~\ref{par:ImplementationHPC}.

\bigskip

Those numerical observations establish that, in the family of all hyperbolic centers
and Misiurewicz parameters, the roots of  $p_{33}$ and $m_{\ell,n}$ with $\ell+n=30$
constitute the final frontier of what can be computed using \texttt{FP80} hardware
arithmetic. As illustrated on Figure~\ref{fig:quickTimes}, our new algorithm based on level-lines
can reach this frontier in just a bit less than one day-core for the hyperbolic parameters and in about 3 hours
for each pre-periodic type (half that for \texttt{misSimpleQuick}).

\bigskip

\begin{remark}
The upper exponent limit of \texttt{FP80} numbers is $2^{2^{14}}\simeq 1.18\times 10^{4932}$.
In \texttt{hypQuick}, we deal with polynomials
of degree up to $2^{32}$. Let us underline that
\[
|z|^{2^{32}} \leq 1.18\times  10^{4932} \quad\Leftrightarrow\quad |z|\leq R_{\operatorname{\tt FP80}}
\quad\text{with}\quad R_{\operatorname{\tt FP80}}\simeq 1 + 2.64\times 10^{-6}
\]
and that $\mathcal{M}\cap \disk(0,R_{\operatorname{\tt FP80}})^c\neq\emptyset$.
To compute accurately one Newton step $p_n(z)/p_n'(z)$ when we detect that both the numerator and denominator
are so large that each will exceed the exponent limit (which happens in the first steps),
we simplify a common large factor, without loss of precision.
\end{remark}

\subsubsection{A-posteriori certification of the \texttt{FP80} results}

There is no realistic hope of certifying \texttt{FP80} computations within their own framework.
The operations on \texttt{long double} are extremely hardware dependent and we have no real garanties that
a given implementation will perform properly, in any circumstances, up to the last bit.
Implementing disk arithmetic is possible and can provide
a reasonable amount of confidence and is therefore included in our library.
However, using disks based on hardware arithmetic still
leaves an unacceptable room for doubt\footnote{A single disk radius mistakenly rounded the wrong way
breaks the logic chain of who contains whom.}
that only \texttt{MPFR} (which implements guarantied
rounding directives) can waive.

\medskip
For all periods 3 to 33, we did compare the listing of $\hyp(n)$ obtained
in \texttt{FP80} arithmetic using \texttt{hypQuick} to the certified listing obtained
independently with high-precision \texttt{MPFR} and certified disk arithmetic (see \S\ref{par:certif} and \S\ref{par:certif2}).
The maximum deviation between two corresponding lists does not exceed $5.24\times 10^{-19}$ (reached, in
our case, for $p_{24}$).
In particular, the lists of roots are identical up to a precision of $10^{-18}$.
This a-posteriori check constitutes the best possible certification of our \texttt{FP80} results.

Similarly, the a-posteriori certification of $\mis(\ell,n)$ for $\ell+n\leq 25$ reports a maximum
deviation of $3.25\times 10^{-19}$ between the lists obtained in hardware arithmetic and
those obtained with \texttt{MPFR}. Subtypes of higher order also pass the certification,
with an obvious exception for the parameters that cannot be represented in hardware
arithmetic and are thus out of the reach of \texttt{misQuick} (see Section~\ref{par:HardwareLimit}).
Of course, they can be certified only if one activates the high-precision option for parameters left of $-2+4^{-16}$.

\subsection{Our HPC approach for splitting a tera-polynomial}\label{par:ImplementationHPC}

Scaling up our algorithm from the splitting of $p_{33}$ to that of $p_{41}$ requires the resources
from a HPC center. The two key ingredients are massively parallel operations and the ability to switch
on the fly between hardware arithmetic, and arithmetic in arbitrary precision.
The splitting process is decomposed in multiple stages.

\subsubsection{Overview of the splitting process}

The first stage, called is a massively parallel raw search that involves~$J=2^{n-27}$ tasks
(respectively $J=2^{n+\ell-27}$ for $q_{n,\ell}$) that are
independent of each other. Each task consists in finding roots in a certain sub-region of the Mandelbrot set. More
precisely, the computation of the discrete level line $\lev_{5, 2^{n-1}}(p_n)$ is split in $J$ pieces of equal cardinality.
We choose $2^{n-1}/J = 2^{27}$ to ensure that the resulting \texttt{nset} files of roots would take about 1.1GB each on disk.
Each task computes a part of the finest level line and finds all the roots that can be reached from it using the Newton map.
In the library, the corresponding apps are \texttt{hypRaw} and \texttt{misRaw}.
The alternative \texttt{misSimpleRaw} uses the simplified polynomials $s_{\ell,n}=p_{\ell+n-1}+p_{\ell-1}$
defined by~\eqref{eq:polySimple}
instead of $q_{\ell,n}$, which eliminates most hyperbolic divisors and ensures that all roots are simple.

\medskip

The superiority of our level-line algorithm is illustrated on Figures~\ref{fig:meshRefine} and~\ref{fig:rawJobs}:
sections of equal cardinality of $\lev_{\lambda, N}$
span physical regions of extremely varied size. The level-line is a naturally self-refining mesh in the sense that the mesh is,
by construction, denser in regions where $\arg P(z)$ cycles more rapidly, which indicates an intrication of the bassins of attraction
for the Newton flow and is thus likely to occur for the Newton map too.
As iso-angles~\eqref{eq:NewtonFlow} and external rays coincide asymptotically at infinity,
it can be expected that those regions still corellate with a higher density of external rays,
and thus capture faithfully, as the ray land, the harmonic measure of $\M^c$ and thus the roots of $p_n$ and $q_{\ell,n}$
(see Fig.~\ref{fig:harmonic} and Section~\ref{par:harmonicMeasure}).

The key to the parallelization is the possibility to compute recursively a point with a given phase on a level line.
One uses discrete iso-angle trajectories~\eqref{eq:NewtonFlowLyapunov}, which are good approximations of
the external rays (Figure~\ref{fig:externalRays}).
A point on $\lev_{\lambda',N}(p_n)$ with a given phase can be obtained from a corresponding point
in $\lev_{\lambda,N/2}(p_{n-1})$ with $\lambda'=\lambda^2$ using only a few Newton steps.
If $\lambda$ is large enough, one can be confident that no skip modulo $2\pi$ has occurred so the phase is correct.
Again by the Newton method, we obtain $\lev_{\lambda,N}(p_n)$ from $\lev_{\lambda',N}(p_n)$.
This process is explained in details in Section~\ref{par:levAlgoMandelSpecifics}.
Use the app \texttt{levelSets} to initiate the low-resolution level sets that will be used by each parallel task.

\medskip

The second stage consist in merging the roots found by all the parallel tasks.
Sorting tera-bytes of data split in thousands of files of about one gigabyte each and deleting the roots found multiple times
appears, at first, as a substantial computational challenge. However, as explained in Section~\ref{par:dataStruct},
we used an appropriate, custom made, data structure that allows for logarithmic search and insertion times and
prohibits duplicates. Overall, the time requirements for the second stage turned out to be negligible
(see Table~\ref{table:overallTime}).
The corresponding apps are \texttt{hypRawCount} and \texttt{misRawCount}. They can also generate \texttt{maps},
which are low-resolution portraits of the Mandelbrot set that count how many roots of a given
type end up in a given pixel (see Fig.~\ref{fig:harmonic}).

\begin{figure}[H]
\captionsetup{width=.9\linewidth}
\begin{center}
\includegraphics[width=0.8\textwidth]{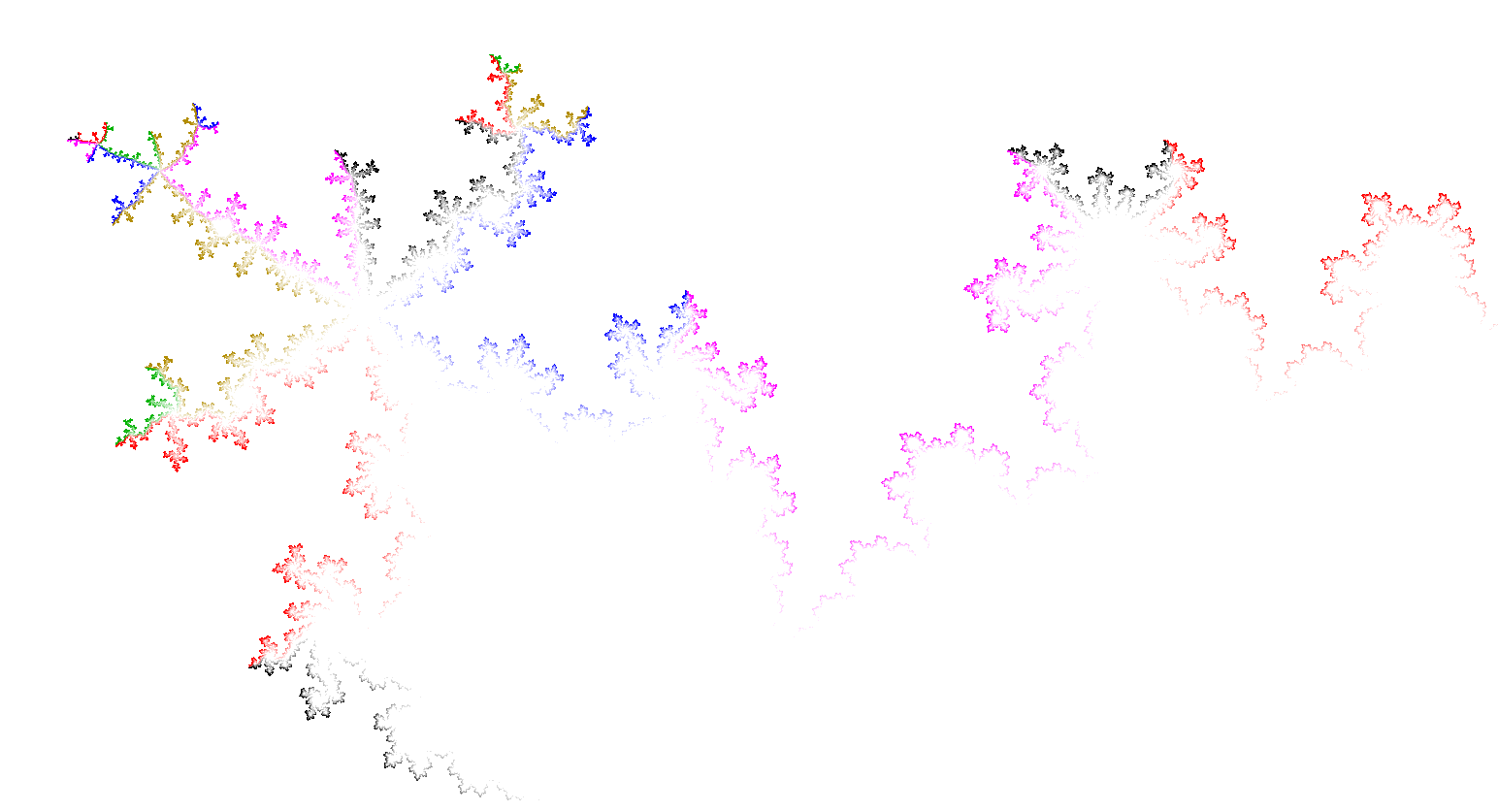}
\caption{\label{fig:rawJobs}\small
Zoom on 18 consecutive raw jobs for $p_{35}$ : the roots found by each parallel task are colored differently from
the next one. Note the extreme spatial variability between tasks, even though each one starts from a piece of
the discrete level line that contains $2^{27}$ points.
}
\end{center}
\end{figure}

\begin{figure}[H]
\captionsetup{width=.95\linewidth}
\begin{center}
\includegraphics[width=.7\textwidth]{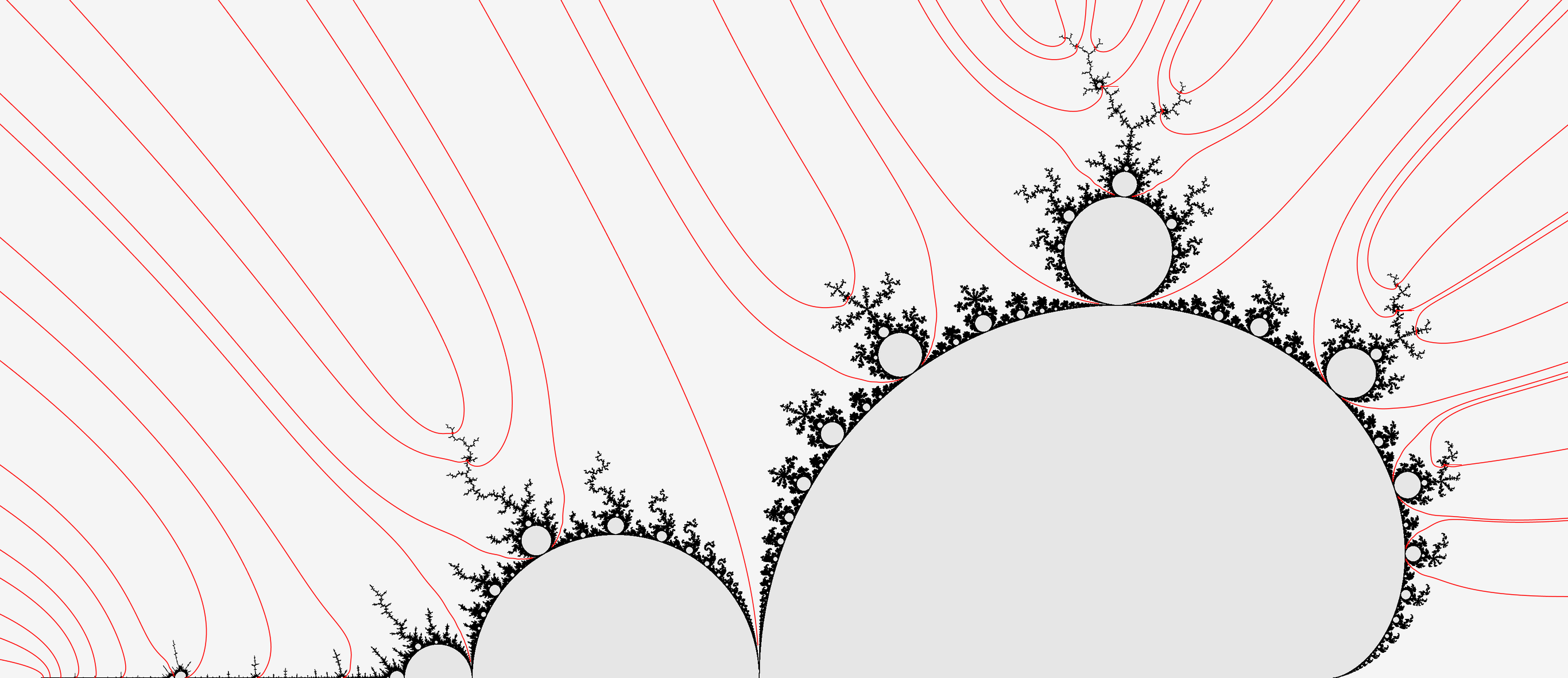}
\caption{\label{fig:externalRays}\small
External rays landing at the root of hyperbolic components of periods up to 6.
}
\end{center}
\end{figure}

\medskip

The third stage is the a-posteriori certification of the data, as explained in Section~\ref{par:certif2}.
It is split in two sub-stages.
First, we read the files in batch and count that we have all roots in a strictly increasing lexicographic order.
The corresponding apps are \texttt{hypCount} and \texttt{misCount}.
For reliability reasons, this task is performed on a a single CPU and the code is as
independent as possible from the rest of the library.
Next (and last), the certification involves redoing all the proofs in high-precision disk arithmetic using
as many CPUs as possible, which is the role of the apps \texttt{hypProve} and \texttt{misProve}.

\medskip
In practice, the main difficulty for the certification of $c\in\hyp(n)$ is a close transit of $p_k(c)$ near zero for some $k\vert n$.
This phenomenon can hinder the certification of a root of $p_n$ and prevents essentially any certification attempt made
with interval arithmetic. On the other hand, disk arithmetic (see Section~\ref{par:controlNum})
is essentially immune to this problem.
Once the proper certification radius (see Section~\ref{par:certif2}) was identified, the process went on flawlessly.

For $\mis(\ell,n)$, we encountered difficulties in the raw search that kept missing pairs of particularly close parameters near 
the tips of some antennas (see Fig.~\ref{fig:rootsHardToFind}).
However, adjusting the parameters in Theorem~\ref{thm:certif} solved the issue.

\medskip
The overall computation times of all stages are given in Table~\ref{table:overallTime}.
As the certification has an overall $O(d \log^2 d)$ complexity, it can be used as a reference time\footnote{$\log^2 d$ as
 we check that new parameters are not roots of divisors.}.
The ratio of 2.4 (for $p_n$) or 9.4 (for $q_{\ell,n}$) between the time necessary for the raw search and
that necessary for the certification attests that our algorithm behaves in practice as $O(d \log^2 d)$
with a small constant,
at least for the $p_n$ and $q_{\ell,n}$ families.
The higher raw search time for Misiurewicz parameters is mostly driven by the fact that,
asymptotically, half the roots of $q_{\ell,n}$ are high-multiplicity hyperbolic divisors that are discarded
in the end. Note that we did not used the simplified polynomials $m_{\ell,n}$, nor $s_{\ell,n}=p_{\ell+n-1}+p_{\ell-1}$.
We can therefore claim that, in practice, the presence of high multiplicities roots slows down the search time with our
level-line algorithm by a mere factor 1.95.

\begin{table}[H]
\captionsetup{width=.95\linewidth}
\begin{center}
\begin{tabular}{c||c|c||c|c||c|c}\cline{2-7}
\multicolumn{1}{c}{} & \multicolumn{2}{c||}{Root finding} & \multicolumn{2}{c||}{Certification} & \multirow{2}{*}{Total} & \small Find/certif. \\\cline{2-5}
\multicolumn{1}{c}{} & \small Parallel  jobs & \small Merge & \small Counting & \small Proofs &  & \small ratio \\\hline
\multirow{3}{*}{\parbox{11ex}{\centering\small Hyperbolic $p_{n,\enspace n\leq 41}$} } 
& 26.2 y 
& 16.1 d 
& 1.7 d 
& 11 y 
& \multirow{2}{*}{37.2 y}
& \multirow{3}{*}{2.4}\\
& 229 kh
&  386 h
&  41 h
&  96.4 kh
& \multirow{2}{*}{\bf $\sim$326 kh}\\ 
& 70.3\%
& 0.1\%
& 0.01\%
& 29.6\%
&
\\\hline
\multirow{3}{*}{\parbox{11ex}{\centering\small Misiurewicz $q_{\ell,n,\enspace\ell+n\leq 35}$} } 
& 40.9 y 
& 9.1 d 
& 1.1 d 
& 4.4 y 
& \multirow{2}{*}{45.3 y} 
& \multirow{3}{*}{9.4}\\ 
& 358 kh
& 218 h
& 26 h
& 38.2 kh 
& \multirow{2}{*}{\bf  $\sim$397 kh}\\ 
& 89.4\%
& 0.05\%
& 0.006\%
& 10.6\%
&
\\\hline
\end{tabular}
\caption{\label{table:overallTime}\small
Overal computation time of each main stage.
See Footnote~\ref{timecore} p.~\pageref{timecore} for a definition of time-core.
Note that for Misiurewicz polynomials, we did not use the simplified version based on $m_{\ell,n}$
in order to measure the effects of high-multiplicities (about 50\% of the roots for $q_{2,n}$)
on our level-line algorithm.
}
\end{center}
\end{table}

\noindent
\begin{minipage}{0.45\textwidth}
\begin{figure}[H]
\captionsetup{width=\linewidth}
\begin{center}
\includegraphics[width=\textwidth]{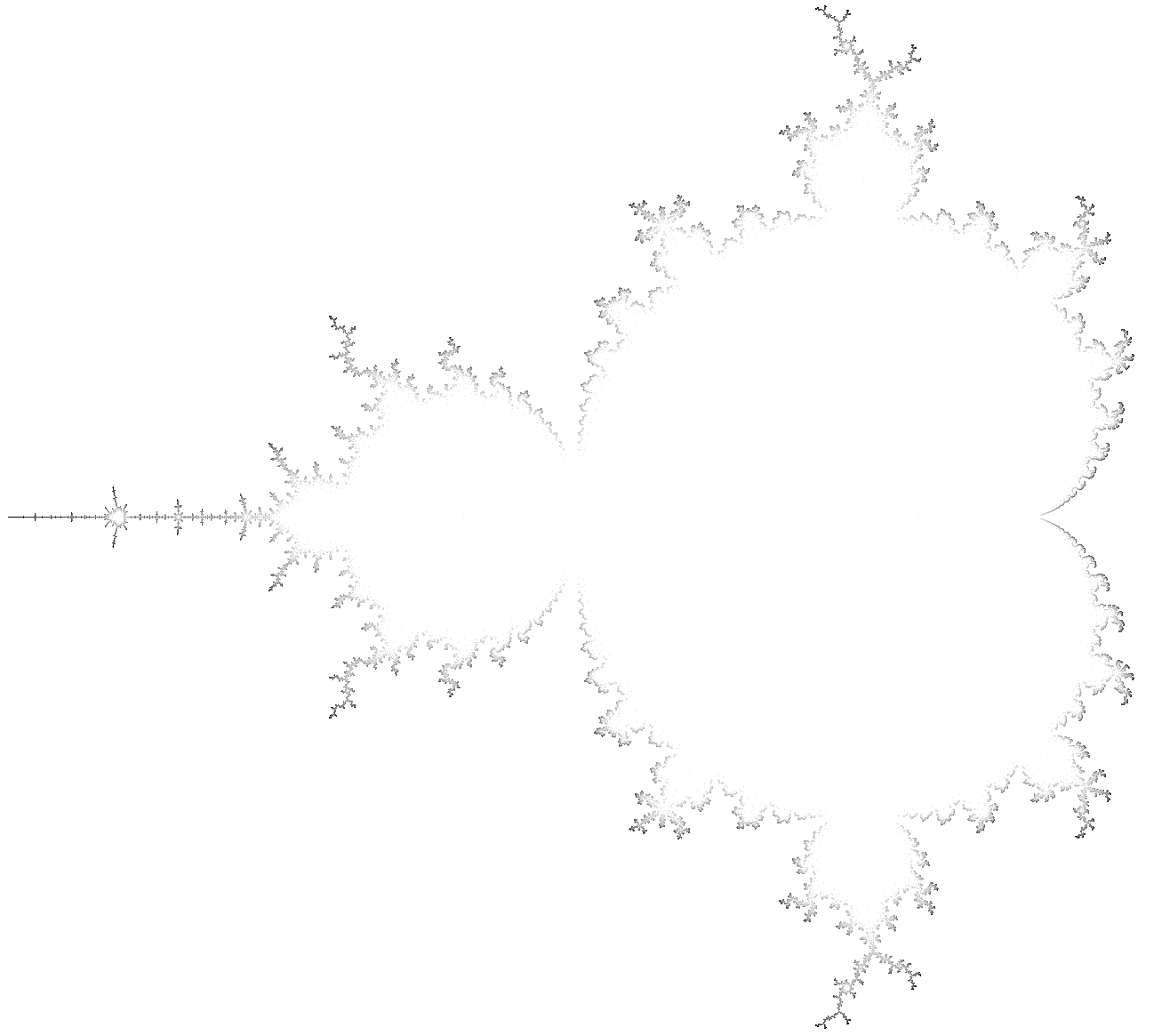}
\caption{\label{fig:harmonic}\small
The harmonic measure of $\M^c$ can be visualized by counting how many hyperbolic centers land on each pixel
and adjusting the shade accordingly (which is the principle of our \texttt{map} files). Here we represent $\hyp(n)$ for $n\leq 24$.
Note the difference with Figure~\ref{fig:mandelbrot} as the inner creeks are visited less frequently than antennaes.
}
\end{center}
\end{figure}
\end{minipage}\hfill\rule[-30ex]{0.2pt}{61ex}\hfill
\begin{minipage}{0.45\textwidth}
\begin{figure}[H]
\captionsetup{width=\linewidth}
\begin{center}
\includegraphics[width=\textwidth]{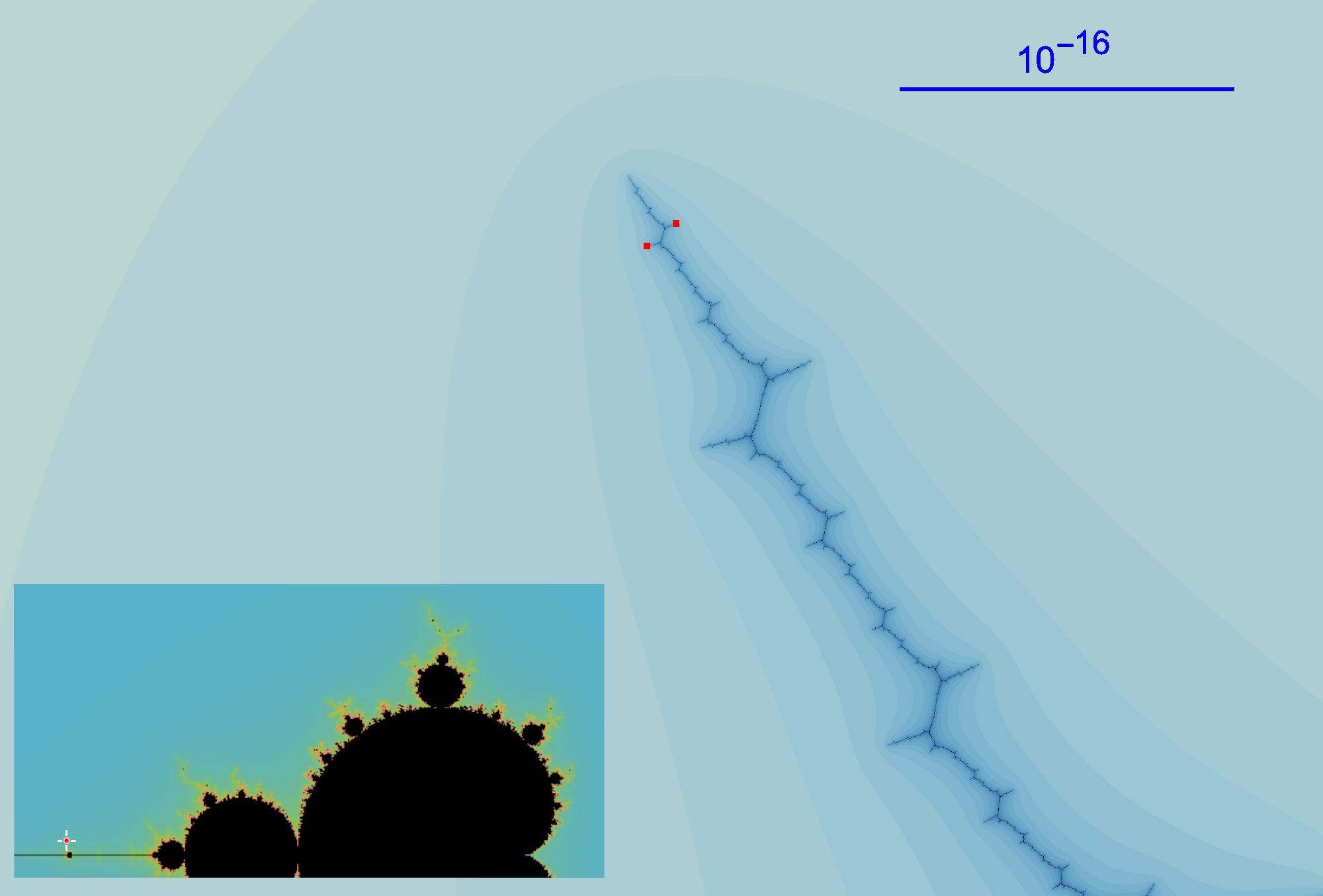}
\caption{\label{fig:rootsHardToFind}\small
Typical example of a pair of $\mis(7,27)$ parameters that are particularly easy to miss during the raw search.
}
\end{center}
\end{figure}
\end{minipage}

\subsubsection{Number of Newton steps}\label{par:BenchNewtonSteps}

In this subsection, we present some statistics regarding the practical complexity of our algorithm,
when applied to the splitting of~$p_n$  or $q_{\ell,n}$ when $n$ (resp. $\ell+n$) is large.
To normalize the statistics, we present them in terms of \textit{Newton steps per root} or in \textit{Newton steps per new parameter}.
Each computable action boils down, eventually, to a Newton steps or equivalent (evaluate
a polynomial expression, its derivative and a quotient). We have sorted the actions in categories defined
by what is actually being computed and we normalize the total number of Newton steps for each action by the degree
of the polynomial (steps per root) or by the cardinal of the set being computed (steps per new parameter).
For $\mis(\ell,n)$ sets, the steps per parameter is roughly double the number of steps per roots because,
asymptotically, only half the roots are kept in the set (see Section~\ref{par:scale}).

\medskip
\begin{table}[H]
\captionsetup{width=.9\linewidth}
\begin{center}\scalebox{.99}{\small
\begin{tabular}{c|c||ccc|cccc}\hline
\bf Period &\bf Level line & \multicolumn{3}{c|}{\bf Convergent descend} & \multicolumn{4}{c}{\bf Discarded descend} \\[1pt]
$n$ & $\lev_{5, 8d}(p_n)$ & \% &  \texttt{FP80} & \texttt{FP128} & Repeat & Overlap & Divisor & Divergent\\
\hline
28 & 51.6 N/r & 26\% &  8.8 N/p & 2.4 N/p & 73\% & -          & 0.01\% & 1\% \\
30 & 50.7 N/r & 25\% & 8.7 N/p & 2.5 N/p & 73\% & 0.07\% & 0.00\% & 1\%\\
35 & 47.9 N/r & 25\% & 8.4 N/p & 2.8 N/p & 73\% & 0.23\% & 0.00\% &1\%\\
39 & 46.9 N/r & 25\% & 8.1 N/p & 3.1 N/p & 73\% & 0.38\% & 0.00\% &1\%\\
40 & 47.2 N/r & 25\% & 7.9 N/p & 3.2 N/p & 73\% & 0.39\%  & 0.00\% &1\%\\
41 &  47.7 N/r & 25\% & 7.7 N/p & 3.4 N/p & 73\% & 0.34\% & 0.00\% &1\%\\\hline
\end{tabular}
}
\caption{\label{table:newtonSteps}\small
Count of Newton steps per new hyperbolic parameter (N/p), which, for $p_n$, is equivalent to Newton steps per root (N/r).
Intermediary periods have similar statistics.
}
\end{center}
\end{table}

\begin{table}[H]
\captionsetup{width=.9\linewidth}
\begin{center}\scalebox{0.91}{\small
\begin{tabular}{c|cc||ccc|cccc}\hline
\bf Type &\multicolumn{2}{c||}{\bf Level line} & \multicolumn{3}{c|}{\bf Convergent descend} & \multicolumn{4}{c}{\bf Discarded descend} \\[1pt]
$\ell+n$ & \multicolumn{2}{c||}{$\lev_{100, 8d}(q_{\ell,n})$} & 
\% & \texttt{FP80} & \texttt{FP128} & Repeat & Overlap & Divisor & Divergent\\
\hline
28 &  110 N/p & 56.8 N/r & 13\% &  9 N/p & 2.5 N/p & 33\% & 0.02\% & 48.1\% & 6\% \\
30 &  107 N/p & 55.4 N/r & 13\% &  8.9 N/p & 2.6 N/p & 34\% & 0.01\% & 48.1\% & 6\%\\
32 &  104 N/p & 53.9 N/r & 12\% &  8.6 N/p & 2.9 N/p & 34\% & 0.00\% & 48.1\% & 6\%\\\hline
33 &  120 N/p & 58.8 N/r & 13\% &  10.8 N/p & 4.3 N/p & 35\% & 0.01\% & 49.2\% & 4\%\\
35 &  131 N/p & 64.4 N/r & 13\% &  10.5 N/p &  4.8 N/p & 35\% & 0.02\% & 49.2\% & 4\%\\\hline
\end{tabular}
}
\caption{\label{table:newtonStepsMis}\small
Count of Newton steps per new Misiurewicz parameter (N/p) and Newton steps per root (N/r).
The normalization is substantially affected by the deflation due to roots of divisors. 
In this benchmark, we did not simplify $q_{\ell,n}$ in order to estimate the effect of high-multiplicities on our algorithm.
}
\end{center}
\end{table}

First, we compute the finest level line $\lev_{\lambda_0, Md}$
with $\lambda_0=5$, $M=8$, $d=2^{n-1}$ for $p_n$
and  $\lambda_0=100$, $M=8$ and $d=2^{\ell+n-1}$ for $q_{\ell,n}$.
According to Tables~\ref{table:newtonSteps} and~\ref{table:newtonStepsMis},
this phase appears to have a constant cost of about 50 Newton steps per root, regardless of the degree of the polynomial.
This indicate that the computation of the level line is essentially not affected by the presence of roots with high multiplicities.
The only exception occurs for $m_{\ell,n}$ and $\ell+n\geq 33$ where a significant jump in the number of steps per roots
is observed; it suggests that the level $\lambda_0=100$ may be slightly too low for those types (higher level curves
are easier to compute).
Note that, for the computation of $\lev_{\lambda_0, Md}$, we do not distinguish here between steps taken in hardware or
in high-precision arithmetic (see \S\ref{par:BenchNewtonSteps2}).

\medskip
Next, we perform Newton descends that we initiate on the subset $\lev_{\lambda_0, 4d}$ (\ie we have~4 starting points per root).
Each Newton descend can either end up as convergent towards a new parameter
(which is kept in the current \texttt{raw*.nset} file)
or it can be discarded for one the following reasons:
\begin{itemize}
\item the descend converges towards a root that was already found in the current job,
\item the descend converges towards a root that overlaps a neighboring job (early detection of duplicates),
\item the descend converges towards a root of a divisor,
\item the descend is considered divergent and is  interrupted, either because of a long-distance jump
(indicating the likely proximity of a critical point) or because it did not converge in the allotted number of steps (timeout).
\end{itemize}
Again, we observe consistently that convergent descends take about 10-12 steps, most of which can be performed entirely
using hardware arithmetic. If one assumes  that the level line $p_n(z)=10^{-5}$ is totally disconnected (\ie it has one connected component in each contracting Newton disk),
the general considerations of Section~\ref{par:split} thus ensure a practical upper bound of $\log 5\times 10^{5} \sim 13$ Newton
steps (up to some universal constant) during the descend from the level line $p_n(z)=5$.
This prediction is conform with our observations.

\medskip
We have consistently observed that about 5 to 7 \% of the roots of $p_n$ are real ones.
For $m_{\ell,n}$, the proportion is asymptotically the same for large $\ell+n$, even though
the proportion starts at a higher level (9 to 15\% for $10\leq \ell+n\leq 20$).

\subsubsection{Usage of high precision arithmetic}\label{par:BenchNewtonSteps2}

\begin{figure}[H]
\captionsetup{width=\linewidth}
\begin{center}
\includegraphics[width=0.9\textwidth]{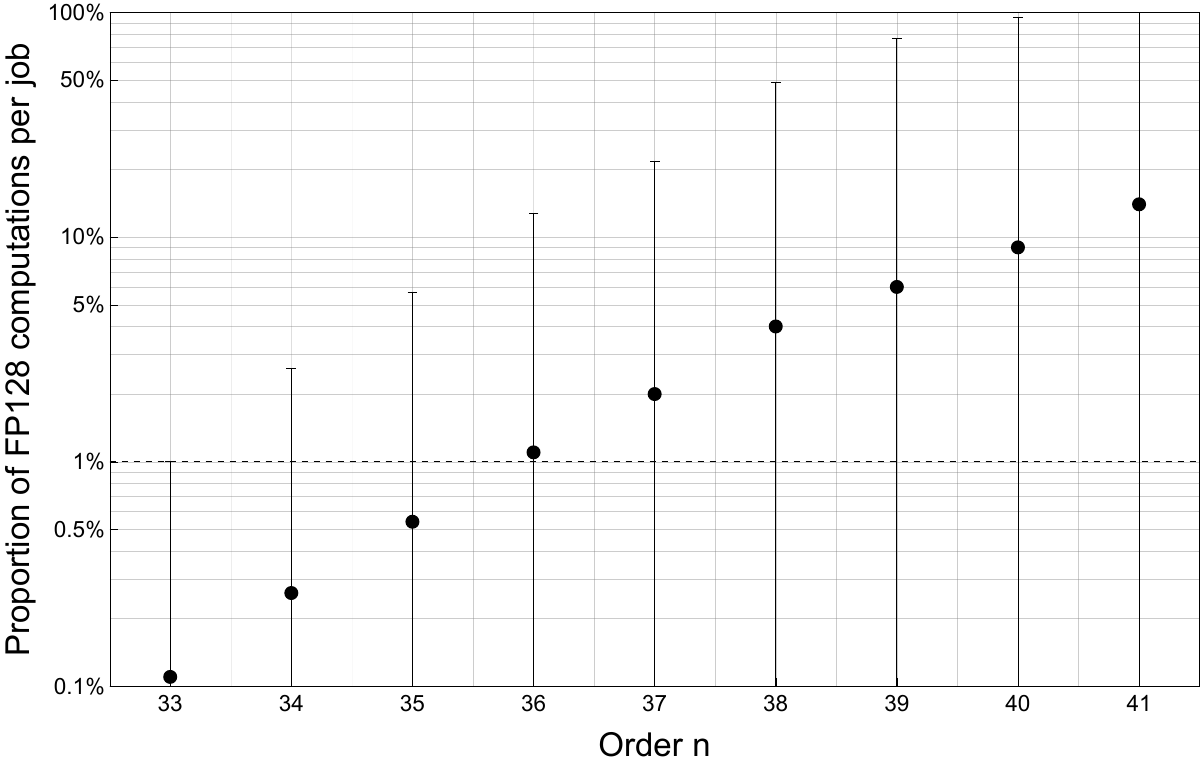}
\caption{\label{fig:mpfrRatio}\small
Proportion of the computations for splitting $p_n$ that were performed
using high-precision arithmetic operations \citelib{MPFR}. The black dot represents the average value. The error bars 
represent the variability between all raw search jobs.
}
\end{center}
\end{figure}

Tables~\ref{table:newtonSteps} and~\ref{table:newtonStepsMis} indicate that the total number of Newton steps
for each stage of our algorithm is relatively stable and scales linearly with the degree of the polynomial.
However, we have observed an extreme variability in job durations, even though each job is calibrated to deal
with the same number of points on the level line. It turns out that the duration is mostly determined by the amount
of computations of the level line that need to be performed in high-precision arithmetic, with the \texttt{MPFR} library.

For our implementation, a practical rule of thumb is that the duration,
in hours, of a large raw search job is about $14\times R +10$ where $R$ is the
ratio of computations that require high-precision. This ratio is illustrated on Figure~\ref{fig:mpfrRatio}.
For period 33, most level lines can be completed using hardware arithmetic. This proportion drops to 86\% on average
for splitting $p_{41}$, with some jobs (typically near the left-most tip, see Section~\ref{par:HardwareLimit})
requiring almost exclusively high-precision computations.

Note that, on the other hand, the proportion of \texttt{MPFR} computations in the Newton descend is almost constant.

\bigskip
Overall, the benchmarks presented in \S\ref{par:BenchNewtonSteps}-\ref{par:BenchNewtonSteps2}
confirm that, even in the presence of roots with high multiplicities, our algorithm based on level lines
is really robust and performs well. Its practical complexity remains $O(d)$, even when $d$ is pushed to the
tera-scale. Those observations concern the family of polynomials related to the Mandelbrot set
and we do not claim nor suggest that those exceptional performances will remain in full generality.

\subsubsection{Memory and hardware requirements}

The raw search of roots was split into many independent jobs.
Jobs where dimensioned so that they would, on average, correspond to 1.1GB of data
with a few exceptions near $z=-2$ where the job size can raise up to a little more than 2GB.
Two reasons prevailed in favor of that choice.

First, the memory requirement for each job is fairly limited: each job compares the
new roots with the previous one (if pressent) so both are loaded in memory and eat up 4GB,
plus the internal data structure can take up another 2GB. Overall, we were able to run all our
computations while requiring only 7GB of RAM per CPU.
This limit happens to roughly coincide with the average
of RAM available on our cluster, when the total RAM of each computing node is equally divided
among all the cores of the node.  The practical consequence was that the scheduler was usually
able to accommodate our requests for a high number of parallel cores, which would not have
been the case if a few CPU would have eaten up the whole memory of a node.

Second, writing the final data back into a few files of 1TB did not strike us a good idea, because
few people have access to resources that allow them to handle such a monster file reliably,
while transferts of 2G files are common.

\bigskip
The second stage (counting that we have all roots and write the final files without duplicates)
is performed on a single CPU and requires as much memory as possible. 
For hyperbolic periods 37 and up, we used 175GB of RAM which allowed most sort operations
to be performed without reloading the same file twice. 

Our implementation can accommodate tight memory limitations, though the time required to sort
and merge the entire database may improve significantly if multiple reloads are necessary.
For example, for Misiurewicz types of order $\ell+n\geq 32$, we had to reduce the memory per core to
accommodate the growing number of subtypes\footnote{For $\mis(\ell,n)$ with $\ell+n=33$, we requested only 26GB for each of the
31 cores dedicated to the count, instead of 73GB for the single core used for $\hyp(33)$. Consequently, each file had to be reloaded about~2.5 times.}

Note that the loading time (waiting for \texttt{fopen} system calls) can be substantial.
In our case, depending on the ambient load on the file system, the loading
times ranged from 18\% to 98\% of the total computation time for the second stage.
In some cases, we where able to preload the caches with the upcoming files
by running a side process that would monitor the logs and \texttt{cat file > /dev/null}
while the previous files where being processed.

\bigskip
The third stage (certification)  requires very little memory (no more than twice size of a
final \texttt{nset} file) so 5GB per CPU dedicated to a certification task is sufficient.

\subsubsection{The jobs market}

The raw search stage is embarrassingly parallel. However, the duration of the jobs presents large
variability even for a given hyperbolic period or a Misiurewicz sub-type: it can range from half the average
compute time up to 8 times the average duration of the jobs that handle the same polynomial. This
heterogeneity calls for a vertical parallelism, where each compute node calls for a new task as soon
as it becomes available, asynchronously from the other nodes that remain busy.

\medskip

The human task of submitting new jobs is also non trivial: to generate our database, even without a single failure,
49\,427 raw jobs are necessary (about 2/3 of them for $\hyp(n)$), followed by 1202 jobs for the later phases (most
of them for $\mis(\ell,n)$).
Identifying the raw jobs that have crashed due to node failure, network saturation, disk lag
(which happened to be a common plague on our busy HPC center) or job timeout can easily
turn into a nightmare.

\medskip

To optimise our use of HPC resources among heterogeneous jobs and automate the task of submitting raw jobs,
we developed our own task scheduler, which is a small \texttt{Java} app, called \textit{job market}.
It records the progress of all computing tasks, distributes new tasks asynchronously to each available
computing node, recycles failed jobs and produces useful statistics.

\subsubsection{Use a GPU or not?}

In order to keep our implementation as simple (and human readable) as possible, we decided against
a mixed code GPU/CPU. However, we may revisit this position in the future. Here are a few thoughts about
this choice.

\medskip

For massively parallel computations, GPUs offers impressive performances.
We briefly explored this path; however, our algorithm reaches the limits of hardware
arithmetic very quickly (see \S\ref{par:HardwareLimit}). For now, GPUs are
essentially limited to FP64 arithmetic, which gives a first objection against their usage
for splitting polynomials of high degree, pending a library that can emulate high-precision arithmetic.

The second objection is that the computational gain is not completely obvious
because it is not easy to predict how long the dynamics
of the Newton map will take before the algorithm can classify the starting point.
It is silly to keep running Newton steps just because other threads are still doing it; and it is equally
counterproductive to cut an iteration short just because the others are done. Considering that the main
gain of GPUs lies in running identical parallel threads, a naive implementation may therefore tend
to impose the worst case as the statistical norm, thus defeating the purpose of running on a GPU.


\appendix
\bigskip\bigskip\par\noindent
\textbf{\huge Appendices}


\section{Proof of the factorization theorem}\label{par:factorization}

Let us give here a direct proof of the factorization theorem (\ie Theorem~\ref{thm:fact_mis}).

\begin{proof}
As mentioned above, thanks to \cite{HT15}, only the multiplicity of the hyperbolic factors has to be established.
For $\ell\in\{0,1\}$, the formula~\eqref{eq:MisFactor} boils down to~\eqref{eq:hypCase} so
we will suppose $\ell\geq 2$ from now on.

\medskip
The case of $\hyp(1)=\{0\}$
is essentially based on the well known fact (proven by direct induction)
that the trailing coefficients of $p_n(z)$ stabilize as $n$ grows; more precisely:
\begin{equation}
\forall k\in\N, \qquad p_n(z) \equiv p_{n+k}(z) \mod z^{n+1},
\end{equation}
thus $q_{\ell,n}(z)\equiv 0 \mod z^{\ell+1}$. On the other hand, one can check that
\begin{equation}
q_{\ell,n}(z)\equiv 2^{\ell-1} z^{\ell+1} \mod z^{\ell+2}
\end{equation}
so zero is exactly of multiplicity $\ell+1 = \eta_\ell(1)$.

\medskip
Next, let us check that the multiplicities in~\eqref{eq:MisFactor} are consistent with $\deg q_{\ell,n}$, \ie:
\begin{equation}\label{eq:MisDeg}
\sum_{k\vert n} \left( \eta_\ell(k)\left|\hyp(k)\right| + \sum_{j=2}^{\ell} \left|\mis(j,k)\right| \right) = 2^{\ell+n-1}.
\end{equation}
Indeed, using~\eqref{eq:HypCount} and~\eqref{eq:MisCount} and a geometric sum, the left-hand side equals
\begin{align*}
\sum_{k\vert n} & \left|\hyp(k)\right|\left( \eta_\ell(k) + \sum_{j=2}^{\ell} \Phi(j,k) \right) \\ & =
\sum_{k\vert n} \left(\sum_{m\vert k} \mu(k/m)2^{m-1}\right)
\left( 2^\ell + \left\lfloor \frac{\ell -1}{k} \right\rfloor - \sum_{j=2}^{\ell}\delta_{k\vert j-1} \right)
\end{align*}
where $\delta_{k\vert j-1} =1$ if $k$ divides $j-1$ and $0$ otherwise.
For any integers $\lambda, k \in \N^\ast$, let us observe that 
\begin{equation}
\left\lfloor \frac{\lambda}{k} \right\rfloor = 
\left| \left\{ j\in\ii{1}{\lambda} \,;\,  k \vert j\right\} \right|
= \sum_{j=1}^{\lambda}\delta_{k\vert j}.
\end{equation}
The claim~\eqref{eq:MisDeg} thus follows from Möbius inversion formula \cite{Mob1832}-\cite{Rota1963}:
\begin{equation}
\sum_{k\vert n} \left(\sum_{m\vert k} \mu(k/m)2^{m}\right)= 2^{n}.
\end{equation}

\medskip
Let us now consider the case of $\hyp(k)$ for $k\vert n$ and $1<k\leq n$.
Note that if $n$ is prime, there is only one hyperbolic factor left, namely $k=n$ and
the factorization~\eqref{eq:MisFactor} follows from an argument of
divisibility and the identity of degrees~\eqref{eq:MisDeg}.
We can however treat the general case in a unified way, regardless of wether $n$ is composite or not
and prove~\eqref{eq:MisFactor} by recurrence on $\ell$. One has
\begin{align}
q_{\ell,n}(z)
&= p_{\ell+n}(z) - p_\ell(z)  \notag\\
&= p_{\ell+n-1}^2(z) -p_{\ell-1}^2(z)  \notag\\
&= q_{\ell-1,n}(z) (p_{\ell+n-1}(z)+p_{\ell-1}(z)) \label{eq:factorDiffSquares}\\
&= q_{\ell-1,n}(z) (q_{\ell-1,n}(z)+2p_{\ell-1}(z)). \notag
\end{align}
The recurrence assumption reads
\begin{equation}\label{eq:proofA}
q_{\ell-1,n} = \prod_{k\vert n} \left( h_k^{\eta_{\ell-1}(k)} \prod_{j=2}^{\ell-1} m_{j,k} \right)
\end{equation}
and, from the hyperbolic case, we know that
\[
p_{\ell-1} = \prod_{j\vert \ell-1} h_j.
\]
If $k$ divides $n$ but not $\ell-1$, then $h_k$ divides $q_{\ell-1,n}(z)$ but not $p_{\ell-1}$ so
$h_k$ does not divide $q_{\ell-1,n}+2p_{\ell-1}$.
Alternatively, if $k$ divides both $n$ and $\ell-1$ then $h_k$ divides $q_{\ell-1,n}+2p_{\ell-1}$.
However $h_k^2$ does not because it is a factor of $q_{\ell-1,n}$ (see the Lemma~\ref{LM} below)
but  not of~$p_{\ell-1}$.
Finally, the polynomials $m_{\ell,k}$ for $k\vert n$ are known factors of $q_{\ell,n}$ that do not divide $q_{\ell-1,n}$;
as $q_{\ell,n}/q_{\ell-1,n}$ has simple roots, they are simple factors of $q_{\ell,n}$. There are no other
Misiurewicz-type factors. We have thus established that:
\begin{equation}\label{eq:proofB}
q_{\ell,n} =   \left(\prod_{k\vert \operatorname{gcd}(n,\ell-1)} h_k \right)
\left(\prod_{k\vert n} m_{\ell,k} \right) q_{\ell-1,n}.
\end{equation}

Let us finally observe that, for $\ell, k\geq2$:
\begin{equation}\label{eq:recEta}
\eta_{\ell-1}(k) =
\begin{cases}
\eta_\ell(k)-1 & \text{if }k\vert \ell-1,\\
\eta_\ell(k) & \text{else.}
\end{cases}
\end{equation}
Indeed, if $k\vert\ell-1$ then \eqref{eq:recEta} follows directly from~\eqref{eq:defMultHyp};
conversely, if $k$ does not divide $\ell-1$ then $\frac{\ell-1}{k}\in \Z+[\frac{1}{k},1)$ and
\[
\eta_{\ell-1}(k) =\left\lfloor \frac{\ell-1}{k} - \frac{1}{k} \right\rfloor + 2
=\left\lfloor \frac{\ell-1}{k} \right\rfloor + 2 = \eta_\ell(k).
\]
Combining~\eqref{eq:proofA}, \eqref{eq:proofB} and~\eqref{eq:recEta}
ensures that~\eqref{eq:MisFactor} holds for the next pre-period.
\end{proof}

\begin{remark}
The previous proof establishes the following identity (see also~\eqref{eq:polySimple}):
\begin{equation}\label{eq:polyMisSimple}
s_{\ell,n}=\frac{q_{\ell,n}}{q_{\ell-1,n}} = p_{\ell+n-1}+p_{\ell-1} = \left(\prod_{k\vert \operatorname{gcd}(n,\ell-1)} h_k \right)
\left(\prod_{k\vert n} m_{\ell,k} \right).
\end{equation}
\end{remark}

To illustrate our method, the simplest case of a composite period is:
\[
q_{2,4}(z) = p_6(z) - p_2(z) = p_5^2(z) -p_1^2(z) = q_{1,4}(z) (p_5(z)+p_1(z)) = p_4^2(z) (p_4^2(z)+2z).
\]
As $p_4 = h_4 h_2 h_1$, the polynomial
$q_{2,4}$ is divisible by $h_4^2$ but not by $h_4^3$. The exponent of $\hyp(4)$ is therefore exactly 2
and we are left with a single hyperbolic factor, namely~$\hyp(2)$. Using the degree identity~\eqref{eq:MisDeg}
thus ensures that
\[
q_{2,4}(z) = h_1^3(z)h_2^2(z)h_4^2(z) m_{2,1}(z) m_{2,2}(z) m_{2,4}(z).
\]

The only check-up left is the following statement.
\begin{lemma}\label{LM}
If $k\geq 1$ divides both $\ell\geq1$ and $n\geq1$, then $h_k^2$ is a factor of $q_{\ell,n}$.
\end{lemma}
\begin{proof}
It is sufficient to establish the lemma for $q_{k,jk}$ with $j\in\N$ because
\[
\forall r,q\in\N^\ast, \quad
q_{kr,kq}=p_{k(q+r)}-p_{k}-p_{kr}+ p_{k} = q_{k,k(q+r-1)}-q_{k,k(r-1)}.
\]
As $q_{k,0}=0$, let us assume that $j\geq 1$; using~\eqref{eq:factorDiffSquares} repeatedly one has.
\begin{align*}
q_{k,jk}
&= q_{k-1,jk} \times (p_{(j+1)k-1}+p_{k-1}) \\
&= q_{k-2,jk} \times (p_{(j+1)k-2}+p_{k-2})  (p_{(j+1)k-1}+p_{k-1}) \\
&= \ldots \\
&= q_{1,jk} \times (p_{(j+1)k-2}+p_{k-2})  (p_{(j+1)k-1}+p_{k-1}) \ldots (p_{jk+1}+p_{1}) 
\end{align*}
Therefore, $p_{jk}^2=q_{1,jk}$ is a factor of $q_{k,jk}$. As $k\vert jk$, then $h_k$ is itself a factor of $p_{jk}$
and so $h_k^2$ divides $p_{jk}^2$ and consequently  $q_{k,jk}$ too.
\end{proof}


\section{Useful mathematical results}

\label{par:mathBackground}
We expect our readers to have (like us authors) varied mathematical backgrounds.
For the convenience of all, we recall here briefly a few mathematical results that are quite standard to
the specialists, but might not be well known to all.

\subsection{Localization of the roots of polynomials}

To renormalize the roots of $P\in\C[z]$ within $\disk(0,1)$ one may consider $P(z/\rho)$ where
$\rho>r$ and $r$ is given by the following statement.
\begin{thm}[{\cite[Theorem~1.2.4]{BORERD}}]\label{thm:rootLoc}
All the zeros of $P(z)=\sum\limits_{k=0}^n a_k z^k$ are located in the disk $\overline{\disk(0,r)}$ where
\[
r = \inf_{\frac{1}{p}+\frac{1}{q}=1} \left\{ 1+\left(\sum_{j=0}^{n-1} \frac{|a_j|^p}{|a_n|^p} \right)^{q/p} \right\}^{1/q}
\]
and provided $a_n\neq0$ and $p,q>1$.
\end{thm}
\begin{remark}
For various classical results on this mater, see~\cite[\S6.4]{HEN1974}.
\end{remark}

Small perturbations of analytic functions do not affect much the location of their roots,
as stated by Rouch\'{e}'s theorem, stated here in its strong symmetric form.

\begin{thm}[Rouch\'{e}-Estermann]\label{thm:rouche}
Given two holomorphic functions $f$, $g$ on a domain $G$ and a bounded region $K\subset G$
with continuous boundary $\partial K$.
If the following strict inequality holds:
\[
\forall z\in\partial K, \qquad |f(z)-g(z)| < |f(z)| + |g(z)|
\]
then $f$ and $g$ have the same number of roots (counted with multiplicity) in $K$.
\end{thm}\noindent%
A common case of application is when $|h(z)|<|f(z)|$ on $\partial K$. In that case, the theorem can be applied with $g=f+h$ and ensures
that the perturbation $h$ does not change the number of roots of $f$.

\subsection{Injectivity of analytic functions}

The following result quantifies the local injectivity of analytic functions (see~\cite{CARLESON}).

\begin{thm}[Koebe's lemma]\label{thm:kobe}
If $f:\disk(0,1)\to\C$ is an injective analytic function, then
$f(\disk(0,1)) \supset \disk(f(0), |f'(0)|/4)$ and, more generally:
\begin{equation}\label{eq:kobe}
\forall r\in [0,1),\qquad
\disk\left( f(0), \frac{r |f'(0)|}{(1+r)^2} \right) \subset f(\disk(0,r)) \subset \disk\left( f(0), \frac{r |f'(0)|}{(1-r)^2} \right) \,.
\end{equation}
\end{thm}\noindent%
The universality of the constant $1/4$ constitutes a radical breach
between real and complex analysis.
The optimality of this constant can be checked on the function
\[
f(z)=\frac{z}{(1-z)^2}
\]
because $f'(0)=1$ but the value $f(-1)=-1/4\not\in f(\disk(0,1))$.

\medskip
We also use of the following stronger result, based on the Bieberbach conjecture,
which was proven by de Branges \cite{Bra85}. 

\begin{thm}\label{thm:target}
Given $g:\disk(0,1)\to\C$ an injective analytic function with $g(0)=0$ and $g'(0)=1$
and $f \defequal g^{-1}$ its inverse,
there exist universal constants $\rho_1 \simeq 0.121507$ and $\beta_1\simeq 0.0987043$ such that $g(\disk(0,\beta_1))\subset \disk(0,\rho_1)$, and $\disk(0,\rho_1)$ is contained in the immediate basin of attraction of $0$ for the Newton map~$N_f$.
\end{thm}
\begin{proof}
Let $\Omega \defequal g(\disk(0,1))$, $z \in \Omega$ and $y \defequal f(z)$. Then, combining~\eqref{eq:Newton}
and the chain rule for derivatives:
\[
|N_f(z)| = |g(y)-yg'(y)|.
\]
A direct computation with the expansion $g(y) = y + a_2 y^2 + a_3 y^3 + \ldots$ gives
\[
|g(y)-yg'(y)| \leq \sum_{n \geq 2} (n-1)|a_n| r^{n}
\quad\text{and}\quad
|g(y)|\geq r-\sum_{n\geq 2}|a_n|r^n
\,,
\]
where $r \defequal |y|$.
In order to ensure $|N_f(z)| < |z| = |g(y)|$, it is therefore enough that 
\[
\sum_{n \geq 2} n |a_n| r^{n} < r.
\]
By de Branges's theorem \cite{Bra85}, one has $|a_n|\leq n$ for all $n\geq2$.
Let $\beta_0$ be the unique positive solution of
\[
\frac{x(x^2-3x+4)}{(1-x)^3} = \sum_{n \geq 2} n^2 x^{n-1}=1,
\]
which satisfies that $\beta_0\simeq 0.164878 \in (164,165)/1000$. We have thus shown that $r < \beta_0$ \ie
$z\in g(\disk(0,\beta_0))$ ensures that $|N_f(z)| < |z|$.
To ensure this property for subsequent iterations, let us take $z\in\disk(0,\rho_1)$ with $\rho_1=\beta_0(1+\beta_0)^{-2}\simeq
0.121507$. Then~\eqref{eq:kobe} applied to $g$ guaranties that $\disk(0,\rho_1) \subset g(\disk(0,\beta_0))$
and the previous considerations give
\[
N_f(z) \in \disk(0,|z|)\subset \disk(0,\rho_1).
\]
Consequently, by Montel's theorem, the iterates of $N_f$ on $\disk(0,\rho_1)$ form a normal familly,
which converges towards zero, the only zero of $f$ in $\disk(0,\rho_1)$.
Finally, taking $\beta_1\simeq 0.0987043$ such that $\beta_1(1-\beta_1)^{-2}=\rho_1$, the estimate~\eqref{eq:kobe}
ensures that $g(\disk(0,\beta_1))\subset \disk(0,\rho_1)$ is also included in the immediate bassin of attraction of the origin for~$N_f$.
\end{proof}

\begin{corollary}\label{cor:targetQuad}
Under the assumptions of Theorem~\ref{thm:target}, 
there exists universal constants $\rho_2\simeq 0.110549$ and $\beta_2\simeq 0.091287$ such that
$g(\disk(0,\beta_2)) \subset \disk(0,\rho_2)$ and if $z\in\disk(0,\rho_2)$, then for any $\epsilon\in(0,1/8)$
and
\[
k\geq 3 + \log_2 |\log_2 \epsilon|,
\]
the iterate $N_f^k (z)\in\disk(0,\epsilon)$. If $\epsilon>1/8$, then $z$ is already in $\disk(0,\epsilon)$.
\end{corollary}
\proof
If $|z|<\rho_2$, then $r=|f(z)|\leq r_2$ s.t. $r_2(1+r_2)^{-2}=\rho_2$
and
\[
|N_f(z)| \leq \sum_{n \geq 2} n(n-1) r^{n} \leq 8 \Big(r - \sum_{n \geq 2} n r^{n}\Big)^2 \leq 8|z|^2.
\]
The inequality in the middle comes from a direct majoration
\[
\forall r\in[0,r_2],\qquad
\frac{\sum_{n \geq 2} n(n-1) r^{n}}{\Big(r - \sum_{n \geq 2} n r^{n}\Big)^2} = \frac{2 (1 -r)}{(2 r^2 -4r+1)^2} \leq
\frac{2 (1 -r_2)}{(2 r_2^2 -4r_2+1)^2} = 8.
\]
This choice leads to the numerical values $r_2\simeq 0.144910$ and $\rho_2\simeq 0.110549$.
In particular, as $\rho_2<\rho_1$, the disk $\disk(0,\rho_2)$ is stable under $N_f$ and the iterates converge uniformly to zero.
Note that $8|z|^2<0.9 |z|$ because $8|z|\leq 8\rho_2 < 0.9$. We thus have a bi-exponential rate:
\[
|N^{k}_f(z)|\leq \frac{1}{8}(8|z|)^{2^n}
\]
and it takes at most $\log_2\left| \frac{\log_2 8\epsilon }{\log_2 8|z|}\right|\leq \log_2 \frac{|\log_2\epsilon|}{-\log_28\rho_2}$ iterates to ensure that $|N_f^k(z)|<\epsilon$.

Finally, taking $\beta_2\simeq 0.091287$ such that $\beta_2(1-\beta_2)^{-2}=\rho_2$, the estimate~\eqref{eq:kobe}
ensures that $g(\disk(0,\beta_2))\subset \disk(0,\rho_2)$.
\endproof

\subsection{Riemann map of the exterior of $\M$}\label{par:RiemanExtM}

\begin{thm}[\cite{JUNG1985}]\label{thm:JUNG1985}
There exists a sequence of analytic maps $\Phi_k:\M^c \to\CC$ such that
\begin{equation}
\forall  k\geq 3, \qquad
\Phi_k(c)^{d_k} = p_k(c) \qquad\text{and}\qquad \Phi_k(c)\underset{|c|\to\infty}{\sim} c
\end{equation}
with $d_k=2^{k-1}$. The map $\Phi_k$ is one-to-one from $U_k=\{z\in\CC\,;\, |p_{k+1}(c)| >2\}$
onto the outer disk $\{z\in\CC \,;\, |z|>2^{2^{-k}}\}$.
The sequence converges uniformly on compact sets to the Riemann
map~$\M^c\,\cup\{\infty\}\to\bar{\C}\backslash\overline{\\disk(0,1)}^c$.
\end{thm}

\begin{remark}\label{rmkRiemannMap}
The Riemann map is given by $\Phi(c) = \varphi_c(c)$ where $\varphi_c$ is the map
\[
\varphi_c(z) = z \prod_{n=0}^\infty \left( 1+\frac{c}{f_c^n(z)^2} \right)^{2^{-n-1}}
= \lim_{n\to\infty} f_c^n(z)^{2^{-n}}
\]
that conjugates the dynamics on $f_c^{-1}(\CC\,\backslash\disk(0,|c|))$ to that of $z^2$, \ie
\[
f_c(z) = \varphi_c^{-1}(\varphi_c(z)^2).
\]
It satisfies $\varphi_c(z)  = z + o(1)$ at infinity.
The function $G_c(z) = \log|\varphi_c(z)|$ is the Green function of the unbounded component of $\fatou_c$, with pole at $\infty$.
\end{remark}

\newpage
\section{Index of notations}\label{par:notations}
We provide here a short index of our notations. By default, we use the American standard names, notations and spellings.

{\small
\paragraph{About integer, real and complex numbers}
\begin{description}[itemsep=-0.3ex, labelsep=0pt, align=left]
\item $k\vert n$ : the integer $k$ is a divisor of the integer $n$.
\item $\div(n)$ : set of all divisors of $n\in\N^\ast$.
\item $\div(n)^\ast = \div(n)\backslash\{n\}$ : strict divisors of $n$.
\item $\lfloor x \rfloor$ : floor function; $\lfloor x \rfloor = k$ if $k\in\Z$ and $x\in [k,k+1)$.
\item $\interval(x,r) = (x-r,x+r)$ : open interval along the real line.
\item $z=a+ib\in\C$ : complex numbers (with $a,b\in\R$); the euclidian and $\ell^1$ norms are given by:
\[
|z| = \sqrt{x^2+y^2} \qquad \text{ and }\qquad |z|_1 = |x|+|y|.
\]
\item $\disk(z,r) = \{z'\in\C \,;\, |z'-z|<r\}$ : complex open disk.
\item $\CC = \C \cup\{\infty\}$ : the Riemann sphere.
\item $\U_N = \left\{ e^{2ik \pi/N} \,;\, 0\leq k<N\right\}$ : the roots of unity of order $N\in\N^\ast$.
\item $\overset{\circ}{\Omega}$, $\overline{\Omega}$ : interior and closure of a subset $\Omega\in\CC$.
\end{description}

\paragraph{About the Mandelbrot set}
\begin{description}[itemsep=-0.3ex, labelsep=0pt, align=left]
\item $f_c(z)=z^2+c$ : the fundamental map~\eqref{eq:defFundMap} for the dynamics associated to $\M$.
\item $p_n$ and $q_{\ell,n}$ : hyperbolic~\eqref{def:hyp_poly}  and Misiurewicz-Thurston~\eqref{def:mis_poly} polynomials.
\item $\M = \left\{ c\in\C \,;\, \forall n\in\N, \enspace |p_n(c)|\leq 2\right\}$ : Mandelbrot set (Figure \ref{fig:mandelbrot}).
\item $\hyp(n)$ : set of hyperbolic centers of order $n$, defined by \eqref{def:hyp}.
\item $\mis(\ell,n)$  : set of pre-periodic (Misiurewicz-Thurston) parameters of type $(\ell,n)$, defined by \eqref{def:mis}.
\item $h_{n}$, $m_{\ell,n}$ and $s_{\ell,n}$ : reduced polynomials~\eqref{eq:hyp_red}, \eqref{eq:mis_red} and \eqref{eq:polySimple}.
\end{description}

\paragraph{About polynomials and splitting algorithms}
\begin{description}[itemsep=-0.3ex, labelsep=0pt, align=left]
\item $\crit(P)$ : set of critical points of $P$, \ie $P'(z)=0$.
\item $N_P$ : Newton map associated with a polynomial $P$ (see Section~\ref{par:Newton}).
\item $\zeta(t)$ : Newton's flow defined by the ODE~\eqref{eq:NewtonFlow}.
\item $\lambda(t)$ : level line defined by the ODE~\eqref{eq:LevelLinesFlow}.
\item $\lev_{\lambda_0,N}(P)$ : discrete level line of $P$ defined by \eqref{eq:defDiscreteLev}
and central to our splitting algorithm.
\item $\varepsilon_R$, $\varepsilon_N$, $\varepsilon_S$ : various radii related to the certification process
(see Theorem~\ref{thm:certif}).
\item $\mesh_r(d)$ : universal mesh of \cite{HSS2001} for splitting a polynomial of degree $d$ (see Section~\ref{par:HSS}).
\end{description}

\paragraph{About finite precision arithmetic}
\begin{description}[itemsep=-0.3ex, labelsep=0pt, align=left]
\item $\fprecRef_{N}$ : fundamental set \eqref{eq:fprecRefZ} of floating points numbers with finite precision.
\item $\hat\fprec_N$ : (theoretical) set of all \eqref{eq:fprecDefRhat} floating points numbers with a given precision.
\item $\fprec_N$ :  (realistic) subset \eqref{eq:fprecDefR} of $\hat\fprec_N$ with limited exponents.
\item $e(x)$ and $\widehat\ulp(x)=2^{e(x) - N}$ : exponent of $x\in\hat\fprec_N\backslash\{0\}$ and associated
\textit{unit on last position} \eqref{eq:defULP}.
\item $\hat\round_N:\R\to\hat\fprec_N$ : rounding operator from the real line onto finite precision numbers.
\end{description}
}


\section{Listing of tasks implemented in~\citelib{MLib}}\label{MTasks}

In our implementation~\citelib{MLib}, the tasks listed in this section are called in the command line
with \texttt{Mandelbrot -task [arguments]}. Use \texttt{-task -help} for more detailed informations.
The core functions of the interface are listed here.

{\small\noindent%
\begin{longtable}{p{.20\textwidth} p{.80\textwidth}} 
\multicolumn{2}{l}{\bf Tools for the inital setup\footnote{If possible, those three tasks will be merged into one common interface in future versions.}}\\   
\tt -levelSets & prepares the level sets for roots search\\
\tt -misSets  & prepares the level sets for Misiurewicz points search\\
\tt -misSimpleSets & prepares the simple level sets for Misiurewicz points search\\
\\
\multicolumn{2}{l}{\bf Tools for splitting $p_n$ and $q_{\ell,n}$ using \texttt{FP80} hardware arithmetic}\\
\tt -hypQuick & computes hyperbolic centers quickly, with low precision\\
\tt -misQuick & computes pre-periodic points quickly, with low precision\\
\tt -misSimpleQuick & computes pre-periodic points quickly, with low precision, using simplified polynomials $s_{\ell,n}=p_{\ell+n-1}+p_{\ell-1}$\\
\\
\multicolumn{2}{l}{\bf Tools for splitting $p_n$ using certifiable \texttt{FP128} arithmetic}\\
\tt -hypRaw & computes the hyperbolic centers, saves results in $\sim$1GB binary files\\
\tt -hypRawCount & counts the results of \texttt{hypRaw}\\
\tt -hypCount & counts and checks the unicity of hyperbolic centers\\
\tt -hypProve & re-proves the hyperbolic centers and the convergence of the Newton map\\
\\
\multicolumn{2}{l}{\bf Tools for splitting $q_{\ell,n}$ using certifiable \texttt{FP128} arithmetic}\\
\tt -misRaw & computes the pre-periodic points, saves results in $\sim$1GB binary files\\
\tt -misSimpleRaw & computes pre-periodic points using simplified polynomials $s_{\ell,n}$\\
\tt -misRawCount  &  counts the results of \texttt{misRaw}\\
\tt -misCount  & counts and checks the unicity of Misiurewicz parameters\\
\tt -misProve & re-proves the Misiurewicz points and the convergence of the Newton map\\
\\
\multicolumn{2}{l}{\bf Statistical and graphical tools for analyzing the database}\\
\tt -hypMinDist & computes minimum distance between hyperbolic centers\\
\tt -misMinDist & computes minimum distance between pre-periodic points\\
\tt -hypTree  & computes the maps of the hyperbolic centers with different resolutions\\
\tt -misTree  & computes the maps of the Misiurewicz points with different resolutions\\
\tt -bitmap & renders the bitmaps described in the input file\\
\\
\multicolumn{2}{l}{\bf Tools for handling \texttt{csv} files and our custom formats \texttt{nset} and \texttt{mpv}}\\
\tt -header & explains the header of a \texttt{nset} file and check for basic file integrity\\
\tt -nset2csv & exports data from binary \texttt{nset} files to \texttt{csv} files\\
\tt -csvCompare & compare \texttt{csv} files with numerical values\\
\tt -csv2mpv  &   packs numbers from \texttt{csv} files to binary \texttt{mpv} files\\
\tt -mpv2mpv & exports and converts data from binary \texttt{mpv} files to \texttt{mpv} files\\
\tt -mpvCompare  &  compare the values stored in two binary \texttt{mpv} files
\end{longtable}
} 

\newpage
\nocite{SCRSSS2020}
\nocite{Pan02}
\nocite{KimShu1991}

\bibliographystyle{alpha}\small

\begin{thebibliography}{MPRW22}

\bibitem[Abe73]{A73}
O.~Aberth.
\newblock Iteration methods for finding all zeros of a polynomial
  simultaneously,.
\newblock {\em Math. Comput.}, 27(122):339--344, 1973.

\bibitem[ACE09]{ACE09}
J.T. Albrecht, C.P. Chan, and A.~Edelman.
\newblock Sturm sequences and random eigenvalue distributions.
\newblock {\em Foundations of Computational Mathematics}, 9:461--483, 2009.

\bibitem[AMV22]{AMV21}
R.~Anton, N.~Mihalache, and F.~Vigneron.
\newblock Fast evaluation of complex polynomials.
\newblock \href{https://arxiv.org/abs/2211.06320}{{arXiv}:2211.06320}, 2022.

\bibitem[AMV23]{AMV2023a}
R.~Anton, N.~Mihalache, and F.~Vigneron.
\newblock A short {ODE} proof of the fundamental theorem of algebra.
\newblock {\em The Mathematical Intelligencer}, 2023.

\bibitem[BAS16]{BAS2016}
T.~Bilarev, M.~Aspenberg, and D.~Schleicher.
\newblock On the speed of convergence of {N}ewton's method for complex
  polynomials.
\newblock {\em Math. Comput.}, 85(298):693--705, 2016.

\bibitem[BDG04]{BDG04}
D.~A. Bini, F.~Daddi, and L.~Gemignani.
\newblock On the shifted {QR} iteration applied to companion matrices.
\newblock {\em Electronic Transactions on Num. Analysis}, 18:137--152, 2004.

\bibitem[BE95]{BORERD}
P.~Borwein and T.~Erd\'{e}lyi.
\newblock {\em Polynomials and Polynomial Inequalities}.
\newblock Springer, 1995.

\bibitem[Ber93]{Ber93}
W.~Bergweiler.
\newblock Iteration of meromorphic functions.
\newblock {\em Bull. AMS}, 29(2):151--188, 1993.

\bibitem[BH03]{BufHen03}
X.~Buff and C.~Henriksen.
\newblock On {K}\"onig's root finding algorithms.
\newblock {\em Nonlinearity}, 16(3):989, 2003.

\bibitem[Bin96]{B96}
D.~A. Bini.
\newblock Numerical computation of polynomial zeros by means of aberth's
  method.
\newblock {\em Numerical Algorithms}, 13:179--200, 1996.

\bibitem[BF00]{BG2000}
D.~A. Bini and G.~Fiorentino.
\newblock Design, analysis and implementation of a
  multiprecision polynomial rootfinder.
 \newblock {\em Numer. Algorithms}, 23:127--173, 2000.
  
\bibitem[BR14]{BR2014}
D.~A. Bini and L.~Robol.
\newblock Solving secular and polynomial equations: a
  multiprecision algorithm.
\newblock {\em J. Comput. Appl. Math.}, 272:276--292, 2014.


\bibitem[BM92]{BERMAD}
C.~Bernardi and Y.~Maday.
\newblock {\em Approximations spectrales de probl{\`e}me aux limite
  elliptiques}.
\newblock Springer, 1992.

\bibitem[BMF12]{BMF2012}
I.~Bogaert, B.~Michiels, and J.~Fostier.
\newblock O(1) computation of legendre polynomials and gauss--legendre nodes
  and weights for parallel computing.
\newblock {\em SIAM J. Sci. Comput.}, 34:C83--C101, 2012.

\bibitem[Bra85]{Bra85}
L.~de Branges.
\newblock A proof of the Bieberbach conjecture.
\newblock {\em Acta Mathematica}, 154 (1):137--152, 1985.

\bibitem[Bre17]{Bre17}
J.~Bremer.
\newblock On the numerical calculation of the roots of special functions
  satisfying second order ordinary differential equations.
\newblock {\em SIAM Journal on Scientific Computing}, 39(1):A55--A82, 2017.

\bibitem[BS05]{BelSmi04}
D.~Beliaev and S.~Smirnov.
\newblock Harmonic measure on fractal sets.
\newblock In European~Mathematical Society, editor, {\em 4ECM Stockholm 2004},
  pages 41--59, 2005.

\bibitem[Buf18]{BUFF2018}
X.~Buff.
\newblock On postcritically finite unicritical polynomials.
\newblock {\em New York Journal of Mathematics}, 24:1111--1122, 2018.

\bibitem[Cam19]{Cam19}
T.R. Cameron.
\newblock An effective implementation of a modified laguerre method for the
  roots of a polynomial.
\newblock {\em Numerical Algorithms}, 82:1065--1084, 2019.

\bibitem[CG93]{CARLESON}
L.~Carleson and T.W. Gamelin.
\newblock {\em Complex dynamics}.
\newblock Springer, 1993.

\bibitem[CGXZ07]{CGXZ07}
S.~Chandrasekaran, M.~Gu, J.~Xia, and J.~Zhu.
\newblock A fast {QR} algorithm for companion matrices.
\newblock {\em Recent Advances in Matrix and Operator Theory}, 179:111--143,
  2007.

\bibitem[Che10]{CHE10}
A.~Cheritat.
\newblock L'ensemble de mandelbrot.
\newblock {\em Images des Math{\'e}matiques},
  \url{https://images.math.cnrs.fr/L-ensemble-de-Mandelbrot.html}, 2010.

\bibitem[DH82]{DOUHUB82}
A.~Douady and J.H. Hubbard.
\newblock It\'{e}ration des polyn\^{o}mes quadratiques complexes.
\newblock {\em C.R. Acad. Sci. Paris, S\'{e}r. I Math.}, 294(3):123--126, 1982.

\bibitem[DH85]{DOUHUB84}
A.~Douady and J.H. Hubbard.
\newblock {\em Etude dynamique des polyn\^{o}mes complexes}.
\newblock Pr\'{e}-publications math\'{e}matiques d'Orsay, 1984-1985.

\bibitem[DT93]{DT93}
E.R. Davidson and W.J. Thompson.
\newblock Monster matrices: their eigenvalues and eigenvectors.
\newblock {\em Computers in Physics}, 7(5):519--522, 1993.

\bibitem[Fat19]{FATOU1919}
P.~Fatou.
\newblock Sur les \'{e}quations fonctionnelles.
\newblock {\em Bull. SMF}, 47:161--271, 1919.

\bibitem[FG15]{FG2013}
C.~Favre and T.~Gauthier.
\newblock Distribution of postrictically finite polynomials.
\newblock {\em Israel Journal of Mathematics}, 2015.

\bibitem[Fra89]{F89}
P.~Fraigniaud.
\newblock Analytic and asynchronous root finding methods on a distributed
  memory multi- computer.
\newblock {\em Research Report LIP-IMAG}, 1989.

\bibitem[FS89]{FreSak89}
M.L. Fredmann and M.E. Saks.
\newblock The cell probe complexity of dynamic data structures.
\newblock In {\em Proceedings of the Twenty-First Annual ACM Symposium on
  Theory of Computing}, STOC' 89, pages 345--354. Association for Computing
  Machinery, 1989.

\bibitem[GMSB16]{GMSB2016}
A.~Gholami, D.~Malhotra, H.~Sundar, and G.~Biros.
\newblock Fft, fmm, or multigrid? a comparative study of state-of-the-art
  poisson solvers for uniform and nonuniform grids in the unit cube.
\newblock {\em SIAM J. Sci. Comput.}, 38(3), 2016.

\bibitem[GSCG17]{GSCG19}
K.~Ghidouche, A.~Sider, R.~Couturier, and C.~Guyeux.
\newblock Efficient high degree polynomial root finding using gpu.
\newblock {\em Journal of Computational Science}, 18:46--56, 2017.

\bibitem[GSKC16]{GSKCx}
K.~Ghidouche, A.~Sider, L.Z. Khodja, and R.~Couturier.
\newblock Two parallel implementations of {E}hrlich-{A}berth algorithm for
  root-finding of polynomials on multiple {GPU}s with {OpenMP} and {MPI}.
\newblock In {\em Intl Conference on Computational Science and Engineering},
  2016.

\bibitem[Guc]{EXPL1}
K.~Guciek.
\newblock \url{https://guciek.github.io/web_mandelbrot.html}.

\bibitem[Gug86]{G86}
H.~Guggenheimer.
\newblock Initial approximations in durand-kerner's root finding method.
\newblock {\em BIT Numerical Mathematics volume}, 26:537--539, 1986.

\bibitem[GV17]{GV2016b}
T.~Gauthier and G.~Vigny.
\newblock Distribution of postrictically finite polynomials {II}: speed of
  convergence.
\newblock {\em American Institute of Mathematical Science}, 11:57--98, 2017.

\bibitem[GV19]{GV2016a}
T.~Gauthier and G.~Vigny.
\newblock Distribution of postrictically finite polynomials {III}:
  combinatorial continuity.
\newblock {\em Fundamenta Math}, 244(1):17--48, 2019.

\bibitem[Hen74]{HEN1974}
P.~Henrici.
\newblock {\em Applied and computational complex analysis, Volume 1: Power
  series, integration, conformal mapping, location of zeros.}
\newblock Wiley, 1974.

\bibitem[HSS01]{HSS2001}
J.H. Hubbard, D.~Schleicher, and S.~Sutherland.
\newblock How to find all roots of complex polynomials by {N}ewton's method.
\newblock {\em Invent. math.}, 146:1--33, 2001.

\bibitem[HT13]{HT2013}
N.~Hale and A.~Townsend.
\newblock Fast and accurate computation of gauss-legendre and gauss-jacobi
  quadrature nodes and weights.
\newblock {\em SIAM J. Sci. Comput.}, 35:A652--A674, 2013.

\bibitem[HT15]{HT15}
B.~ Hutz and A.~Towsley.
\newblock Misiurewicz points for polynomial maps and transversality.
\newblock {\em New York Journal of Mathematics}, 21:297--319, 2015.

\bibitem[{IEE}]{IEEE}
{IEEE~754}.
\newblock \url{https://en.wikipedia.org/wiki/IEEE_754}.

\bibitem[IYM11]{IYM2011}
T.~Imamura, S.~Yamada, and M.~Machida.
\newblock Development of a high-performance eigensolver on a peta-scale
  next-generation supercomputer system.
\newblock {\em Progress in Nuclear Science and Technology}, 2:643--650, 2011.

\bibitem[Jun85]{JUNG1985}
I.~Jungreis.
\newblock The uniformisation of the complement of the {M}andelbrot set.
\newblock {\em Duke Math. J.}, 52(4):935--938, 1985.

\bibitem[KI04]{KI04}
N.~Kyurkchiev and A.~Iliev.
\newblock Failure of convergence of the newton-weierstrass iterative method for
  simultaneous rootfinding of generalized polynomials.
\newblock {\em Computer and Mathematics with Applications}, 47:441--446, 2004.

\bibitem[KS94]{KimShu1991}
M.-H Kim and S.~Sutherland.
\newblock Polynomial root-finding algorithms and branched covers.
\newblock {\em SIAM Journal on Computing}, 23(2):415--436, 1994.

\bibitem[KS16]{KS16}
A.~Kobel and M.~Sagraloff.
\newblock Fast approximate polynomial multipoint evaluation and applications.
\newblock \href{https://arxiv.org/abs/1304.8069}{{arXiv}:1304.8069}, 2016.

\bibitem[Lar13]{LAR13}
K.G. Larsen.
\newblock {\em Models and Techniques for Proving Data Structure Lower Bounds
  Models and Techniques for Proving Data Structure Lower Bounds Models and
  Techniques for Proving Data Structure Lower Bounds}.
\newblock PhD thesis, Aarhus University, Denmark, 2013.

\bibitem[Lev90]{LEVIN1990}
G.M. Levin.
\newblock On the theory of iterations of polynomial families in the complex
  plane.
\newblock {\em J. Soviet Math.}, 52(6):3512--3522, 1990.

\bibitem[Mak85]{Mak84}
N.G. Makarov.
\newblock On the distortion of boundary sets under conformal mappings.
\newblock {\em Proceedings of the London Mathematical Society}, 51(2):369--384,
  1985.

\bibitem[MV$_{n+1}$]{MVCoeffs}
N. Mihalache and F. Vigneron.
\newblock How to compute the coefficients of holomorphic maps.
\newblock {\em In preparation}.

\bibitem[Mil90]{MILNOR}
J.~Milnor.
\newblock {\em Dynamics in one complex variable}.
\newblock Number 160 in Annals of Mathematics Studies. Princeton Univ. Press,
  1990.

\bibitem[M{\"o}b32]{Mob1832}
A.F. M{\"o}bius.
\newblock \"{U}ber eine besondere art von umkehrung der reihen.
\newblock {\em Journal f\"{u}r die reine und angewandte {M}athematik},
  9:105--123, 1832.

\bibitem[Mor13]{Mor13}
G.~Moroz.
\newblock Fast polynomial evaluation and composition.
\newblock Technical Report 453, Inria Nancy - Grand Est, LORIA - ALGO -
  Department of Algorithms, Computation, Image and Geometry, 2013.

\bibitem[MPRW22]{MRW22}
D.~Mart{\'\i}-Pete, L.~Rempe, and J.~Waterman.
\newblock Bounded {F}atou and {J}ulia components of meromorphic functions.
\newblock \href{https://arxiv.org/abs/2204.11781}{{arXiv}:2204.11781}, 2022.

\bibitem[\OF91]{CARDHk}
\OFISlabel{Online Encyclopedia of Integer Sequence}.
\newblock Sequence {A}000740.
\newblock \url{https://oeis.org/A000740}, 1991.

\bibitem[Pan02]{Pan02}
V.Y. Pan.
\newblock Univariate polynomials: Nearly optimal algorithms for numerical
  factorization and root-finding.
\newblock {\em J. Symbolic Computation}, 33:701--733, 2002.

\bibitem[PP98]{COMPLEXDISKARITH}
M.S. Petkovi\'{c} and L.D. Petkovi\'{c}.
\newblock {\em Complex interval arithmetic and its applications}, volume 105.
\newblock Wiley-{VCH}, 1998.

\bibitem[RAY19]{RAY2019}
K.~Ravikumar, D.~Appelhans, and P.K. Yeung.
\newblock Gpu acceleration of extreme scale pseudo-spectral simulations of
  turbulence using asynchronism.
\newblock Technical report, The International Conference for High Performance
  Computing, Networking, Storage and Analysis, DOI: 10.1145/3295500.3356209,
  2019.

\bibitem[Rok01]{ROKNE2001}
J.G. Rokne.
\newblock {\em Interval Arithmetic and Interval Analysis: An Introduction}.
\newblock in book Granular Computing: An Emerging Paradigm, 2001.

\bibitem[Rot63]{Rota1963}
G.-C. Rota.
\newblock On the foundations of combinatorial theory, {I}: {T}heory of
  {M}\"{o}bius functions.
\newblock {\em Z. Wahrscheinlichkeitstheorie u. verw. Gebiete}, 2:340--368,
  1963.

\bibitem[RR05]{MPFI}
N.~Revol and F.~Rouillier.
\newblock Motivations for an arbitrary precision interval arithmetic and the
  {MPFI} library.
\newblock {\em Reliable Computing}, 11:275--290, 2005.

\bibitem[RSS17]{RSS2017}
M.~Randig, D.~Schleicher, and R.~Stoll.
\newblock Newton's method in practice {II}: The iterated refinement {N}ewton
  method and near-optimal complexity for finding all roots of some polynomials
  of very large degrees.
\newblock \href{https://arxiv.org/abs/1703.05847}{{arXiv}:1703.05847}, 2017.

\bibitem[RSS20]{RST2020}
B.~Reinke, D.~Schleicher, and M.~Stoll.
\newblock The weierstrass root finder is not generally convergent.
\newblock \href{https://arxiv.org/abs/2004.04777}{{ArXiv}:2004.04777}, 2020.

\bibitem[Sch23]{Sch2023}
D.~Schleicher.
\newblock On the efficient global dynamics of newton's method for complex
  polynomials.
\newblock {\em Nonlinearity}, 36:1349--1377, 2023.

\bibitem[SCR{\etalchar{+}}20]{SCRSSS2020}
S.~Shemyakov, R.~Chernov, D.~Rumiantsau, D.~Schleicher, S.~Schmitt, and
  A.~Shemyakov.
\newblock Finding polynomial roots by dynamical systems -- a case study.
\newblock {\em Discrete and Continuous Dyn. Systems}, 40(12):6945--6965, 2020.

\bibitem[Sib84]{Sib}
N.~Sibony.
\newblock Expos{\'e}s {\`a} {O}rsay non publi{\'e}s \& {C}ours {UCLA}.
\newblock 1981-1984.

\bibitem[SK19]{SK2019}
A.~Schwarz and L.~Karlsson.
\newblock Scalable eigenvector computation for the non-symmetric eigenvalue
  problem.
\newblock {\em Parallel Computing}, 85:131--140, 2019.

\bibitem[SS17]{SS2017}
D.~Schleicher and R.~Stoll.
\newblock {N}ewton's method in practice: finding all roots of polynomials of
  degree one million efficiently.
\newblock {\em Theor. Comput. Sci.}, 681:146--166, 2017.

\bibitem[Sut89]{SU1989}
S.~Sutherland.
\newblock {\em Finding root of complex polynomials with {N}ewton's method.}
\newblock PhD thesis, Boston {U}niversity, 1989.

\bibitem[SZI{\etalchar{+}}17]{SZIYKH2015}
T.~Sakurai, S.-L. Zhang, T.~Imamura, Y.~Yamamoto, Y.~Kuramashi, and T.~Hoshi,
  editors.
\newblock {\em Eigenvalue Problems: Algorithms, Software and Applications in
  Petascale Computing}, volume 117 of {\em Lecture Notes in Comp. Science and
  Engineering}. Springer, 2017.

\bibitem[TAB{\etalchar{+}}19]{Hecht2019}
P.-H. Tournier, I.~Aliferis, M.~Bonazzoli, M.~de~Buhan, M.~Darbas, V.~Dolean,
  F.~Hecht, P.~Jolivet, I.~El Kanfoud, C.~Migliaccio, F.~Nataf, Ch. Pichot, and
  S.~Semenov.
\newblock Microwave tomographic imaging of cerebrovascular accidents by using
  high-performance computing.
\newblock {\em Parallel Computing}, pages 88--97, 2019.

\bibitem[TTO16]{TTO2016}
A.~Townsend, T.~Trogdon, and S.~Olver.
\newblock Fast computation of gauss quadrature nodes and weights on the whole
  real line.
\newblock {\em IMA J. Numer. Anal.}, 36:337--358, 2016.

\bibitem[vdH09]{MATHEMAGIX}
Joris van~der Hoeven.
\newblock Ball arithmetic.
\newblock {\em Technical report: HAL-00432152.}, 2009.

\bibitem[Wil84]{W84}
J.~H. Wilkinson.
\newblock {\em The perfidious polynomial}, pages 1--28.
\newblock Studies in Numerical Analysis. G. H. Golub, 1984.

\bibitem[\ZA12]{ARB}
\ZARBlabel{F. Johansson}.
\newblock Arb library.
\newblock \url{https://github.com/fredrik-johansson/arb}, 2012.

\bibitem[\ZF13]{FLINT}
\ZFLINTlabel{W. Hart and F. Johansson and S. Pancratz}.
\newblock {FLINT}: {F}ast {L}ibrary for {N}umber {T}heory.
\newblock \url{http://flintlib.org}, 2013.

\bibitem[\ZV22]{FPELib}
\ZVFPElabel{N. Mihalache and F. Vigneron}.
\newblock {FPE} library: a {F}ast {P}olynomial {E}valuator.
\newblock \\\url{https://github.com/fvigneron/FastPolyEval}, 2022.

\bibitem[\ZW24]{MLib}
\ZWMANDELlabel{N. Mihalache and F. Vigneron}.
\newblock {M}andel library: a {N}umerical {M}icroscope onto the {M}andelbrot
  set.
\newblock \url{https://github.com/fvigneron/Mandelbrot}, 2024.

\bibitem[\ZX24]{MLibData}
\ZXMANDDBlabel{N. Mihalache and F. Vigneron}.
\newblock Complete list of hyperbolic centers of period $\leq 41$ and of all
  {M}isiurewicz-{T}hurston points whose pre-period and period sum is $\leq 35$.
\newblock 

\bibitem[\ZZ07]{MPFR}
\ZZPFRlabel{L. Fousse, G. Hanrot, V. Lef{\`e}vre, P. P{\'e}lissier and P.
  Zimmermann}.
\newblock M{PFR}: a {M}ultiple-{P}recision binary {F}loating-point library with
  correct {R}ounding.
\newblock {\em ACM Trans. Math. Software}, 33(2):13--28.
  \url{https://www.mpfr.org}, 2007.
  
\end{thebibliography}

\newcommand{\etalchar}[1]{$^{#1}$}
\newcommand\MUlabel{}\newcommand\MU[2]{M83}\newcommand\OFISlabel{}\newcommand\OF[2]{OEIS}\newcommand\ZZPFRlabel{}\newcommand\ZZ[2]{\texttt{MPFR}}\newcommand\ZARBlabel{}\newcommand\ZA[2]{\texttt{ARB}}\newcommand\ZFLINTlabel{}\newcommand\ZF[2]{\texttt{FLINT}}\newcommand\ZVFPElabel{}\newcommand\ZV[2]{\texttt{FPE}}\newcommand\ZWMANDELlabel{}\newcommand\ZW[2]{\texttt{Mandel}}\newcommand\ZXMANDDBlabel{}\newcommand\ZX[2]{\texttt{Mand.DB}}

\vspace*{3em}\noindent
$^{\small 1}$ \authornicu\\[2ex]
$^{\small 2}$ \authorfv

\end{document}